\documentclass[a4paper, bibliography=totoc]{scrartcl}

\newif\ifshort
\newif\ifarxiv

\arxivtrue 

\ifshort\else\ifarxiv\else\arxivtrue\fi\fi 

\def\short{\ifshort\ifarxiv\color{journalcolor}\fi}
\def\arxiv{\ifarxiv\ifshort\color{arxivcolor}\fi}

\usepackage{amsmath}
\usepackage{amssymb}
\usepackage{enumerate}
\usepackage{amsthm}
\usepackage{todonotes}
\usepackage{mathtools} 
\usepackage[UKenglish]{babel}
\usepackage[T1]{fontenc}
\usepackage{slantsc} 
\usepackage[utf8]{inputenc}
\usepackage{tikz}
\usepackage{thmtools} 
\usepackage{thm-restate} 
\usepackage{stackrel} 
\usepackage[mathlines]{lineno}

\usepackage[normalem]{ulem}

\ifarxiv\else 
\linenumbers
\fi

\usepackage{color}
\definecolor{grey}{rgb}{.7,.7,.7}
\definecolor{blue}{rgb}{0,0,.8}
\definecolor{red}{rgb}{.8,0,0}
\definecolor{green}{rgb}{0,.4,0}
\definecolor{gold}{rgb}{0.8,0.6,0.1}
\definecolor{brown}{rgb}{0.8,0.4,0.1}
\definecolor{arxivcolor}{rgb}{0.5,0.5,0}
\definecolor{journalcolor}{rgb}{0.5,0,1}
\definecolor{purple}{rgb}{0.6,0.2,0.6}
\definecolor{pastelgreen}{rgb}{0,0.65,0.1}

\long\def\hl#1{\color{black} #1\normalcolor}
\long\def\jl#1{\color{black} #1\normalcolor}

\long\def\sug#1{\color{black} #1\normalcolor}
\ifshort
\long\def\cancel#1{\color{red}\ifmmode\text{\sout{\ensuremath{#1}}}\else\sout{#1}\fi\normalcolor}
\else
\def\hideme{\ifdim\lastskip>0pt \ignorespaces\fi}
\long\def\cancel#1{\hideme}
\fi

\usepackage[unicode]{hyperref} 
\definecolor{blue3}{rgb}{.1,.0,.4}
\hypersetup{colorlinks=true, linkcolor=red, urlcolor=blue3, citecolor=blue3, pdfpagemode=UseNone, pdfstartview=FitH, bookmarksopen=true} 
\hypersetup{pdftitle=Highly irregular separated nets.}
\usepackage[open]{bookmark} 

\declaretheorem[name=Theorem,numberwithin=section]{thm} 
\newtheorem*{thm*}{Theorem}

\newtheorem*{define*}{Definition}
\newtheorem{define}[thm]{Definition}

\newtheorem*{lemma*}{Lemma}
\newtheorem{lemma}[define]{Lemma}

\newtheorem*{algorithm*}{Algorithm}
\newtheorem{algorithm}[define]{Algorithm}

\newtheorem*{construction*}{Construction}

\newtheorem*{prop*}{Proposition}
\newtheorem{prop}[define]{Proposition}

\newtheorem*{obs*}{Observation}

\newtheorem*{fact*}{Fact}

\newtheorem*{remark*}{Remark}
\newtheorem{remark}[define]{Remark}

\newtheorem*{quest*}{Question}
\newtheorem{quest}[define]{Question}

\newtheorem*{cor*}{Corollary}
\newtheorem{cor}[define]{Corollary}

\newtheorem*{conjecture*}{Conjecture}

\newtheorem*{question*}{Question}

\newtheorem*{example*}{Example}

\newcounter{claimcounter}[define]
\numberwithin{claimcounter}{define}
\newtheorem*{claim*}{Claim}
\newtheorem{claim}[claimcounter]{Claim}
\numberwithin{equation}{section}

\newcommand{\R}{\mathbb{R}}

\newcommand{\Z}{\mathbb{Z}}
\newcommand{\lip}{\operatorname{Lip}}

\newcommand{\dist}{\operatorname{dist}}
\newcommand{\N}{\mathbb{N}}
\newcommand{\M}{\mathcal{M}}

\newcommand{\leb}{\mathcal{L}}

\DeclareMathOperator{\diam}{diam}

\DeclareMathOperator{\intr}{int}

\DeclareMathOperator{\dom}{domain}

\newcommand{\mb}[1]{\mathbf{#1}}

\newcommand{\mc}[1]{\mathcal{#1}}

\newcommand{\niceint}[3]{\int_{#1}{#2\,\mathrm{d}#3}}
\newcommand{\lint}[2]{\niceint{#1}{#2}{\leb}}

\newcommand{\poly}[1]{\operatorname{poly}\left(#1\right)}
\newcommand{\ply}[2]{\operatorname{poly}_{#1}\left(#2\right)}
\newcommand{\pl}[3]{\operatorname{poly}^{#1}_{#2}\left(#3\right)}
\newcommand{\abs}[1]{\left|#1\right|}
\newcommand{\lnorm}[2]{\left\|#2\right\|_#1}
\newcommand{\enorm}[1]{\lnorm{2}{#1}}
\newcommand{\norm}[1]{\left\|#1\right\|}

\newcommand{\Sq}{\mathcal{S}}

\newcommand{\bilip}{\operatorname{biLip}}
\newcommand{\lipomeg}{\mathfrak{L}_{\omega}}
\newcommand{\lipgen}[1]{\mathfrak{L}_{#1}}
\newcommand{\bilipomeg}{\mathfrak{biL}_{\omega}}
\newcommand{\bilipgen}[1]{\mathfrak{biL}_{#1}}

\newcommand{\set}[1]{\left\{#1\right\}}
\newcommand{\cl}[1]{\overline{#1}}
\newcommand{\rest}[2]{#1|_{#2}}

\DeclarePairedDelimiter{\floor}{\lfloor}{\rfloor}
\DeclarePairedDelimiter{\br}{(}{)}

\def\XXint#1#2#3{{\setbox0=\hbox{$#1{#2#3}{\int}$ }
		\vcenter{\hbox{$#2#3$ }}\kern-.6\wd0}}

\makeatletter
\newcommand{\mylabel}[2]{#2\def\@currentlabel{#2}\label{#1}}
\def\blfootnote{\xdef\@thefnmark{}\@footnotetext} 
\makeatother

\title{Highly irregular separated nets.}
\author{Michael Dymond, Vojt\v ech Kalu\v za}
\date{}
\begin{document}
	\maketitle
\abstract{In 1998 Burago and Kleiner and (independently) McMullen gave examples of separated nets in Euclidean space which are bilipschitz non-equivalent to the integer lattice. We study weaker notions \jl{than bilipschitz equivalence} and demonstrate that such notions also \jl{distinguish between separated nets}. Put differently, we find occurrences of particularly strong divergence of separated nets from the integer lattice. Our approach generalises that of Burago and Kleiner and McMullen which takes place largely in a continuous setting. Existence of irregular separated nets is verified via the existence of non-realisable density functions $\rho\colon [0,1]^{d}\to(0,\infty)$. In the present work we obtain stronger types of non-realisable densities.}			
\blfootnote{\jl{This work was done while both authors were employed at the University of Innsbruck and enjoyed the full support of Austrian Science Fund (FWF): P 30902-N35.}}	
	\section{Introduction}
The question of whether two separated nets of a Euclidean space may carry inherently different metric structures has been considered by many authors. Indeed, Gromov's 1993 question \cite{Grom} of whether any two such nets are necessarily bilipschitz equivalent\footnote{Several sources, e.g.~\cite{BK1}, \cite[p.63]{Katok_handbook}, \cite[p.80]{metric_geom}, state that this question arose first in the work of Furstenberg in the 1960's in the context of ergodic theory and dynamical systems, but we are not aware of any reference for that.} remained open for several years, before being resolved negatively by McMullen~\cite{McM} and Burago and Kleiner~\cite{BK1} in 1998.
Whilst the aforementioned works provide examples of separated nets in $\R^{d}$, $d\geq 2$, which are not bilipschitz equivalent to the integer lattice, the present article focuses on the \emph{extent} to which this divergence can occur.

The first insight on this question comes from McMullen~\cite{McM}, who measures divergence between separated nets by introducing the notion of \emph{homogeneous bi-Hölder \jl{mappings}}. To permit finer description of how much separated nets may differ we generalise McMullen's notion of a homogeneous Hölder \jl{mapping} from \cite{McM} in the following way:
\begin{define}\label{d:hom_omega_map}
Given two separated nets $X,Y\subseteq \R^{d}$ and a modulus of continuity\footnote{The notions of the \emph{separated net} and the \emph{modulus of continuity} are defined in Section~\ref{s:def} \jl{in Definitions~\ref{def:sep_net} and \ref{def:modulus}, respectively}.}                                                                        
 $\omega$, a mapping $f\colon X\to Y$ is called a \emph{homogeneous $\omega$-mapping} if there \jl{are} constant\jl{s} $K>0$ \jl{and $a\in(0,1)$} such that
\begin{linenomath}
\begin{equation*}
\lnorm{2}{f(x_{2})-f(x_{1})}\leq KR\omega\left(\frac{\lnorm{2}{x_{2}-x_{1}}}{R}\right)
\end{equation*}
\end{linenomath}
for all $R>0$ and $x_{1},x_{2}\in X\cap B(\mb{0},R)$ \jl{with $\enorm{x_2-x_1}<aR$}.
The net $X$ \jl{is said to be \emph{$\omega$-regular with respect to $Y$}} if there is a bijection $f\colon X\to Y$ so that both $f$ and $f^{-1}$ are homogeneous $\omega$-mappings. \jl{Otherwise, we call $X$ \emph{$\omega$-irregular with respect to $Y$}. In the case that $Y=\Z^{d}$, we shorten these terms and just say that $X$ is $\omega$-regular or $\omega$-irregular respectively.}
\end{define}
For \jl{Hölder moduli of continuity $\omega(t)=t^{\beta}$ with $\beta\in (0,1)$}, our notion of \jl{$\omega$-regularity} \jl{refers to the existence of a homogeneous bi-Hölder bijection in the sense of McMullen~\cite[Thm.~5.1]{McM}.} In the case \jl{of the Lipschitz modulus of continuity $\omega(t)=t$,} it reduces to bilipschitz equivalence.
Clearly, for moduli of continuity \jl{$\omega$ satisfying $\lim_{t\to 0}\frac{\omega(t)}{t}=\infty$}, $\omega$-\jl{regularity} is weaker than bilipschitz equivalence \jl{to the integer lattice}, at least formally\footnote{In a new preprint~\cite{DK2021divergence} we verify that for the moduli of continuity $\omega$ of Theorem~\ref{thm:main_result_weak}, $\omega$-regularity is indeed strictly weaker than bilipschitz equivalence. Although this appears to be intuitively true, the only proof of which we are aware is rather long.}. Accordingly, determining the moduli of continuity $\omega$ with respect to which \jl{a separated net} $X$ is \jl{$\omega$-regular} provides a finer \jl{comparison} of \jl{its} metric structure \jl{with that of the integer lattice}.

McMullen~\cite{McM} proves that \jl{any separated net in the Euclidean space $\R^{d}$ is $\omega$-regular for some Hölder modulus of continuity $\omega$}. \jl{Since the composition of two homogeneous bi-Hölder bijections (in the sense of \cite[Thm.~5.1]{McM}) is again a bi-Hölder bijection,} this provides an upper bound on the divergence between any two separated nets in Euclidean space\jl{: for any two separated nets $X,Y\subseteq \R^{d}$, $X$ is $\omega$-regular with respect to $Y$ for some Hölder modulus of continuity $\omega$.}
McMullen's result leads naturally to the question of whether this Hölder bound is tight.
Indeed, alongside the existence of bilipschitz non-equivalent separated nets, it invites the question of whether there are any moduli of continuity $\omega$ asymptotically greater than Lipschitz which admit \jl{$\omega$-irregular} separated nets. 
The present article initiates the study of \jl{$\omega$-regularity} for the `missing range' of moduli $\omega$ not treated by McMullen, Burago and Kleiner, that is, those moduli $\omega$ lying in between Lipschitz and Hölder. We will provide the first examples of such $\omega$ admitting \jl{$\omega$-irregular separated nets}.
\begin{thm}\label{thm:main_result_weak}
Let $d\geq 2$. Then there is $\alpha_0=\alpha_0(d)>0$ \jl{and a} separated net \jl{$X$} in $\R^d$ \jl{such that $X$ is $\omega$-irregular} for \jl{any} modulus of continuity
\jl{$\omega$ such that there is $a>0$ for which }$\omega(t)=t\jl{\bigl(}\log\frac{1}{t}\bigr)^{\alpha_0}$ \jl{for all $t\in (0,a)$}.
\end{thm} 
The ultimate aim of continuing research in this programme will be to establish a sharp threshold for existence of \jl{$\omega$-irregular separated nets}: \jl{a} reasonable conjecture, given the state of the art, is that \jl{$\omega$-irregular separated nets exist} in multidimensional Euclidean space if and only if the modulus of continuity $\omega$ is asymptotically smaller than all Hölder moduli, \jl{that is, if and only if $\lim_{t\to 0}\frac{\omega(t)}{t^\beta}=0$ for all $\beta\in (0,1)$}.
We emphasise that the validity of both implications in this conjecture are open questions.

In a wider setting, the present work contributes to a large amount of research interest in separated nets. Motivation for the study of separated nets comes from geometric classification problems such as for Banach spaces or metric groups (see, e.g., \cite[Sec.~10.3]{benyamini1998geometric} and \cite{Grom, GGT}), from the study of dynamical systems (e.g., \cite{Katok_handbook})  and quasi-periodic structures in mathematical physics (e.g., \cite{quasicrystals, aperiodic1, aperiodic2}).
The work \cite{BK1} of Burago and Kleiner has been further studied, refined and extended in \cite{Garber, Magazinov, Navas, DKK2018}. In \cite{BK2} Burago and Kleiner establish a sufficient criterion for two separated nets in $\R^d$ to be bilipschitz equivalent. Moreover, Magazinov~\cite{Magazinov} proves that the number of bilipschitz equivalence classes of separated nets in multidimensional Euclidean space has the cardinality of the continuum.

Let us be more specific about how the notion of $\omega$-\jl{regularity} provides a deeper insight into the metric structures of separated nets. Bilipschitz non-equivalence of two nets $X$ and $Y$ may be naturally ordered according to the optimal asymptotic growth of the \hl{Lipschitz} constants $\lip(f|_{B(\mb{0},R)\jl{\cap{X}}})$ \hl{and $\lip(f^{-1}|_{B(\jl{\mb{0}},R)\jl{\cap{Y}}})$} as $R\to\infty$ among bijections $f\colon X\to Y$.
Intuitively, if two nets $X$ and $Y$ are bilipschitz non-equivalent, but see slow asymptotic growth of $(\lip(f|_{B(\mb{0},R)\jl{\cap{X}}})_{R>0}$ \hl{and $(\lip(f^{-1}|_{B(\mb{0},R)\jl{\cap{Y}}})_{R>0}$} for some bijection $f\colon X\to Y$, it means that high distortion is only seen by comparing very large portions of $X$ and $Y$.
\hl{We} note that for concave \jl{and increasing} moduli $\omega$, which are the only moduli that we consider, a control on the growth of \hl{$(\lip(g|_{B(\mb{0},R)\jl{\cap{\dom(g)}}}))_{R>0}$} \hl{for $g\in\set{f,f^{-1}}$} is a necessary condition for $X$ \jl{to be $\omega$-regular with respect to} $Y$. This follows from the easy observation that\jl{, in this case,} whenever \jl{$X$ is $\omega$-regular with respect to $Y$ via}
$f\colon X\to Y$, we have
\begin{linenomath}
\begin{equation}\label{eq:growth_condition}
\hl{\lip(g|_{B(\mb{0},R)\jl{\cap{\dom(g)}}})}\leq KR\omega\left(\frac{\jl{a}}{R}\right)
\end{equation}
\end{linenomath}
for \hl{$g\in\set{f,f^{-1}}$}, all $R>\jl{1}$ \hl{and some constant\jl{s} $K>0$} \jl{and $a\in(0,1)$}.
Given this connection between the notions of \jl{$\omega$-regularity} and the rate of growth of restricted \hl{Lipschitz} constants, one might ask whether it is possible to formulate Theorem~\ref{thm:main_result_weak} in the language of restricted \hl{Lipschitz} constants.
A natural expectation is that for the separated net $X$ given by Theorem~\ref{thm:main_result_weak} and any bijection $f\colon X\to \Z^{d}$ we have the failure of \eqref{eq:growth_condition} with $\omega(t)=t\left(\log\frac{1}{t}\right)^{\alpha_0}$, that is, that the restricted \hl{Lipschitz} constants \hl{$\lip(g|_{B(\jl{\mb{0}},R)\jl{\cap{\dom(g)}}})$} for $g=f$ or $g=f^{-1}$ grow asymptotically faster than $(\log R)^{\alpha_{0}}$ as $R\to\infty$.
\hl{However, our methods do not appear sufficient to obtain such a statement.
This is because, although the $\omega$-homogeneity of a bijection $f\colon X\to Y$ implies the growth condition \eqref{eq:growth_condition} on its restricted Lipschitz constants, the opposite implication is not formally valid.} Instead, for $g\in\set{f,f^{-1}}$ and $x\neq y\in B(\jl{\mb{0}},R)\cap \dom(g)$, Theorem~\ref{thm:main_result_weak} provides information about the growth of the quotient $\frac{\lnorm{2}{g(y)-g(x)}}{\lnorm{2}{y-x}}$ as $R\to\infty$ in comparison to the growth of the quotient $\frac{R}{\lnorm{2}{y-x}}$.
In this vein, Theorem~\ref{thm:main_result_weak} can be reformulated as follows:
\begin{thm*}[Restatement of Theorem~\ref{thm:main_result_weak}]
Let $d\geq 2$. Then there is $\alpha_0=\alpha_0(d)>0$ and a separated net $X$ in $\R^d$ such that for every bijection $f\colon X\to\Z^d$ \jl{and some\footnote{\jl{Clearly, \eqref{eq:growth_condition} holds true for some $a\in(0,1)$ if and only if it holds for all $a\in(0,1)$.}} $a\in (0,1)$}
the following holds
with $g=f$ or $g=f^{-1}$
\begin{equation}\label{eq:thm_limsup}
	\limsup_{R\to\infty}\sup_{\substack{x\neq y\in B(\jl{\mb{0}},R)\cap \dom(g),\\ \jl{\lnorm{2}{y-x}<aR}}}\frac{\lnorm{2}{g(y)-g(x)}}{\lnorm{2}{y-x}\br*{\log\frac{R}{\lnorm{2}{y-x}}}^{\alpha_0}}=\infty.
\end{equation}
\end{thm*}
If one removes the factor $\br*{\log\frac{R}{\lnorm{2}{y-x}}}^{\alpha_0}$ in the expression \eqref{eq:thm_limsup}, the conclusion of the restatement of Theorem~\ref{thm:main_result_weak} \jl{yields} 
\begin{linenomath}
	\begin{equation*}
	\limsup_{R\to\infty}\max_{g\in\set{f,f^{-1}}}\lip(g|_{B(\jl{\mb{0}},R)\jl{\cap{\dom(g)}}})=\infty
	\end{equation*}
\end{linenomath}
and then it simply asserts the existence of a separated net \jl{$X$} which is bilipschitz non-equivalent to the integer lattice $\Z^d$. This is precisely the result of Burago and Kleiner~\cite[Thm.~1.1]{BK1} and McMullen~\cite{McM}. Since $\br*{\log\frac{R}{\lnorm{2}{y-x}}}^{\alpha_0}\jl{\geq\br*{\log\frac{1}{a}}^{\alpha_0}>0}$ for $x,y\in B(\jl{\mb{0}},R)$ \jl{with $0<\lnorm{2}{y-x}<aR$ and $a\in(0,1)$}, Theorem~\ref{thm:main_result_weak} is a stronger statement than \cite[Thm.~1.1]{BK1} and the corresponding result of \cite{McM}. 

In the context of McMullen's and Burago and Kleiner's work on separated nets, the restatement above can be compared to a consequence of McMullen's result \cite[Thm.~5.1]{McM}: \jl{n}amely, that any separated net $X\subseteq \R^{d}$ admits $\beta\in(0,1)$ and a bijection $f\colon X\to\Z^{d}$ satisfying
\begin{linenomath}
\begin{equation*}
\limsup_{R\to\infty}\sup_{x\neq y\in X\cap \dom(g)}\frac{\lnorm{2}{g(y)-g(x)}}{\lnorm{2}{y-x}}\cdot \br*{\frac{\lnorm{2}{y-x}}{R}}^\beta<\infty
\end{equation*}
\end{linenomath}
for both $g=f$ and $g=f^{-1}$.

To verify this, take a homogeneous $\alpha$-bi-Hölder bijection $f\colon X\to \Z^{d}$ and apply \eqref{eq:growth_condition} to deduce the inequalities above with $\beta=1-\alpha$.

\subsection*{\hl{Connection to the notion of displacement equivalence.}}
\addcontentsline{toc}{subsection}{Connection to the notion of displacement equivalence.}
Theorem~\ref{thm:main_result_weak} may be of interest for the study of bounded displacement equivalence of separated nets; two separated nets $X,Y\subset\R^d$ are \emph{bounded displacement} equivalent if there is a bijection $f\colon X\to Y$ such that the quantity
	\begin{linenomath}	
	\begin{equation*}
		\sup_{x\in X}\lnorm{2}{f(x)-x}
	\end{equation*}
	\end{linenomath}
	is finite. Equivalence of nets under mappings with bounded displacement was studied in the work of Laczkovich~\cite{laczkovich92}, who characterised separated nets $X$ bounded displacement equivalent to $\Z^d$ in terms of the density of $X$ in bounded subsets of $\R^d$. In several subsequent works, such as \cite{Solomon2011, Aliste, Solomon2014} or \cite{HKW2014}, researchers were mainly interested in determining whether separated nets from some natural class are bounded displacement equivalent to $\Z^d$.
	On the other hand, it is easy to see that bounded displacement equivalence is a \emph{stronger} notion than bilip\-schitz equivalence. Therefore, a separated net in $\R^d$ bilipschitz non-equivalent to $\Z^d$ (such as the ones provided by \cite{BK1}) is, in particular, non-equivalent to the integer lattice in the sense of bounded displacement.
	In this direction, however, Theorem~\ref{thm:main_result_weak} provides us with stronger information:
	\begin{prop}\label{prop:first_disp}
	For any separated net $X\subseteq \R^{d}$ satisfying the assertions of Theorem~\ref{thm:main_result_weak} and any bijection $f\colon X\to\Z^{d}$ we have that 
	\begin{linenomath}	
	\begin{equation}\label{eq:displacement_grows_log}
		\limsup_{R\to\infty}\sup_{x\in B(\mb{0},R)\cap X}\frac{\lnorm{2}{f(x)-x}}{\left(\log R\right)^{\alpha_{0}}}=\infty.
	\end{equation}
	\end{linenomath}
\end{prop}
	In other words, any bijection $f\colon X\to\Z^{d}$ displaces points inside the ball $B(\mb{0},R)$ by much more than $(\log R)^{\alpha_{0}}$ for arbitrarily large $R$. 
	
	It is easy to find separated nets in $\R^{d}$ for which a stronger condition than \eqref{eq:displacement_grows_log} holds. Indeed this is true for any separated net not having natural density one. Given a separated net $X\subseteq \R^{d}$, its
	\emph{natural density} is defined as the limit
	\begin{equation}\label{eq:natural_density_1}
	\lim_{R\to\infty}\frac{\abs{X\cap \cl{B}(\mb{0},R)}}{\leb(\cl{B}(\mb{0}, R))}
	\end{equation}
	provided that the limit exists. By an easy comparison of volumes\footnote{\jl{Say, $\lim_{k\to\infty}\abs{X\cap\cl{B}(\mb{0},R_k)}/\abs{\Z^d\cap\cl{B}(\mb{0},R_k)}=:\lambda>1$ for some $\br*{{R_k}}_{k\in\N}, R_k\to\infty$.
	Then there is $c:=c(d,\lambda)>1$ such that $\abs{X\cap\cl{B}(\mb{0},R_k)}>\abs{\Z^d\cap\cl{B}(\mb{0},cR_k)}$ for all $k$ large enough. Consequently, every bijection $X\to\Z^d$ must map some point from $X\cap\cl{B}(\mb{0},R_k)$ outside $\cl{B}(\mb{0},cR_k)$, and hence, displace it by at least $(c-1)R_k$, for all $k$ large enough.
	}}, one can see that if the natural density of $X$ is not defined or does not equal to one, then
	\begin{linenomath}	
		\begin{equation*}
		\limsup_{R\to\infty}\sup_{x\in B(\mb{0},R)\cap X}\frac{\lnorm{2}{f(x)-x}}{R}>0
		\end{equation*}
	\end{linenomath}
	for any bijection $f\colon X\to \Z^{d}$. This is a stronger property than \eqref{eq:displacement_grows_log}.
	
	Therefore, Proposition~\ref{prop:first_disp} is interesting only if there are separated nets $X\subseteq \R^{d}$ satisfying the conclusion of Theorem~\ref{thm:main_result_weak} which have natural density one (i.e., the natural density of $\Z^d$). We verify that this is indeed the case: 
	
	\begin{prop}\label{p:displacement}
	For every $d\geq 2$ there is a separated net $X\subseteq \R^{d}$ satisfying the assertions of Theorem~\ref{thm:main_result_weak} with natural density one.
	\end{prop}
	We show \jl{using} Lemma~\ref{lemma:displacement} that the separated nets arising in the proof of Theorem~\ref{thm:main_result_weak} can always be modified so that they additionally have well-defined natural density \eqref{eq:natural_density_1}. Proposition~\ref{p:displacement} then follows by observing that \jl{$\omega$-regularity} of separated nets is invariant under scaling; thus, we can scale any separated net with well-defined natural density to a separated net with natural density one \jl{without affecting whether it is $\omega$-regular or not. Details are included at the end of Subsection~\ref{subs:displacement}}.

\subsection*{\hl{Parallel results in the continuous setting.}}
\addcontentsline{toc}{subsection}{Parallel results in the continuous setting.}
Our approach to proving Theorems~\ref{thm:main_result_weak} follows the strategy established by McMullen~\cite{McM} and Burago, Kleiner~\cite{BK1}, where existence of bilipschitz non-equivalent nets is established via the construction of non-realisable density\footnote{By a \emph{density} we mean a non-negative integrable function. Sometimes such functions are called \emph{weights} in the literature.} functions.
McMullen~\cite{McM} and Burago, Kleiner~\cite{BK1} constructed examples of bounded measurable functions $\rho\colon [0,1]^d\to(0,\infty)$ with $\inf\rho>0$ that do not admit any bilipschitz solution $f\colon [0,1]^d\to\R^d$ to the \emph{pushforward equation}
\begin{equation}\label{eq:pushforward}
	f_{\sharp}\rho\leb=\rest{\leb}{f([0,1]^{d})}.
\end{equation}
In fact, Burago and Kleiner~\cite{BK1} produced $\rho$ that is, in addition, continuous. These examples answered another long-standing open question, reportedly posed by Moser and Reimann in the 1960's (see \cite[p.1]{BK2}), and fall into a larger area of research in which the solvability of \eqref{eq:pushforward} is studied under various regularity assumptions on $\rho$ and/or imposing restrictions on the solution $f$; see, e.g., \cite{Moser, Reimann, DMo, Ye, RY, AlgorithmicRY, CDK09}. Accordingly, the present work also contributes to this setting; more specifically to the question of solvability of \eqref{eq:pushforward} for $f$ of a prescribed regularity.

Rivi\'ere and Ye~\cite{RY} prove that \eqref{eq:pushforward} admits a Hölder solution $f$ whenever $\rho\in L^{\infty}([0,1]^{d})$. If $\rho$ is in addition continuous, such a solution $f$ may be found in the intersection of all bi-Hölder classes. Given that Burago and Kleiner~\cite{BK1} construct continuous $\rho$ for which \eqref{eq:pushforward} has no bilipschitz solutions, this result appears close to optimal. It suggests that the class of bilipschitz mappings may be the largest class of homeomorphisms in which solutions to \eqref{eq:pushforward} do not always exist for continuous $\rho$.
However, the aforementioned results leave open the question of solvability of \eqref{eq:pushforward} inside classes of homeomorphisms $f$ lying in between bilipschitz and bi-Hölder. This provides strong motivation for examining moduli of continuity $\omega$ lying asymptotically in between the Lipschitz and Hölder moduli, such as $\omega$ of the form of Theorem~\ref{thm:main_result_weak}. For such $\omega$, we show that the corresponding class of homeomorphisms, which is larger than the bilipschitz class, may fail to provide solutions to \eqref{eq:pushforward} for continuous $\rho$. Thus, we narrow the gap left open by the results of Rivi\'ere and Ye and Burago and Kleiner.

To state this result formally, we need to introduce additional definitions.
We will use the notation $f\colon A\hookrightarrow B$ to signify an injective mapping $A\to B$. For such a mapping we write $f^{-1}$ for the inverse mapping $f^{-1}\colon f(A)\to A$.
\begin{define}\label{d:omega_map}
Given a strictly increasing function $\omega\colon (0,a)\hookrightarrow (0,\infty)$
with $\lim_{t\to 0}\omega(t)=0$, we call a mapping $f\colon A\subseteq\R^{d}\to \R^{k}$ an \emph{$\omega$-mapping} or \emph{$\omega$-continuous} if there is a constant $K>0$ such that
\begin{linenomath}
\begin{equation*}
	x,y\in A\jl{, \enorm{y-x}<a} \qquad\Longrightarrow \qquad \lnorm{2}{f(y)-f(x)}\leq K\omega(\lnorm{2}{y-x}).
\end{equation*}
\end{linenomath}
\end{define}

We call a mapping $f\colon A\subseteq \R^{d}\hookrightarrow \R^{k}$ a \emph{bi-$\omega$-mapping} if both $f$ and $f^{-1}$ are $\omega$-continuous.

\begin{thm}\label{thm:non_rlz_dsty_weak}
Let $d\geq 2$. Then there is $\alpha_0=\alpha_0(d)>0$ with the following property:	
	\jl{l}et $\omega\colon (0,a)\hookrightarrow(0,\infty)$ have the form 
	\begin{linenomath}	
	\begin{equation*}
	\omega(t)=t\left(\log\frac{1}{t}\right)^{\alpha_0}.
	\end{equation*}
	\end{linenomath}	
	Then there is a continuous function $\rho\colon [0,1]^{d}\to (0,\infty)$ for which the pushforward equation \eqref{eq:pushforward}
	admits no bi-$\omega$ solution $f\colon [0,1]^{d}\hookrightarrow \R^{d}$.
\end{thm}

We show that Theorem~\ref{thm:non_rlz_dsty_weak} implies Theorem~\ref{thm:main_result_weak} in Section~\ref{s:main}, Lemma~\ref{lemma:discrete_to_cts}, following the approach used by Burago and Kleiner~\cite{BK1}.

In what follows, we will call a function $\rho\hl{\colon[0,1]^{d}\to (0,\infty)}$ which admits a bi-$\omega$ solution $f$ to \eqref{eq:pushforward} \emph{bi-$\omega$ realisable}. In the particular case of $\omega(t)=t$, we call a bi-$\omega$ realisable function \emph{bilipschitz realisable}.
In \cite{DKK2018} the authors and Kopecká strengthen the bilipschitz non-realisability results of Burago and Kleiner~\cite{BK1} and McMullen~\cite{McM} in a direction different to that of Theorem~\ref{thm:non_rlz_dsty_weak}.
They show that there are continuous densities on $[0,1]^d$ for which the pushforward equation~\eqref{eq:pushforward} admits no \emph{Lipschitz}
solution $f\colon [0,1]^d\to\R^d$.
In fact, \cite{DKK2018} establishes that the set of those continuous functions inside the space $C([0,1]^d)$ that admit a Lipschitz solution $f$ to the pushforward equation~\eqref{eq:pushforward} is \emph{$\sigma$-porous}\footnote{The definition of porosity and $\sigma$-porosity appears in Section~\ref{s:def}. For the time being, it suffices to know that every $\sigma$-porous set is meagre \jl{(of first Baire category), and thus, in a complete metric space, cannot contain any non-empty open set}. \jl{From this it follows that Theorem~\ref{thm:non_rlz_dsty} is stronger than Theorem~\ref{thm:non_rlz_dsty_weak}.}}.
In this sense, almost all continuous functions are not Lipschitz realisable, and thus, also not bilipschitz realisable. The same paper \cite{DKK2018} also shows that the set of $L^\infty([0,1]^d)$ densities that are bilipschitz realisable in the sense of \eqref{eq:pushforward} is $\sigma$-porous. Independently, results in this direction were also obtained by Viera~\cite{Viera_final}.

In this work, we show that Theorem~\ref{thm:non_rlz_dsty_weak} allows for the stronger conclusion as in the aforementioned result of \cite{DKK2018}, establishing that bi-$\omega$ realisable densities for $\omega$ as in the statement of Theorem~\ref{thm:non_rlz_dsty_weak} form a negligible subset of the whole spaces of continuous/integrable densities.
\begin{thm}\label{thm:non_rlz_dsty}
Let $d\geq 2$. Then there is $\alpha_0=\alpha_0(d)>0$ with the following property:	
	\jl{l}et $\omega\colon (0,a)\hookrightarrow(0,\infty)$ have the form 
	\begin{linenomath}	
	\begin{equation*}
	\omega(t)=t\left(\log\frac{1}{t}\right)^{\alpha_0}.
	\end{equation*}
	\end{linenomath}	
	Then the set of functions $\rho\in C([0,1]^{d})$ for which the pushforward equation~\eqref{eq:pushforward}
	admits a bi-$\omega$ solution $f\colon [0,1]^{d}\hookrightarrow \R^{d}$ forms a $\sigma$-porous subset of the space $C([0,1]^d)$ of continuous functions with the supremum norm.
	
	The analogous result is true in the space $L^\infty([0,1]^d)$ as well.
\end{thm}
\begin{remark*}
	For functions $\rho\in C([0,1]^{d})$ attaining negative values the left hand side of the pushforward equation \eqref{eq:pushforward} should be understood as the signed measure
	\begin{linenomath}
	\begin{equation*}
	f_{\sharp}\rho\leb=f_{\sharp}\rho^{+}\leb-f_{\sharp}\rho^{-}\leb.
	\end{equation*}
	\end{linenomath}
	Of course any function $\rho\in C([0,1]^{d})$ attaining negative values cannot admit any injective solution $f$ of the pushforward equation. Therefore, the set of functions $\rho\in C([0,1]^{d})$ referred to in Theorem~\ref{thm:non_rlz_dsty} is contained in the set of positive valued functions in $C([0,1]^{d})$.
\end{remark*}

The proof of Theorem~\ref{thm:non_rlz_dsty} is based on geometric properties of bi-$\omega$-mappings investigated in Section~\ref{section:geometric}. This part of our work is an extensive refinement of the similar investigations of properties of bilipschitz mappings in \cite{DKK2018}, which in turn, were extracted from the work of Burago and Kleiner~\cite{BK1}.

\section{Preliminaries and Notation.}\label{s:def}
We write $B(x,r)$ for the open \hl{Euclidean} ball of radius $r$ centred at $x$; the corresponding closed ball is denoted by $\cl{B}(x,r)$. Moreover, if $B=B(x,r)$, then by $cB$ we mean $B(x,cr)$.
We extend this notation to tubular neighbourhoods of sets in a natural way. We write $\cl{A}, \intr{A}$ and $\partial A$ for the closure, interior and boundary of $A$, respectively.
The expression $\diam(A)$ stands for the diameter of the set $A$. 
Let $k\in\N$; we denote by $[k]$ the set $\set{1,\ldots, k}$.
We write $I^d$ for the unit cube $[0,1]^d$ and $\mb{0}$ for the origin in $\R^d$. We use the symbol $:=$ to signify a definition by equality.

\jl{We use the standard asymptotic notation $O, \Omega$ and $\Theta$. Formally, for two positive functions $\phi, \psi$ defined on $(0,a)$ for some $a>0$ we say that $\psi\in O(\phi)$ if $\limsup_{x\to 0} \psi(x)/\phi(x)<\infty$. Moreover, $\psi\in\Omega(\phi)$ if $\phi\in O(\psi)$ and $\Theta(\phi):=O(\phi)\cap\Omega(\phi)$. Occasionally we will use the same notation for sequences of real numbers $(x_i)_{i\in \N}$ instead of functions, in which case the classes are defined analogously, but according to the asymptotics as $i\to \infty$.} 

Throughout the article we use expressions of the type $\phi(x,y,z)$ to denote a parameter $\phi$ depending only on $x$, $y$ and $z$, but in some cases these dependencies are suppressed after the first appearance.

Moreover, we use the notation $\poly{\hl{t}}$ to denote a function of type $\Theta(\hl{t}^\alpha)$, where the particular power $\alpha>0$ is irrelevant, may depend on other objects present and change from occurrence to occurrence.
In order to emphasise on which parameters the function of the form $\poly{\hl{t}}$ may depend and of which it is independent, we sometimes use the extended notation $\hl{\pl{x,y}{z}{t}$}. This denotes a function of $t$ in $\Theta(t^\alpha)$, where the value of $\alpha$ \hl{can depend only on $x$ and $y$} \hl{and} the implicit multiplicative constant may depend on \hl{$x$, $y$ and $z$} but not on any other parameters present. \hl{The lists of the parameters in the super- and subscript can also have different sizes than in the previous example or be empty.}

Given a mapping $f$ defined on a set $A$ and $B\subset A$, we denote by $\rest{f}{B}$ the restriction of $f$ to $B$. We use the same notation for restrictions of measures as well. We write $\leb_d$ for the $d$-dimensional Lebesgue measure. If the dimension is understood, we usually drop the subscript and write just $\leb$.
Given an integrable function $\rho\colon A\subseteq\R^d\to [0, \infty)$, we write $\rho\leb$ for the measure defined via the formula
\begin{linenomath}
$$
\rho\leb(S):=\niceint{S}{\rho}{\leb}, \qquad S\subseteq A.
$$
\end{linenomath}
We refer to such a function $\rho\colon A\to [0,\infty)$ as a \emph{density}.
Given a measurable mapping $f\colon A\to \R^d$, we define the \hl{\emph{pushforward}} of a measure $\nu$ as $f_\sharp\nu(S):=\nu(f^{-1}(S))$\hl{, where $S\subseteq\R^d$}.

Let $A\subset \R^d$. We write $C(A)$ for the space of continuous functions $A\to\R$ with the supremum norm. 
We denote by $L^\infty(A)$ the space of essentially bounded functions $A\to\R$ with the $L^\infty$-norm.

Let $(X,d)$ be a metric space. We call a set $P\subseteq X$ \emph{porous} if for every $x\in X$ there are $\varepsilon_{0}>0$ and $\alpha\in (0,1)$ such that for every $\varepsilon\in(0,\varepsilon_{0})$ there exists $y\in X$ satisfying
	$d(y,x)\leq\varepsilon$ and $B(y,\alpha\varepsilon)\cap P=\emptyset$.
A set $E\subseteq X$ is called \emph{$\sigma$-porous} if it may be expressed as a countable union of porous subsets of $X$. For the original definitions and more background on these sets, see \cite{zajicek2005}.

\begin{define}\label{def:sep_net}
Given $A\subset X$ in the metric space $(X,d)$ and a number $r>0$, we say that $A$ is \emph{$r$-separated} if $d(a,a')\geq r$ for every $a\neq a' \in A$. We say that $A$ is an \emph{$r$-net} of $X$ if $d(x,A)\leq r$ for every $x\in X$. If there are $r,s>0$ such that $A$ is \jl{an} $s$-separated $r$-net, we call $A$ a \emph{separated net}.
\end{define}

We write $\mb{e}_{1},\ldots,\mb{e}_{d}$ for the standard basis of $\R^{d}$. For $\lambda>0$ we let $\mathcal{Q}^{d}_{\lambda}$ denote the standard tiling of $\R^{d}$ by cubes of sidelength $\lambda$ and vertices in the set $\lambda\Z^{d}$. We call a family of cubes \emph{tiled} if it is a subfamily of $\mathcal{Q}^{d}_{\lambda}$ for some $\lambda>0$. We say that two cubes $S,S'\in \mathcal{Q}^{d}_{\lambda}$ are $\mb{e}_{1}$-\emph{adjacent} if $S'=S+\lambda\mb{e}_{1}$.

\paragraph{Moduli of continuity.}
\begin{define}\label{def:modulus}
We use the term \emph{modulus of continuity} to refer to a strictly increasing, concave function $\omega\colon (0,a)\hookrightarrow (0,\infty)$
with $\lim_{t\to 0}\omega(t)=0$.
\end{define}
\jl{The exact value of the parameter $a$ in the preceding definition will mostly be immaterial; for instance, Theorems~\ref{thm:main_result_weak} or \ref{thm:non_rlz_dsty} do not refer to $\omega$ described there with a specific value of $a$ in mind, but include all possible choices of the parameter $a$. Thus, to simplify the formulas and improve readability, we introduce a general convention that whenever
any modulus of continuity $\omega$ appears in a logical statement, the statement implicitly applies only to arguments of $\omega$ that are smaller than $a$. As a simple example of use of our convention, we can write the condition in Definition~\ref{d:omega_map} in a shorter form as `$x,y\in A \Longrightarrow \enorm{f(y)-f(x)}\leq K\omega(\enorm{y-x})$'.
Our convention dictates that this condition should be interpreted as void for pairs $x,y\in A$ with $\lnorm{2}{y-x}\geq a$.}

For \jl{a modulus of continuity} $\omega$ and a mapping $f\colon A\subseteq \R^{d}\to\R^{\hl{k}}$ we let
\begin{linenomath}
\begin{equation*}
	\mathfrak{L}_{\omega}(f):=\sup\left\{\frac{\lnorm{2}{f(y)-f(x)}}{\omega(\lnorm{2}{y-x})}\colon x,y\in A,\,0<\lnorm{2}{y-x}< a\right\}.
\end{equation*}
\end{linenomath}
In the case that $f$ is injective, we further define 
\begin{linenomath}
\begin{equation*}
\bilipomeg(f):=\max\set{\lipomeg(f),\lipomeg(f^{-1})}.
\end{equation*}
\end{linenomath}
Note that $f$ is an $\omega$-mapping (see Definition~\ref{d:omega_map}) if and only if $\mathfrak{L}_{\omega}(f)<\infty$. Moreover, in the simplest case of $\omega(t)=t$ the quantity $\mathfrak{L}_{\omega}(f)$ coincides with the Lipschitz constant of $f$; in this case, we use the standard notation $\lip(f)$ and $\bilip(f)$ instead of $\mathfrak{L}_{\omega}(f)$ and $\bilipomeg(f)$, respectively.

\begin{remark*}
	We defined the notion of a homogeneous $\omega$-mapping in Definition~\ref{d:hom_omega_map} only for mappings of separated nets because we will not use it for any other domains, but it extends naturally to mappings of arbitrary subsets of $\R^{d}$ to $\R^{k}$. However, we caution the reader that the terms~$\omega$-mapping and homogeneous $\omega$-mapping can then be misleading.  Although the terminology might suggest that the notion of homogeneous $\omega$-mapping is stronger than that of an $\omega$-mapping, the two notions are actually incomparable.
\end{remark*}

We will restrict our attention to moduli of continuity with various special properties. However, we show that this class of moduli is still diverse (see Lemma~\ref{lemma:tlogp}).

For $0< a\leq 1$ we call a function $\phi\colon (0,a)\to(0,\infty)$ \emph{submultiplicative} if 
\begin{linenomath}
\begin{equation*}
\phi(st)\leq \phi(s)\phi(t)
\end{equation*}
\end{linenomath}
for all $s,t\in (0,a)$.
Similarly, we call $\phi$ as above \emph{supermultiplicative} if
\[
\phi(st)\geq \phi(s)\phi(t)
\]
for all $s,t\in (0,a)$.
Observe that for a strictly increasing $\phi$ the function $\phi$ is submultiplicative if and only of $\phi^{-1}$ is supermultiplicative.
\begin{define}\label{def:M}
Let $\mc{M}$ denote the set of strictly increasing, concave
and submultiplicative functions
\begin{linenomath}
\begin{equation*}
\omega\colon(0,a)\hookrightarrow(0,\infty)
\end{equation*}
\end{linenomath}
with $a\leq \hl{\frac{1}{2}}$, $\lim_{t\to 0}\omega(t)=0$ and $\omega(t)\geq t$ for all $t\in(0,a)$.
Given $\omega \in \mc{M}$, we will denote by $a_\omega$ the upper end of the domain of $\omega$.
\jl{Occasionally, we use $\omega(a_\omega)$ as a convenient abbreviation for $sup_{t\in(0,a_\omega)} \omega(t)$.} 
\end{define}
Note that whenever $\omega\in\mc{M}$, then also $\rest{\omega}{(0,b)}$ belongs to $\mc{M}$ for every $b\in(0,a_\omega)$. However, the classes of $\omega$- and $\omega|_{(0,b)}$-mappings coincide on \jl{m}any  domain\jl{s} $A\subseteq \R^{d}$; in the case where $A$ is convex, this follows easily from Definition~\ref{d:omega_map} by the triangle inequality and the concavity of $\omega$. For \jl{more} general \jl{domains} $A\subseteq \R^{d}$, the \jl{equality} 
of the classes follows from the convex case \jl{whenever} any $\omega$-continuous mapping $A\to\R^k$ can be extended to an $\omega$-continuous mapping on \jl{$\operatorname{conv}(A)$, the convex hull of $A$}. \jl{The existence of the extension in the only case in which we will need it is established in the next lemma and follows easily from \cite[Thm.~1.12]{benyamini1998geometric}.}
\begin{lemma}\label{lemma:ex}
	Let $\omega\colon (0,a)\to (0,\infty)$ be a	
	modulus of continuity. Let $\Gamma\subseteq \R^{d}$ be a bounded set such that $\Gamma$ is an $\frac{a}{4}$-net of its convex hull $\operatorname{conv}(\Gamma)$ and let $f\colon \Gamma\to \R^{d}$ be an $\omega$-mapping. Then there exists an $\omega$-mapping $F\colon \R^{d}\to \R^{d}$ such that $F|_{\Gamma}=f$ and
	\begin{linenomath}
	\begin{equation*}
		\lipomeg(F)\leq \left(\frac{4\diam \Gamma}{a}+1\right)\lipomeg(f).
	\end{equation*}
	\end{linenomath}
\end{lemma}
\begin{proof}
We first extend $\omega$ continuously to $a$ and
define a concave function $\cl{\omega}\colon (0,\infty)\to (0,\infty)$ by
	\begin{linenomath}
	\begin{equation*}
		\cl{\omega}(t)=\begin{cases}
			\omega(t) & \text{if }t\in (0,a),\\
			\omega(a) & \text{if }t\geq a.
		\end{cases}
	\end{equation*}
	\end{linenomath}
We will show that
\begin{equation}\label{eq:thing_to_prove}
\lnorm{2}{f(y)-f(x)}\leq \left(\frac{4\diam \Gamma}{a}+1\right)\lipomeg(f)\cdot\cl{\omega}\left(\lnorm{2}{y-x}\right)\quad \text{for every $x,y\in\Gamma$.}
\end{equation}
Once \eqref{eq:thing_to_prove} is established, the proof of the lemma is completed simply by applying \cite[Thm.~1.12]{benyamini1998geometric}. So let us prove \eqref{eq:thing_to_prove}. Let $x,y\in \Gamma$, assuming, as we may, that $\lnorm{2}{y-x}\geq a$, and set $z_{0}:=x$. If $i\geq 1$ and $z_{i-1}$ is already chosen, we draw a (possibly degenerate) line from $z_{i-1}$ to $y$. If this line has length less than $a$, we set $y=z_{i}$. Otherwise we let $z_{i}'$ denote the point on this line at distance $\frac{a}{2}$ from $z_{i-1}$ and choose $z_{i}\in\cl{B}\left(z_{i}',\frac{a}{4}\right)\cap \Gamma$ arbitrarily. 
Inductively, we verify that
\begin{linenomath}
\begin{equation*}
\lnorm{2}{z_{i}-z_{i-1}}<a \quad\text{and}\quad	0\leq \lnorm{2}{z_{i}-y}\leq \max\set{\lnorm{2}{y-x}-\frac{ia}{4},0}
\end{equation*}
\end{linenomath}
for every $i\in\N$. It follows that there exists $N\leq \frac{4\lnorm{2}{y-x}}{a}+1$ such that $z_{N}=y$. We therefore obtain the estimate
\begin{linenomath}
\begin{equation*}
\frac{\lnorm{2}{f(y)-f(x)}}{\cl{\omega}\left(\lnorm{2}{y-x}\right)}\leq \frac{N\lipomeg(f)\omega(a)}{\omega(a)}\leq \left(\frac{4\diam \Gamma}{a}+1\right)\lipomeg(f).\qedhere
\end{equation*}
\end{linenomath}
\end{proof}
\jl{Similarly, also the classes of homogeneous $\omega$- and $\rest{\omega}{(0,b)}$-mappings are the same; but in this case, this follows immediately from Definition~\ref{d:hom_omega_map}. The constant $a$ from that definition can always be taken to be equal to $a_\omega$. These facts also support the convention that we introduced below Definition~\ref{def:modulus} about arguments of $\omega$ in various conditions.}
Additionally, for every $L\geq 1$ and $\omega\in \mc{M}$ the modulus $L\omega$ belongs to $\mc{M}$ as well. This means that $f$ being an $\omega$-mapping with $\lipomeg(f)\leq L$ is equivalent to $f$ being an $L\omega$-mapping with $\lipgen{L\omega}\leq 1$.

It is clear that all Lipschitz and Hölder moduli, i.e., all functions $t\mapsto t^{\alpha}$ with $\alpha\in(0,1]$, belong to the class $\mc{M}$. Our aim is now to show that the class $\mc{M}$ is even larger and contains a diverse spectrum of moduli lying inbetween H\"older and Lipschitz.
Indeed we will show that the functions $t\mapsto t\left(\log\frac{1}{t}\right)^{\alpha}$ for \hl{$\alpha\in(0,\infty)$} belong to $\mc{M}$. The class of $\omega$-mappings $f\colon \R^{d}\to\R^{\hl{k}}$ for such $\omega$ is then larger than the class of Lipschitz mappings, but smaller than that of H\"older.

\begin{lemma}\label{lemma:tlogp}
	For each \hl{$\gamma\in(0,\infty)$} there exists $a\in(0,1)$ so that the function
	\begin{linenomath}
	\begin{equation*}
	\phi_{\gamma}\colon (0,a)\hookrightarrow \R,\qquad t\mapsto t\left(\log\frac{1}{t}\right)^{\gamma}
	\end{equation*}
	\end{linenomath}
	belongs to $\mc{M}$.
\end{lemma}
\begin{proof}
	Fix \hl{$\gamma\in(0,\infty)$}. We determine sufficient conditions on $a\in(0,1)$. First we require that $a<e^{-1}$ so that $\phi_{\gamma}(t)\geq t$ for all $t\in (0,a)$.  
	Choosing \hl{$a\in(0,e^{-\gamma})$} so that $\log(1/t)\geq \gamma$ for every $t\in (0,a)$, one can easily verify that $\phi_\gamma$ is concave and strictly increasing. It only remains to check that $\phi_{\gamma}$ is submultiplicative. We impose the \hl{additional} condition $a\leq \frac{1}{e^{2}}$. Then for $s,t\in(0,a)$ we have $\log\left(\frac{1}{st}\right)\leq \log\left(\frac{1}{s}\right)\log\left(\frac{1}{t}\right)$ and therefore
	\begin{linenomath}
	\begin{align*}
	\phi_{\gamma}(st)=st\left(\log\frac{1}{st}\right)^{\gamma}\leq s\left(\log\frac{1}{s}\right)^{\gamma}t\left(\log\frac{1}{t}\right)^{\gamma}=\phi_{\gamma}(s)\phi_{\gamma}(t).
	\end{align*}
	\end{linenomath}
\end{proof}

We will also briefly use the \emph{Hausdorff dimension} of a set $A$, which we denote by $\dim_H(A)$ (for the definition and its basic properties, see, e.g., \cite{Mattila}). The following \hl{classical} lemma is an easy consequence of the definition of the Hausdorff dimension:
\begin{lemma}\label{lem:Hausdorff}
Let $f\colon \R^d\to\R^{\hl{k}}$ be a continuous mapping that is $\alpha$-Hölder continuous for an \hl{$\alpha\in(0,1)$}. Then for every $A\subseteq\R^d$ we have $\dim_H(f(A))\leq\frac{1}{\alpha}\dim_H(A)$.
\end{lemma}

We shall use the following (standard) corollary of the lemma above to bound above the Lebesgue measure of neighbourhoods of $f$-images of sets under a bi-$\omega$-mapping $f$.
\begin{cor}\label{cor:Hausdorff}
Let $f\colon I^d\to\R^d$ be a homeomorphism that is $\alpha$-Hölder continuous for some \hl{$\alpha\in\left(\frac{d-1}{d},1\right)$}. Then
\begin{linenomath}
$$
\lim_{\varepsilon\to 0}\leb\br*{B(\partial f(I^d), \varepsilon)}= 0.
$$
\end{linenomath}
\end{cor}
\begin{proof}
Since $f(\partial I^d)=\partial f(I^d)$, Lemma~\ref{lem:Hausdorff} implies that $\dim_H(\partial f(I^d))< d$. This, in turn, means that for every $\delta>0$ there is a collection $(B_i)_{i\in\N}$ of balls of radius at most $\delta$ covering $\partial f(I^d)$ such that $\sum_{i=1}^\infty \diam(B_i)^d\leq \delta$.
Moreover, $\partial f(I^d)$ is compact. Thus, we can assume that there is $k\in\N$ such that $\partial f(I^d)$ is already covered by $B_1, \ldots, B_k$. Let $r:=\min_{i\in[k]}\diam(B_i)$. Then
\begin{linenomath}
$$
\bigcup_{i=1}^k 2B_i\supseteq B(\partial f(I^d), r),
$$
\end{linenomath}
which implies that the Lebesgue measure of the latter set is at most $C\delta$ for an absolute constant $C$.
\end{proof}
	\section{Proof of Main Results.}\label{s:main}

\subsection{\texorpdfstring{\jl{$\omega$-regularity}}{\unichar{"03C9}-regularity}}
In this \hl{sub}section we give a proof of our main results Theorems~\ref{thm:main_result_weak} and~\ref{thm:non_rlz_dsty}, partially based on a geometric statement for bi-$\omega$-mappings $\R^{d}\to\R^{d}$ which will be proved in  Section~\ref{section:geometric}. 
Our first objective is to verify that
Theorem~\ref{thm:main_result_weak} is implied by Theorem~\ref{thm:non_rlz_dsty}. We will need one lemma on uniform convergence to a homeomorphism.
\begin{lemma}\label{lem:close_fun_close_img}
	Let $f\colon I^d\to\R^d$ be a homeomorphism and $g\colon I^d\to\R^d$ be continuous. Then $g(I^d)\Delta f(I^d)\subseteq \cl{B}(\partial f(I^d), \lnorm{\infty}{f-g})$, where the notation $\Delta$ denotes the set difference $E\Delta F:=(E\setminus F)\cup (F\setminus E)$.
\end{lemma}
In the proof of the lemma, we will use the topological degree of Brouwer; for its definition and properties, see~\cite{Deim}, for example. By $\deg(f, U, y)$ we denote the degree of a continuous mapping $f\colon \cl{U}\to \R^d$ at the point $y\in\R^d\setminus f(\partial U)$ with respect to the open, bounded set $U\subset \R^d$.
\begin{proof}[Proof of Lemma~\ref{lem:close_fun_close_img}.]
	It is clear that
	\begin{linenomath}
	$$
	g(I^d)\setminus f(I^d)\subseteq\cl{B}\br*{\partial f(I^d), \lnorm{\infty}{f-g}},
	$$
	\end{linenomath}
	since $g(I^d)\subseteq\cl{B}\br*{f(I^d), \lnorm{\infty}{f-g}}$ and the distance from a point $y\in g(I^d)\setminus f(I^d)$ to the set $f(I^d)$ is realised on the boundary of $f(I^d)$, as the latter is compact.
	
	For the other, less clear inclusion,
	we use the topological degree. The multiplication theorem for the degree (see, e.g.~\cite[Thm.\ 5.1]{Deim}) implies that the degree of a homeomorphism is always $\pm 1$, i.e., $\deg(f,\intr I^d, \cdot)$ is constant equal to $\pm 1$. On the other hand, it is a basic property of the degree that $\deg(g,\intr I^d, y)=\deg(f,\intr I^d, y)$ whenever $\dist(y,\partial f(I^d))>\lnorm{\infty}{f-g}$; see \cite[Thm.\ 3.1(d5)]{Deim}. Another basic property is that the degree of any function with respect to a set is zero in every point which is not included in the image of that set; see \cite[Thm.\ 3.1(d4)]{Deim}. That is, every $y\in f(I^d)$ such that $\dist(y,\partial f(I^d))>\lnorm{\infty}{f-g}$ must be included in $g(I^d)$ as well.
\end{proof}

We will also need two auxiliary lemmas on weak convergence of measures which are probably a common part of knowledge in measure theory.
\hl{For their proofs, we refer the reader to \cite[Lem.~5.5 and 5.6]{DKK2018}.}
\begin{lemma}\label{lemma:weakconvcrit}
Let $\nu$ and $(\nu_{n})_{n=1}^{\infty}$ be finite Borel measures with support in a compact set $K\subset \R^d$.
Moreover, assume that there is, for each $n\in\N$, a finite collection $\mc{Q}_{n}$ of Borel subsets of $K$ that satisfy the following:
\begin{enumerate}
\item\label{i:l_weakconvcrit1} $\displaystyle \nu\left(K\setminus\bigcup\mc{Q}_n\right)=0$ and $\displaystyle \nu_n\left(K\setminus\bigcup\mc{Q}_n\right)=0$.
\item\label{i:l_weakconvcrit2} $\displaystyle\sum_{Q\in\mc Q_n}\nu(Q)=\nu(K)$ and $\displaystyle\sum_{Q\in\mc Q_n}\nu_n(Q)=\nu_n(K)$.
\item $\displaystyle\lim_{n\to\infty}\max_{Q\in\mc Q_n}\diam(Q)=0$ and $\displaystyle\max_{Q\in\mc{Q}_{n}}\abs{\nu_{n}(Q)-\nu(Q)}\in o\br*{\frac{1}{\abs{\mc{Q}_n}}}$.
\end{enumerate}
Then $\nu_{n}$ converges weakly to $\nu$.
\end{lemma}

\begin{lemma}\label{lemma:weakconvcrit2}
Let $K$ be a compact set in $\R^d$ and $\br{\nu_n}_{n\in\N}$ be a sequence of finite Borel measures on $K$ converging weakly to a finite Borel measure $\nu$. Let $(h_n)_{n\in\N}$, $h_n\colon K\to \R^m$, be a sequence of continuous mappings converging uniformly to a mapping $h$. Then $(h_n)_\sharp(\nu_n)$ converges weakly to $h_\sharp(\nu)$.
\end{lemma}

\jl{Observe that}
the next lemma \jl{combines with Theorem~\ref{thm:non_rlz_dsty_weak} to imply Theorem~\ref{thm:main_result_weak}}. The statement and part of its proof are a completely straightforward adaptation of \cite[Lem.~2.1]{BK1} by Burago and Kleiner. However, the majority of the proof we give below consists of  important details that are missing in \cite{BK1} and have never been published.
Moreover, these missing parts are especially relevant in our setting, where we consider less restrictive moduli of continuity than Lipschitz. In particular, the proof does not work for all \hl{Hölder} moduli of continuity.
In what follows we show that the \jl{lemma} is valid for all moduli of continuity $\omega$ which are sub-Hölder for sufficiently many Hölder moduli of continuity. The precise meaning of sufficiently many is determined by the dimension $d$ of the space $\R^{d}$.

\begin{lemma}\label{lemma:discrete_to_cts}
	Let $\omega\in \mc{M}$ be a modulus of continuity with the property that there is $\delta>0$, $K>0$ and $\alpha\hl{\in\br*{\frac{d-1}{d},1}}$ such that
	\begin{linenomath}	
	\begin{equation}\label{eq:omega_is_Hoelder}
		\omega(t)\leq K t^\alpha \qquad \text{for all $t\in(0,\delta]$.}
	\end{equation}
	\end{linenomath}
	Suppose that every separated net $X\subseteq \R^{d}$ is \jl{$\omega$-regular.}
	Then for any measurable density $\rho\colon I^{d}\to \R$ with $0<\inf\rho\leq\sup\rho<\infty$ there is a bi-$\omega$-mapping $f\colon I^{d}\to \R^{d}$ satisfying
	\begin{linenomath}
	\begin{equation}\label{eq:pushforward_2}
	f_{\sharp}\rho\leb=\rest{\leb}{f(I^{d})}.
	\end{equation}
	\end{linenomath}
\end{lemma}
\begin{proof}
	Fix a measurable density $\rho\colon I^{d}\to\R$ with $0<\inf\rho\leq \sup\rho<\infty$ and a strictly increasing sequence of natural numbers $(l_{k})_{k=1}^{\infty}$ on which we will impose further conditions in the course of the proof. Let $(S_{k})_{k=1}^{\infty}$ be a sequence of axis parallel, pairwise disjoint cubes in $\R^d$ such that each $S_{k}$ has side length $l_{k}$ and 
	\begin{linenomath}
	\begin{equation*}
	\bigcup_{i=1}^{k}S_{i}\subseteq B\left(\mb{0},\hl{d}\sum_{i=1}^{k}l_{i}\right)
	\end{equation*}
	\end{linenomath}
	for all $k\geq 1$. Fix a sequence of natural numbers $(m_{k})_{k=1}^{\infty}$ satisfying
	\begin{linenomath}
	$$
	\lim_{k\to\infty}m_{k}=\infty \qquad\text{ and }\qquad \lim_{k\to\infty}\frac{m_{k}}{l_{k}}=0.
	$$
	\end{linenomath}
	We impose one further condition on $m_{k}$ later on. For each $k\geq 1$ we let $\phi_{k}\colon \R^{d}\to \R^d$ be the unique affine mapping sending $I^d$ onto $S_{k}$ 	
	with scalar linear part and define $\rho_{k}\colon S_{k}\to \R$ by $\rho_{k}:=\rho\circ \rest{\phi_{k}^{-1}}{S_k}$.
	For each $k\geq 1$ we let $(T_{k,i})_{i=1}^{m_{k}^{d}}$ denote the standard partition of the cube $S_k$ into $m_{k}^{d}$ cubes of sidelength $l_k/m_{k}$.
	We further partition each cube $T_{k,i}$ into $n_{k,i}^{d}$ cubes $(U_{k,i,j})_{j=1}^{n_{k,i}^{d}}$ of equal sidelength, where $n_{k,i}\in\N$ is defined as the integer part of $\sqrt[d]{\lint{T_{k,i}}{\rho_{k}}}$.
	In order to make sure that each $n_{k,i}$ is positive, it is enough to choose the numbers $m_k$ and $l_k$ so that $(l_k/m_k)^d \inf\rho\geq 1$ for every $k\in\N$.
	
	\hl{We construct a separated net $X\subseteq \R^{d}$ in two steps. First we place one point at the centre of each cube $U_{k,i,j}$.
	The resulting set is a separated net of $\bigcup_{k=1}^\infty S_k$ because of the boundedness of $\rho$.
	We will not use any information about $X$ outside $\bigcup_{k=1}^\infty S_k$, and therefore, we do not require any further condition on $X$ there.}
	\begin{equation}\label{sentence_extend_net}
		\begin{split}
			&\text{In the second step, we extend $X$ to a separated net in $\R^{d}$ arbitrarily,}\\
			&\text{adding only points outside of $\bigcup_{k=1}^{\infty}S_{k}$.}
		\end{split}		
		\end{equation}
	We have labelled the above sentence because we wish to refer specifically to it in the proof of Proposition~\ref{p:displacement}, \jl{which will be given at the end of Section~3.}

Let $g\colon X\to\Z^{d}$ be a \jl{bijection such that both $g$ and $g^{-1}$ are homogeneous $\omega$-mappings}. For each $k\geq 1$ we let $X_{k}:=\phi_{k}^{-1}(X\cap S_k)=\phi_{k}^{-1}(X)\cap I^d$ and fix a point $z_{k}\in X_k$. Then we define \sug{a bijection} $f_{k}\colon \phi_{\hl{k}}^{-1}(X)\to \sug{\frac{1}{l_{k}}\Z^{d}}$ by
	\begin{linenomath}
	\begin{equation}\label{eq:f_k_def}
	f_{k}(x)=\frac{1}{l_{k}}(g\circ\phi_{k}(x)-g\circ \phi_{k}(z_{k})),\qquad x\in \phi_{\hl{k}}^{-1}(X).
	\end{equation}
	\end{linenomath}
	 We also set $R_{k}:=\hl{d}\sum_{i=1}^{k}l_{i}$. Mainly we are interested in the behaviour of $\rest{f_k}{X_k}$, but for technical reasons that will become clear at the end of the proof, we will need to work with a bit larger portion of $\phi_{k}^{-1}(X)$ than $X_k$ later on.
	 To this end, we fix a closed ball $\cl{B}$ centred at the origin such that $I^d\subset\intr \cl{B}$ and denote by $\cl{r}$ its radius. We define $\cl{X}_k:=\phi_{k}^{-1}(X)\cap \cl{B}$. Clearly, $X_k=\cl{X}_k\cap I^d$.	
		
	To obtain an estimate for the modulus of continuity of $\rest{f_{k}}{\cl{X}_k}$, we fix $x,y\in \cl{X}_{k}$ and observe that $\phi_k(x),\phi_k(y)\in B(\mb{0}, \cl{r}l_k+R_k)$. \hl{By the $\omega$-homogeneous property of $g$ there is $U'\sug{'}>0$ such that}
	\begin{linenomath}
	\begin{equation*}
	\lnorm{2}{f_{k}(y)-f_{k}(x)}\\
	\leq U\hl{'}\sug{'}\frac{\cl{r}l_k+R_{k}}{l_{k}}\omega\left(\frac{l_{k}\lnorm{2}{y-x}}{\cl{r}l_k+R_{k}}\right).
	\end{equation*}
	\end{linenomath}
	We now require a condition on the sequence $(l_{k})_{k=1}^{\infty}$ to ensure that the ratio $(\cl{r}l_k+R_{k})/l_{k}$ is bounded. Since $R_{k}=R_{k-1}+\hl{d}l_{k}$, it is sufficient to take $l_{k}\geq R_{k-1}$ for all $k$. Then $\hl{1\leq} R_{k}/l_{k}\leq \hl{d+1}$, and thus, $\hl{\cl{r}+1\leq}(\cl{r}l_k+R_{k})/l_{k}\leq\cl{r}+\hl{d+1}$.	
	Thus \hl{for $U\sug{'}:=(\cl{r}+d+1)U'\sug{'}$} we get 
	\begin{linenomath}
	\begin{equation}\label{eq:f_k_mod}
	\lnorm{2}{f_{k}(y)-f_{k}(x)}\leq  \hl{U}\sug{'}\omega(\lnorm{2}{y-x})
	\end{equation}
	\end{linenomath}
	for all $k$.
	
Next we use \jl{Lemma~\ref{lemma:ex}} to extend $\rest{f_{k}}{\cl{X}_k}$ to the whole $\cl{B}$ so that the extension, denoted by $\cl{f}_k$, is $\omega$-continuous with
	\begin{equation}\label{eq:lip_omega_f_k}
	\sug{\sup_{k\in \N}}\lipomeg(\cl{f}_{k})\leq U<\infty,
	\end{equation}
	\sug{where $U\geq U'$ is a constant depending only on $U\sug{'}$, $\omega$ and $\diam\cl{B}$. For the finitely many $k\in\N$ for which $\cl{X}_{k}$ is not an $\frac{a_\omega}{4}$-net of its convex hull $\operatorname{conv}(\cl{X})$, the conditions of Lemma~\ref{lemma:ex} are not satisfied; we solve this problem simply by discarding these finitely many $k$'s and relabelling the sequences indexed by $k$ accordingly.}

	We \sug{let $M>1$ be a large enough parameter whose choice will be specified shortly,} $h_k:=\sug{\rest{\br*{f_{k}^{-1}}}{{\cl{B}(\mb{0},M)\cap\frac{1}{l_{k}}\Z^{d}}}}$ and show that \sug{$h_{k}$} is also $\omega$-continuous. \sug{The parameter $M$ is chosen sufficiently large so that $\cl{f}_{k}(\cl{B})\subseteq \cl{B}(\mb{0},M)$ for every $k$. To see that this is possible,}
\sug{n}ote that $\mb{0}\in f_k(\cl{X}_k)\sug{\subseteq \cl{f}_{k}(\cl{B})}$, and thus, $\diam \jl{\cl{\normalcolor f}}_k(\sug{\cl{B}})$ is an upper bound on $\enorm{u}$ for every $u\in \jl{\cl{\normalcolor f}}_k(\sug{\cl{B}})$. 
	We have that 
	\begin{linenomath}
	\begin{equation}\label{eq:u_norm_bound}
	\diam\cl{f}_k\br*{\cl{B}}\leq\jl{\left(\frac{2\cl{r}}{a_\omega}+1\right)}\lipomeg(\cl{f}_k)\omega(a_\omega)
		\leq\jl{\left(\frac{2\cl{r}}{a_{\omega}}+1\right)}U\omega(a_\omega)\sug{,}
	\end{equation}
	\end{linenomath}
	\sug{so it suffices to choose $M$ larger than the latter quantity of \eqref{eq:u_norm_bound}.}
	\sug{Now} we show that \sug{$h_{k}$} is $\omega$-continuous\sug{.}
	In the argument that follows we call on a basic property of homogeneous $\omega$-mappings, namely that any homogeneous $\omega$-mapping $h$ may increase norms of non-zero vectors by at most some constant factor $C_{h}$;
this fact will be referred to as the `scaling property'.
The verification of this property is an easy exercise in the definition of homogeneous $\omega$-mapping, which we leave to the reader. \sug{From \eqref{eq:f_k_def} we get that for $u\in \frac{1}{l_k}\Z^d$
\begin{linenomath}
	\begin{equation*}
	f_k^{-1}(u)=\phi_k^{-1}\br*{g^{-1}\br*{l_ku+c_k}},
	\end{equation*}
\end{linenomath}
where $c_k:=g\br*{\phi_k(z_k)}$. Note that $\phi_k(z_k)$ can be $\mb{0}$ for at most one $k\in\N$, \jl{since $\phi_{k}(z_{k})\in S_{k}$}.}
	The scaling property of $g$\jl{, together with $\phi_{k}(z_{k})\in S_{k}\subseteq \cl{B}(\mb{0},R_{k})$ and $c_{k}=g(\phi_{k}(z_{k}))$, yields that $\lnorm{2}{c_{k}}\leq C_{g}R_{k}$ for every $k\in\N$ large enough. Therefore, for all sufficiently large $k\in \N$ and} every $u\in \frac{1}{l_k}\Z^d\cap\cl{B}(\mb{0},M)$
	\[
	\lnorm{2}{l_ku+c_k}\leq M l_k+C_g R_k\leq (M+C_g)R_k.
	\]
	Hence, for every $u,v\in \frac{1}{l_k}\Z^d\cap\cl{B}(\mb{0},M)$ and $k\in\N$ large enough
	we get that
	\begin{linenomath}
	\begin{equation}\label{eq:f_k_inverse_mod}
	\begin{split}	
	&\lnorm{2}{f_k^{-1}(u)-f_k^{-1}(v)}=\frac{1}{l_k}\lnorm{2}{g^{-1}(l_ku+c_k)-g^{-1}(l_kv+c_k)}\\
	&\leq\frac{L'\sug{'}(M+C_g)R_k}{l_k}\omega\br*{\frac{l_k\lnorm{2}{u-v}}{(M+C_g)R_k}}\leq L'\sug{'}(M+C_g)(d+1)\omega\br*{\enorm{u-v}},
	\end{split}
	\end{equation}
	\end{linenomath}
	where $L'\sug{'}>0$ is the multiplicative constant from the $\omega$-homogeneous property of $g^{-1}$. The last inequality above follows from the relations $1\leq R_k/l_k\leq d+1$. Thus, there is $L\sug{'}>\sug{L''}$ such that for ever\jl{y} $k\in\N$ and all $u,v\in \sug{\cl{B}(\mb{0},M)\cap\frac{1}{l_{k}}\Z^{d}}$
	\begin{equation}\label{eq:h_k_mod}
	\lnorm{2}{h_k(u)-h_k(v)}\leq L\sug{'}\omega\br*{\lnorm{2}{u-v}}.
	\end{equation}

Now we extend each $h_k$ to \sug{$\cl{B}(\mb{0},M)$} preserving its $\omega$-continuity using Lem\-ma~\ref{lemma:ex}; the extension is denoted by $\cl{h}_k$ \sug{and we note that our application of Lemma~\ref{lemma:ex} ensures that $\sup_{k\in\N}\lipomeg(\cl{h}_{k})\leq L<\infty$ for some constant $L\sug{\geq L'}$ depending only on $L\sug{'}$, $\omega$ and M}. \sug{In this step we again discard the at most finitely many indices $k$ for which the conditions of Lemma~\ref{lemma:ex} are not satisfied.} By the Arzel\`a--Ascoli theorem, we may pass to a subsequence of $(\cl{f}_{k})_{k=1}^{\infty}$ \hl{so that both $\br*{\cl{f}_k}_{k=1}^\infty$ and $\br*{\cl{h}_k}_{k=1}^\infty$ converge uniformly. Let $f:=\lim_{k\to\infty}\cl{f}_{k}$, $h:=\lim_{k\to\infty}\cl{h}_{k}$ and note that both $f$ and $h$ are $\omega$-mappings} with $\lipomeg(f)\leq U<\infty$ \hl{and $\lipomeg(h)\leq L<\infty$}.
\hl{We will show that $h(f(x))=x$ for every $x\in\cl{B}$. This implies that $f^{-1}$ is well-defined and equals $\rest{h}{f(\cl{B})}$.}

	There are positive constants $s=s(\rho)$, $b'=b'(\rho)$ such that $X$ is an $s$-separated $b'$-net in $\R^{d}$.
	It follows that for $\hl{b:=2b'}$ each set $\cl{X}_{k}$ is an $s/l_{k}$-separated $b/l_{k}$-net in $\cl{B}$.
Fix $x\in \overline{B}$ and choose $x_{k}\in \overline{X}_{k}$ such that $\lnorm{2}{x_{k}-x}\leq \frac{b}{l_{k}}$. Then for all sufficiently large $k$
		\begin{linenomath}		
		\begin{align*}
		\lnorm{2}{h\br*{f(x)}-x}\leq& \lnorm{2}{h\br*{ f(x)}-h\br*{f(x_{k})}}+\lnorm{2}{h\br*{f(x_{k}})-\cl{h}_{k}\br*{\cl{f}_{k}(x_{k})}}\\
		&+\lnorm{2}{x_{k}-x}\\
		\leq& L\omega\left(U\omega\left(\frac{b}{l_{k}}\right)\right)+\lnorm{\infty}{h\circ f-\cl{h}_{k}\circ \cl{f}_{k}}+\frac{b}{l_{k}}.
		\end{align*}
		\end{linenomath}
Since the right-hand side converges to zero as $k$ goes to infinity, we verify that $h(f(x))=x$.

	We now prove \eqref{eq:pushforward_2}. For $k\geq 1$, define a measure $\mu_{k}$ on $I^{d}$ by
\begin{linenomath}
\begin{equation*}
\mu_{k}(A):=\frac{1}{l_{k}^{d}}\left|A\cap X_{k}\right| \qquad \forall A\subseteq I^{d},
\end{equation*}
\end{linenomath}
\begin{claim}
The measure $\mu_k$ converges weakly to $\rho\rest{\leb}{I^d}$.
\end{claim}
\begin{proof}
This follows immediately from Lemma~\ref{lemma:weakconvcrit}. Note that the required collection $\mc{Q}_k$ can be defined as $\set{\phi_k^{-1}(T_{k,i})\colon i\in[m_k^d]}$. Then $\diam\br*{\phi_k^{-1}\br*{T_{k,i}}}\leq \hl{\sqrt{d}/m_k}\to 0$ as $k\to \infty$. Moreover, we have that
\begin{linenomath}
$$
\mu_k\left(\phi_k^{-1}(T_{k,i})\right)=\frac{n_{k,i}^d}{l_{k}^d}\leq  \frac{1}{l_{k}^d} \lint{T_{k,i}}{\rho_k}=\rho\leb\br*{\phi_{k}^{-1}(T_{k,i})},
$$
\end{linenomath}
and similarly, using the Binomial theorem,
\begin{linenomath}
\begin{align*}
\mu_k\left(\phi_k^{-1}(T_{k,i})\right)&\geq \frac{1}{l_{k}^d} \br*{\sqrt[d]{\lint{T_{k,i}}{\rho_k}}-1}^d\\
&\geq \rho\leb\br*{\phi_{k}^{-1}(T_{k,i})}-\frac{2^d\sup\rho^{\frac{d-1}{d}}}{l_k m_k^{d-1}},
\end{align*}
\end{linenomath}
for all $k$ large enough.
Since $\frac{m_k}{l_k}\to 0$, this proves that
\begin{linenomath}
$$
\abs{\mu_k\left(\phi_k^{-1}(T_{k,i})\right)-\rho\leb\left(\phi_k^{-1}(T_{k,i})\right)}\in o\br*{\frac{1}{m_k^d}}.
$$
\end{linenomath}
\end{proof}

The claim above also implies, by Lemma~\ref{lemma:weakconvcrit2}, that $(\rest{\cl{f}_k}{I^d})_\sharp \mu_k$ converges weakly to $(\rest{f}{I^d})_\sharp \rho\leb$, since $\cl{f}_{k}$ converges uniformly to $f$.

It remains to show that $(\rest{\cl{f}_k}{I^d})_\sharp \mu_{k}$ converges weakly to $\rest{\leb}{f(I^{d})}$. To this end we compare $(\rest{\cl{f}_k}{I^d})_\sharp \mu_{k}$ with the standard normalised counting measure on $\frac{1}{l_{k}}\Z^{d}$
\begin{linenomath}
\begin{equation*}
\nu_{k}(A):=\frac{1}{l_{k}^{d}}\abs{A\cap \frac{1}{l_{k}}\Z^{d}},\qquad A\subseteq \R^{d},
\end{equation*}
\end{linenomath}
which clearly converges weakly to the Lebesgue measure. For a given continuous function $\varphi\colon\R^{d}\to\R$ with compact support we need to verify
\begin{linenomath}
\begin{equation}\label{eq:weakconv}
\abs{\niceint{f(I^d)}{\varphi}{\nu_{k}}-\niceint{\cl{f}_k(I^d)}{\varphi}{(\rest{\cl{f}_k}{I^d})_\sharp \mu_{k}}}\stackrel[\hl{k\to\infty}]{}{\longrightarrow}0.
\end{equation}
\end{linenomath}

We bound the expression in \eqref{eq:weakconv} above by the sum of two terms:
\begin{linenomath}
\begin{align}\label{eq:weakconv_split}
\begin{split}
\abs{\niceint{f(I^d)}{\varphi}{\nu_{k}}-\niceint{\cl{f}_k(I^d)}{\varphi}{\nu_{k}}}+\abs{\niceint{\cl{f}_k(I^d)}{\varphi}{\nu_{k}}-\niceint{\cl{f}_k(I^d)}{\varphi}{(\rest{\cl{f}_k}{I^d})_\sharp \mu_{k}}}
\end{split}
\end{align}
\end{linenomath}
The first term is at most $\lnorm{\infty}{\varphi}\nu_{k}(f(I^{d})\Delta \cl{f}_{k}(I^{d}))$, which vanishes as $k\to\infty$ due to Lemma~\ref{lem:close_fun_close_img}, Corollary~\ref{cor:Hausdorff}, the weak convergence of $\nu_{k}$ to $\leb$ \hl{and the fact that $f$ is bi-Hölder \eqref{eq:omega_is_Hoelder}}. The second term may be bounded above by
\begin{linenomath}
\begin{equation}\label{eq:second_term_bound}
\frac{\lnorm{\infty}{\varphi}}{l_{k}^{d}}\abs{A_{k}},\qquad \text{where $A_{k}:=\cl{f}_{k}(I^{d})\cap \frac{1}{l_{k}}\Z^{d}\setminus f_{k}(X_{k})$.}
\end{equation}
\end{linenomath}
We will argue that
\begin{linenomath}
\begin{equation}\label{eq:Ak_inclusion}
A_{k}\subseteq \hl{\cl{\normalcolor B}}\left(\partial f(I^{d}),\lnorm{\infty}{\cl{f}_{k}-f}\right)
\end{equation}
\end{linenomath}
for all $k$ sufficiently large. Once this is established the quantity of \eqref{eq:second_term_bound} is seen to be at most
\begin{linenomath}
\begin{equation*}
\lnorm{\infty}{\varphi}\leb\br*{\hl{\cl{\normalcolor B}}\br*{\partial f(I^{d}),\lnorm{\infty}{\cl{f}_{k}-f}+\frac{\sqrt{d}}{l_{k}}}},
\end{equation*}
\end{linenomath}
which converges to zero as $k\to\infty$ by Corollary~\ref{cor:Hausdorff}. Hence, to complete the verification of the weak convergence of $(\rest{\cl{f}_{k}}{I^d})_{\sharp}\mu_{k}$ to $\leb|_{f(I^{d})}$, we prove \eqref{eq:Ak_inclusion}.

From now on we treat $k$ as fixed but sufficiently large, \jl{and use that expressions involving $\omega$, like $\omega\left(\frac{b}{l_{k}}\right)$, are well defined for all sufficiently large $k$.}
Since $f_{k}$ is a bijection $\phi_{k}^{-1}(X)\to \frac{1}{l_{k}}\Z^{d}$ and $X_{k}=\phi_{k}^{-1}(X)\cap I^{d}$, any point in $A_{k}$ has the form $f_{k}(x)$ for some $x\in \phi^{-1}_{k}(X)\setminus I^{d}$.
If $f_{k}(x)\notin f(I^{d})$ then $f_{k}(x)\in \hl{A_k\setminus f(I^d)}\subseteq \cl{f}_{k}(I^{d})\setminus f(I^{d})$, and therefore, $\dist(f_{k}(x),\partial f(I^{d}))\leq\lnorm{\infty}{\cl{f}_{k}-f}$.

In the remaining case we have $f_{k}(x)=f(y)$ for some $y\in I^{d}$. By the definition of $A_k$, there must also be $v\in I^d$ such that $f_k(x)=\cl{f}_k(v)$.
Because $\cl{X}_k$ is a $b/l_k$-net of $\cl{B}\supset\hl{\intr\cl{B}}\supset I^d$ and $k$ is large enough, there is $v'\in \cl{X}_k\cap \cl{B}(I^d,b/l_k)$ such that
$\lnorm{2}{f_k(x)-f_k(v')}=\lnorm{2}{\cl{f}_k(v)-\cl{f}_k(v')}\leq U\omega\br*{b/l_k}$ thanks to \hl{\eqref{eq:lip_omega_f_k}}.
\hl{Since $f_k(x),f_k(v')\in f_k(\cl{B})\cap\frac{1}{l_k}\Z^d$, \eqref{eq:u_norm_bound} and \eqref{eq:f_k_inverse_mod} imply that}
\begin{linenomath}
\begin{equation*}
\hl{
\enorm{x-v'}=\enorm{f_k^{-1}(f_k(x))-f_k^{-1}(f_k(v'))}\leq C\omega\br*{\enorm{f_k(x)-f_k(v')}}
}
\end{equation*}
\end{linenomath}
\hl{
for some constant $C>1$ independent of $k$. Combining it with the upper bound $U\omega(b/l_k)$ on $\enorm{f_k(x)-f_k(v')}$  we derive 
}
\begin{equation}\label{eq:dist_to_bdry}
\hl{
\enorm{x-v'}\leq C\omega\br*{U\omega\br*{\frac{b}{l_k}}},
}
\end{equation}
which goes to zero as $k$ goes to infinity.

To show that $\dist(f_{k}(x),\partial f(I^{d}))\leq\lnorm{\infty}{\cl{f}_{k}-f}$, we may assume that the upper bound of \eqref{eq:dist_to_bdry} is smaller than $\cl{r}-\sqrt{d}-\frac{b}{l_k}$ as $k$ is sufficiently large; \jl{recall that $\cl{r}>\sqrt{d}$ because $I^{d}\subseteq \operatorname{int}\cl{B}= B(\mb{0},\cl{r})$.}
Since $v'\in \cl{B}\br*{I^d, \frac{b}{l_k}}\subseteq\hl{\cl{B}\br*{\mb{0},\sqrt{d}+\frac{b}{l_k}}}$,
\hl{we have}
$x\in \cl{B}$\hl{, and thus,} $f$ is defined at $x$.
Note that $f(x)\notin f(I^{d})$, because $x\notin I^{d}$ and $f\colon\cl{B}\hookrightarrow\R^{d}$ is an \hl{injection}.
Hence\hl{, as $f_k(x)=f(y)\in f(I^d)$,}
\begin{linenomath}
\begin{equation*}
\dist(f_{k}(x),\partial f(I^{d}))\leq\lnorm{2}{f_{k}(x)-f(x)}\leq \lnorm{\infty}{\cl{f}_{k}-f},
\end{equation*}
\end{linenomath}
as required.
\end{proof}

It now only remains to prove Theorem~\ref{thm:non_rlz_dsty}. We provide an argument based on the following geometric statement, proved in Section~\ref{section:geometric}. 
Constructions of non-realisable densities based on statements of this type have already been written in great detail in \cite{DKK2018} and originally in \cite{BK1}. Therefore, following Lemma~\ref{lemma:geometric_amalg} we only give an informal sketch of the proof of Theorem~\ref{thm:non_rlz_dsty}.
\begin{lemma}\label{lemma:geometric_amalg}
Let $d\geq 2$. Then there is $\alpha_0=\alpha_0(d)\hl{\in(0,\infty)}$ such that for $\omega\in\M$ of the form
		\begin{linenomath}
		\begin{equation*}
		\omega(t)\leq t\left(\log\frac{1}{t}\right)^{\alpha_0},\qquad \text{for all $t\in(0,a_{\omega})$},
		\end{equation*}
		\end{linenomath}
		the following statement holds:
		
		Let $k\in\N$, $c\in(0,a_{\omega})$, $\varepsilon\in(0,1)$ and $L\geq 1$. Then there exists $r=r(d,L\sqrt{k}\omega,\varepsilon,c)\in\N$ such that for every open ball $U\subseteq\R^{d}$ of radius at least $2c\sqrt{d}$ there exist finite tiled families $\Sq_{1},\Sq_{2},\ldots,\Sq_{r}$ of cubes contained in $U$ with the following properties:
		\begin{enumerate}
			\item\label{lemma:geometric1_amalg} For each $1\leq i<r$ and each cube $S\in\Sq_{i}$
			\begin{linenomath}
				\begin{equation*}
				\leb \Biggl(S\cap\bigcup_{j=i+1}^{r}\bigcup\Sq_{j}\Biggl)\leq \poly{\varepsilon}\leb (S). 
				\end{equation*}
			\end{linenomath}
			\item\label{lemma:geometric2_amalg} For any $k$-tuple $(h_{1},\ldots,h_{k})$ of bi-$\omega$-mappings $h_{j}\colon U\to\R^{d}$ for which $\max \bilipomeg(h_{j})\leq L$ there exist $i\in[r]$ and $\mb{e}_{1}$-adjacent cubes $S,S'\in \Sq_{i}$ such that
			\begin{linenomath}
				\begin{equation}\label{eq:volume_difference}
				\frac{\abs{\leb(h_j(S))-\leb(h_j(S'))}}{\leb(S)}\leq\kappa(\varepsilon)
				\end{equation}
			\end{linenomath}
			for all $j\in[k]$, where $\lim_{\varepsilon \to 0}\kappa(\varepsilon)=0$.
		\end{enumerate}
\end{lemma}
We note that the upper bounds in Statements~\ref{lemma:geometric1_amalg} and \ref{lemma:geometric2_amalg} depend on $d, \omega, k, L$ and $c$.

\begin{remark*}[The role of the parameter $k$ in Lemma~\ref{lemma:geometric_amalg}]
We will only require Lemma~\ref{lemma:geometric_amalg} for the case $k=1$, that is, we only apply it to single bi-$\omega$ \jl{mapping} and not to $k$-tuples of bi-$\omega$ \jl{mappings}.
However, the work \cite{DKK2018} shows that such statements for $k$-tuples can be very useful and so we prove Lemma~\ref{lemma:geometric_amalg} for general $k$ in case it finds future applications.
\end{remark*}

\begin{proof}[Proof of Theorem~\ref{thm:non_rlz_dsty}]
\hl{For $L>1$} let $\mc{G}_L$ be the set of those continuous functions $\rho\colon I^d\to\R$ for which the pushforward equation \eqref{eq:pushforward} admits a bi-$\omega$ solution $f\colon I^d \to \R^d$ with $\bilipomeg(f)\leq L$.
We want to argue that there is a constant $M=M(L)>1$ such that for any $\rho\in \mc{G}_L$ and every $\xi>0$ there is $\tilde{\rho}\in C(I^d)$ with $\lnorm{\infty}{\rho-\tilde{\rho}}\leq \xi$ and such that the ball $B(\tilde{\rho},\xi/M)$ in the space $C(I^d)$ is disjoint from $\mc{G}_L$.

We will describe the argument here only informally, since the argument of \cite[Thm.~4.8]{DKK2018} could be used here essentially without a~change, only replacing the use of \cite[Lem.~3.1]{DKK2018} with its stronger form Lemma~\ref{lemma:geometric_amalg} presented above and making the construction continuous as in \cite[Lem.~4.6]{DKK2018}.

\hl{Fix $\rho\in \mc{G}_{L}$ and $\xi>0$. Then} for every sequence of tiled families $\mc{S}_1, \ldots, \mc{S}_r$ as in statement~\ref{lemma:geometric1_amalg} of Lemma~\ref{lemma:geometric_amalg}, there is a continuous function $\psi\colon I^d\to\R$ with $\lnorm{\infty}{\psi}\leq\xi$ with the following properties:
\begin{enumerate}
	\item[(1)] $\psi(x)=0$ for every $x\in I^d\setminus \bigcup_{i=1}^r\bigcup\mc{S}_i$,
	\item[(\mylabel{it:chessboard}{2})] for every $i\in[r]$ and every $\mb{e}_1$-adjacent $S,S'\in\mc{S}_i$
	 \begin{linenomath}
	 $$
	 \frac{1}{\leb(S)}\abs{\lint{S}{\psi}-\lint{S'}{\psi}}\geq \xi.
	 $$
	 \end{linenomath}
\end{enumerate}
This is easy to see: \jl{s}tart by defining a chessboard function with values $\pm \xi$ on the tiled family $\Sq_{1}$. Then modify this function on the tiled family $\Sq_{2}$, creating a chessboard pattern of $\pm \xi$ values there and repeat for the remaining tiled families $\Sq_{3},\ldots,\Sq_{r}$. Call the final function $\psi$. Provided that $\varepsilon$ is chosen small enough relative to $\xi$, statement \ref{lemma:geometric1_amalg} of Lemma~\ref{lemma:geometric_amalg} ensures that for every $i\in[r]$ the $\pm \xi$ values of $\psi$ on the cubes in $\Sq_{j}$ for $j>i$ have negligible impact on the average value of $\psi$ on the much larger cubes from $\Sq_{i}$. Thus, the final function $\psi$ satisfies \eqref{it:chessboard}. Continuity of $\psi$ is taken care of by smoothing in a small enough neighbourhood of the boundaries of the cubes in each step; see \cite[Lem.~4.6]{DKK2018}.

Now applying Lemma~\ref{lemma:geometric_amalg} with $U$ small enough, $0<c<\diam(U)/(4\sqrt{d})$,
 $L, k=1, \omega$ and $\varepsilon>0$ small enough so that $\kappa(\varepsilon)$ becomes smaller than, say, $\xi/4$, one gets $r\in\N$ and tiled families $\mc{S}_1,\ldots, \mc{S}_r$ contained in $U$. Applying the construction sketched above to these tiled families $\mc{S}_1,\ldots, \mc{S}_r$ and the given $\xi$, we get $\psi$ and define the desired $\tilde{\rho}$ as $\tilde{\rho}:=\rho+\psi$.

Choosing $U$ small enough, $\rho$ is almost constant on $U$, and thus, any continuous function $\phi$ with $\lnorm{\infty}{\tilde{\rho}-\phi}\leq \xi/M$ must follow essentially the same chessboard pattern as $\psi$ in property~\eqref{it:chessboard}, just with $\xi$ replaced with $\xi/2$ (for $M$ large enough and $U$ small enough).
However, this excludes $\phi$ from the set $\mc{G}_{L}$. Indeed, the existence of a bi-$\omega$ mapping $h$ witnessing $\phi\in\mc{G}_{L}$ forces property~\eqref{it:chessboard} of the function $\phi$ to become incompatible with statement~\ref{lemma:geometric2_amalg} from Lemma~\ref{lemma:geometric_amalg}. (To see this, note that whenever there is a bi-$\omega$ mapping $h\colon I^d\to\R^d$ satisfying $h_{\#}\phi\leb=\rest{\leb}{h(I^d)}$, we also have that $\leb(h(A))=\lint{A}{\phi}$ for any measurable $A\subset I^d$, i.e., also for $A=S$ and $A=S'$, where $S,S'$ are the cubes from property \eqref{it:chessboard} or from statement~\ref{lemma:geometric2_amalg} in Lemma~\ref{lemma:geometric_amalg}.)

The proof for the space $L^\infty(I^d)$ follows a similar pattern as sketched above, with a slightly different method to create the chessboard pattern in a $\xi$-neighbourhood of $\rho$ (this is described in \cite[Lem.~4.9]{DKK2018}). The proof of \cite[Thm.~4.8]{DKK2018} applies almost literally in this case.
\end{proof}

\subsection{\texorpdfstring{Displacement equivalence}{Displacement equivalence}}\label{subs:displacement}

The aim of the present subsection is to prove Propositions~\ref{prop:first_disp} and \ref{p:displacement}.

\begin{proof}[Proof of Proposition~\ref{prop:first_disp}]
	Instead of Proposition~\ref{prop:first_disp}, we prove the following more general statement: 
	\begin{changemargin}{1.5cm}{1.5cm}
	\medskip
	\itshape
	Let $\omega\in\mc{M}$ and $X\subseteq \R^{d}$ be an \jl{$\omega$-irregular} separated net. Let $f\colon X\to \Z^{d}$ be a bijection. Then
	\begin{equation}\label{eq:displacement_grows_gen}
	\lim_{R\to\infty}\sup_{x\in B(\jl{\mb{0}},R)\cap X}\frac{\lnorm{2}{f(x)-x}}{R\omega\left(\frac{1}{R}\right)}=\infty.	
	\end{equation}
	\medskip
	\end{changemargin}
	
Let s stand for the \jl{minimum of the} separation constant\jl{s} of $X$ \jl{and $\Z^{d}$}. \jl{The $\omega$-irregularity} of $X$ implies that for $g=f$ or $g=f^{-1}$ it holds that
\begin{equation}\label{eq:limsup_gen}
\limsup_{R\to\infty}\sup_{x\neq y\in B(\mb{0},R)\cap\dom(g)}\frac{\enorm{g(y)-g(x)}}{R\omega\br*{\frac{\enorm{y-x}}{R}}}=\infty.
\end{equation}
We use \eqref{eq:limsup_gen} and the bound
\[
\frac{\enorm{g(y)-g(x)}}{R\omega\br*{\frac{\enorm{y-x}}{R}}}\leq 
\frac{\enorm{g(y)-y}+\enorm{g(x)-x}}{R\omega\br*{\frac{s}{R}}}+
\frac{\frac{\enorm{y-x}}{R}}{\omega\br*{\frac{\enorm{y-x}}{R}}}.
\] 	
This and Definition~\ref{def:M} imply that the supremum from \eqref{eq:limsup_gen} is bounded above by
\[
\sup_{x\in B(\jl{\mb{0}},R)\cap\dom(g)}\frac{2\enorm{g(x)-x}}{R\omega\br*{\frac{s}{R}}}+1.
\]
\sug{This, together with} \eqref{eq:limsup_gen} \sug{and the concavity of $\omega$}
implies also that for $g=f$ or $g=f^{-1}$ we have
\begin{equation}\label{eq:g_disp_omega_growth}
\limsup_{R\to\infty}\sup_{x\in B(\mb{0},R)\cap\dom(g)}\frac{\enorm{g(x)-x}}{R\omega\br*{\frac{1}{R}}}=\infty.
\end{equation}
Finally, \eqref{eq:g_disp_omega_growth} for $g=f$ 
is just \eqref{eq:displacement_grows_gen}.
So we may assume that \eqref{eq:g_disp_omega_growth} holds for $g=f^{-1}$. For contradiction, assume that \eqref{eq:displacement_grows_gen} does not hold. This means that there is $K>0$ such that for every $x\in X$
\[
\enorm{f(x)-x}\leq K\enorm{x}\omega\br*{\frac{1}{\enorm{x}}}.
\]
Hence, there is $R_0>0$ such that if $\norm{x}\geq R_0$, then $\enorm{f(x)}\geq\enorm{x}/2$. Rewriting this inequality in terms of $f^{-1}$ shows that there is $C>1$ such that for every non-zero $z\in \Z^d$ it holds that $\enorm{f^{-1}(z)}\leq C\enorm{z}$.

Now \eqref{eq:g_disp_omega_growth} for $g=f^{-1}$ yields
\[
\limsup_{R\to\infty}\sup_{z\in B(\mb{0},R)\cap \Z^d}\frac{\enorm{f^{-1}(z)-z}}{R\omega\br*{\frac{1}{R}}}\leq\limsup_{R\to\infty}\sup_{x\in B(\mb{0},CR)\cap X}\frac{\enorm{f(x)-x}}{R\omega\br*{\frac{1}{R}}}=\infty.
\]
The last equality  
is equivalent to \eqref{eq:displacement_grows_gen}; a contradiction.
\end{proof}

As we noted in the introduction, the key part of the proof of Proposition~\ref{p:displacement} is the following lemma.

\begin{lemma}\label{lemma:displacement}
	Let $H$ be a half-space in $\R^d$ with the boundary hyperplane containing $\mb{0}$ and $S$ be a set contained in $H$.
	Additionally, assume that $X\subset S$ is a separated net of $S$.
	Then $X$ can be extended to a separated net $Y$ of $\R^d$ such that $Y$ has a well-defined natural density and $Y\cap S=X$. 
\end{lemma}

\begin{proof}
\jl{We first extend $X$ to a separated net of $H$ arbitrarily (adding only points inside $H\setminus S$). To simplify the notation, we will denote this extension by $X$, too. We also write $s$ for the separation constant of the (extended) set $X$.}
We write
$B_i:=\cl{B}(\mb{0},i)$ and $A_i:=B_i\setminus B_{i-1}$ for $i\in\N$ (we also set $B_0:=\emptyset$). Moreover, we define
\[
\delta_i:=\frac{\abs{X\cap A_i}}{\leb\br*{\jl{H}\cap A_i}}\jl{=\frac{2\abs{X\cap A_i}}{\leb(A_i)}}
\]
 for every $i\in\N$ and $\delta:=\sup_{i\in\N}\delta_i$.
Since $X$ is $s$-separated, $\delta<\infty$.
We will construct a separated net $Y\supset X$ with natural density $\delta$ as $Y:=\bigcup_{i=0}^\infty Y_i$, where $Y_0\subset Y_1\subset\cdots\subset Y_i\subset\cdots$. We set $Y_0:=X$.
We define a \jl{sequence} $(n_i)_{i=1}^\infty$ of non-negative integers where each $n_i$ is chosen as the unique number satisfying
\begin{linenomath}
\begin{equation*}
\delta\leb\br*{B_i}-\abs{X\cap B_i}-\sum_{j=1}^{i-1}n_{j}\geq n_i> \delta\leb\br*{B_i}-\abs{X\cap B_i}-\sum_{j=1}^{i-1}n_{j}-1.
\end{equation*}
\end{linenomath}
Next for every $i\in\N$ we place $n_i$ points inside $A_i\setminus \jl{H}$ (the exact position will be determined later). The set $Y_i$ is then formed by \jl{the union of} $Y_{i-1}$ \jl{and} the $n_i$ points \jl{inside} $A_i\setminus \jl{H}$.
Consequently, for every $i\in\N$ we have
\begin{equation}\label{eq:n_i_in_Y_i_rewritten}
\abs{Y_i\cap B_i}=\abs{Y_{i-1}\cap B_i}+n_i=\abs{Y_{i-1}\cap B_{i-1}}+\abs{A_i\cap X}+n_i
\end{equation}
and
\begin{equation}\label{eq:n_i_in_Y_i}
\delta\leb\br*{B_i}-\abs{Y_{i-1}\cap B_i}\geq n_i> \delta\leb\br*{B_i}-\abs{Y_{i-1}\cap B_i}-1.
\end{equation}

The definition of $\delta_i$ together with \eqref{eq:n_i_in_Y_i} and \eqref{eq:n_i_in_Y_i_rewritten} for all $i\in\N$ yield
\begin{linenomath}	
\begin{align*}
n_i&\leq\delta\leb\br*{B_i}-\delta\leb\br*{B_{i-1}}+1-\delta_i\leb\br*{A_i\cap \jl{H}}\leq \delta\leb\br*{A_i}+1,\\
n_i&\geq\delta\leb\br*{B_i}-\delta\leb\br*{B_{i-1}}-\delta_i\leb\br*{A_i\cap \jl{H}}-1\geq \frac{\delta}{2}\leb\br*{A_i}-1.
\end{align*}
\end{linenomath}
Altogether, we get that
\begin{equation}\label{eq:n_i_asymptotics}
n_i\in\Theta(i^d-(i-1)^d)=\Theta(i^{d-1}).
\end{equation}
Now we specify more precisely how to choose the position of the points in $Y_i\setminus Y_{i-1}$. 
For \jl{the finitely many} $i\in\N$ such that $\partial B_i\setminus B(\jl{H},s)=\emptyset$ we simply place the corresponding $n_i$ points arbitrarily inside $A_i\setminus \jl{H}$\jl{, as this does not affect whether $Y$ is a separated net or not}.
For the remaining indices $i\in\N$ we place the $n_i$ points inside $\partial B_i\setminus B(\jl{H},s)$.
Let $s_i$ denote the separation of $Y_i\setminus Y_{i-1}$.

We will show that the position of the points can be chosen so that $s_i\in\Theta(1)$ and so that $Y_i\setminus Y_{i-1}$ is a $\Theta(1)$-net of $\partial B_i\setminus B(\jl{H},s)$.

\newcommand{\srf}{\zeta^{d-1}}
\begin{claim}\label{claim:sep_net_on_sphere}
Let $A\subseteq\partial B(\mb{0},1)$ be a compact set of positive surface measure. Then, \jl{for every $n\in\N$,} there exist $\sigma(n)\in\Theta(n^{-\frac{1}{d-1}})$ and a set of $n$ points in $A$ such that its separation constant is $\sigma(n)$ and it is a $\sigma(n)$-net of $A$. 
\end{claim}
\begin{proof}
Let $\sigma(n)$ be the maximum separation for a set of $n$ points in $A$.
Let $Z:=\set{z_1,\ldots, z_n}\subset A$ be a set with separation $\sigma(n)$. Then
\[
n\leb(B(\mb{0},\sigma(n)/2))\leq\leb(B(\mb{0},1+\sigma(n)/2))-\leb(B(\mb{0},1-\sigma(n)/2)),
\]
which implies that
\[
\sigma(n)^d\in O\br*{\frac{\sigma(n)}{n}}.
\]
Thus, $\sigma(n)\in O\br*{n^{-\frac{1}{d-1}}}$.

Let $d_i:=\min_{j\neq i}\lnorm{2}{z_i-z_j}$. From now on we will additionally assume that $Z$ is chosen so that it minimises the number of $i\in[n]$ such that $d_i=\sigma(n)$. We will show that $Z$ is a $\sigma(n)$-net of $A$. Assume not. Then there is $y\in A$ such that $\lnorm{2}{y-z_i}>\sigma(n)$ for every $i\in[n]$. Take $i_0\in[n]$ such that $d_{i_0}=\sigma(n)$. Redefining $z_{i_0}$ to $y$ yields either a set that is more than $\sigma(n)$-separated, or one which is still $\sigma(n)$-separated, but the number of indices $i$ with $d_i=\sigma(n)$ has decreased. In both cases we get a contradiction.

Let $\srf$ denote the surface measure on $\partial B(\mb{0},1)$. The above implies that
\[
\sum_{i=1}^n\srf\br*{\jl{\cl{\normalcolor B}}(z_i,\sigma(n))\cap \partial B(\mb{0},1)}\geq\srf(A).
\]
Thus, we get that $\sigma(n)^{d-1}\in \Omega(1/n)$, which finishes the claim.
\end{proof}
We apply the claim above to a copy of $\partial B_i\setminus B(\jl{H},s)$ scaled down by the factor $i$ and then scale it back.
Since $s$ is a constant, we see that $\partial B_i \cap B(H,s)$ occupies asymptotically at most a constant fraction of the measure of $\partial B_i$.
The claim above then implies that the points in $Y_i\setminus Y_{i-1}$ can be chosen so that
\begin{equation}\label{eq:separation_on_sphere}
s_i\in\Theta\br*{i\cdot n_i^{-\frac{1}{d-1}}}.
\end{equation}

Thus, \eqref{eq:n_i_asymptotics} implies that $s_i\in\Theta(1)$. Consequently, $Y_i$ is $\Theta(1)$-separated.
Because the position of the new points in $Y_i$ was chosen inside $\partial B_i\setminus B(\jl{H},s)$
according to Claim~\ref{claim:sep_net_on_sphere}, $Y_i\setminus Y_{i-1}$ is an $s_i$-net of $\partial B_i\setminus B(\jl{H},s)$ for all $i\in\N$ large enough. Therefore, it is a $\Theta(1)$-net of $A_i\setminus \jl{H}$.
Altogether, we get that $Y$ is a separated net of $\R^d$. 

It remains to show that $Y$ has natural density $\delta$. Let $(r_l)_{l\in\N}$ be a sequence \jl{of positive real numbers} going to infinity. For every $l\in\N$ we set $i_l:=\floor{r_l}$. We get that
\begin{linenomath}
\begin{align}
\lim_{l\to\infty}&\frac{\abs{Y\cap\cl{B}(\mb{0},r_l)}}{\leb\br*{\cl{B}(\mb{0},r_l)}}=\nonumber\\
\lim_{l\to\infty}&\frac{\abs{Y\cap\cl{B}(\mb{0},i_l)}}{\leb\br*{\cl{B}(\mb{0},i_l)}}\cdot\lim_{l\to\infty}\frac{\leb\br*{\cl{B}(\mb{0},i_l)}}{\leb\br*{\cl{B}(\mb{0},r_l)}}+\lim_{l\to\infty}&\frac{\abs{Y\cap\br*{\cl{B}(\mb{0},r_l)\setminus \cl{B}(\mb{0},i_l)}}}{\leb\br*{\cl{B}(\mb{0},r_l)}}. \label{eq:three_limits}
\end{align}
\end{linenomath}
The first limit \jl{of \eqref{eq:three_limits}} is $\delta$ by \jl{\eqref{eq:n_i_in_Y_i_rewritten} and} \eqref{eq:n_i_in_Y_i}.
The second limit is obviously equal to $1$.
The term in the third limit is bounded above by a constant multiple (depending on \jl{$d$ and} the separation of $Y$) of the expression
\[
\frac{\leb\br*{\cl{B}(\mb{0},r_l)}-\leb\br*{\cl{B}(\mb{0},r_l-1)}}{\leb\br*{\cl{B}(\mb{0},r_l)}}=\frac{r_l^d-(r_l-1)^d}{r_l^d},
\]
which converges to zero.
\end{proof}

\begin{proof}[Proof of Proposition~\ref{p:displacement}]
\jl{First, we observe that for any two separated nets $Y,Z$ in $\R^d$, a modulus of continuity $\omega$, a homogeneous $\omega$-mapping $f\colon Y\to Z$ and two constants $c, c'>0$ the mapping $\tilde{f}\colon cY\to c'Z$ defined as $\tilde{f}(y):=c'f(y/c)$ for every $y\in cY$ is also a homogeneous $\omega$-mapping; this follows immediately from Definition~\ref{d:hom_omega_map}. This implies that $Y$ is $\omega$-regular if and only if $cY$ is $\omega$-regular for any fixed $c>0$.}
		\jl{Thus, to prove Proposition~\ref{p:displacement},}
		it is enough to verify the existence of a\jl{n $\omega$-irregular} separated net with well-defined natural density (not necessarily equal to one) \jl{for $\omega$ given by Theorem~\ref{thm:main_result_weak}}.
		For this, it suffices to establish that the `every separated net' assumption of Lemma~\ref{lemma:discrete_to_cts} can be weakened to `every separated net with well-defined natural density'.
		The desired statement then follows from this stronger version of Lemma~\ref{lemma:discrete_to_cts} in combination with Theorem~\ref{thm:non_rlz_dsty}.
		
		To prove the aforementioned strengthening of Lemma~\ref{lemma:discrete_to_cts}, we only make more specific two choices left unspecified in the proof of Lemma~\ref{lemma:discrete_to_cts}: \jl{f}irstly, we choose the \jl{cubes} $S_{k}$ so that they are all contained in a half-space \jl{$H$} whose boundary contains $\jl{\mb{0}}$, and secondly, at the sentence labelled \eqref{sentence_extend_net}, we specifically extend the net according to Lemma~\ref{lemma:displacement} outside $\bigcup_{k=1}^\infty S_k$.
		The resulting separated net $X$ then has well-defined natural density, allowing us to apply the weakened hypothesis of Lemma~\ref{lemma:discrete_to_cts} to it.
	\end{proof}
	\section{Geometric properties of homeomorphisms of prescribed modulus of continuity.}\label{section:geometric}
The present section is an extensive refinement of \cite[Sec.~3]{DKK2018} and is based on a construction of Burago and Kleiner in \cite{BK1}. Lemma~\ref{lemma:M1} and Subsection~\ref{subsec:largeness} are entirely new; the remaining proofs follow the structure of their analogues in \cite{DKK2018}. The present construction is dependent on many parameters, whose precise or even asymptotic values were mostly irrelevant in \cite{DKK2018}. On the other hand, in this work it is crucial to analyse the dependence between various parameters; this is the reason why we have to present the constructions here in full detail and cannot only refer to \cite[Sec.~3]{DKK2018}. \arxiv Inside some of the proofs of the present section, parts of the arguments of \cite{DKK2018} transfer without any change. Although it would be possible to refer the reader at these places to the relevant parts of \cite{DKK2018}, we will include these passages here with references for the reader's convenience.
\fi\normalcolor

\paragraph{Notation.} For mappings $h\colon \R^{d}\to\R^{n}$ we denote by $h^{(1)},\ldots,h^{(n)}$ the co-ordinate functions of $h$. For a cube $S\subset\R^d$ we write $\ell(S)$ for its sidelength. In the present section we will often encounter rather complicated calculations. To simplify the expressions, we will use the letter $\Lambda$ for a general purpose constant, that is we write $\Lambda(x,y,z)$, if $\Lambda$ is a positive and finite constant depending only on $x$, $y$ and $z$ whose precise value is irrelevant.
In particular, we allow the value of $\Lambda(x,y,z)$ to change in each occurrence. Often, the parameters $x,y,z$ determining $\Lambda$ will be suppressed.

The objective of this section is to prove Lemma~\ref{lemma:geometric_amalg}, which provided the basis for the proof of the main results in the previous section.
\subsection{A Dichotomy.}
We begin by proving a dichotomy statement for bi-$\omega$-mappings. The dichotomy will be established first in dimension $d=1$ and then extended to higher dimensions by induction.
\begin{restatable}{lemma}{lemmaonedimdich}\label{lemma:1dimdich}
	Let $\omega\in\mc{M}$, \hl{$c,\varepsilon\in (0,a_{\omega})$} and $N\in\N$ with $N\geq 2$. Moreover, let $n\in\N$ and $h\colon [0,c]\to\R^{n}$ be an $\omega$-mapping with $\mathfrak{L}_{\omega}(h)\leq 1$. Then for any values of the parameters $\varphi$ and $M\in\N$ such that
	\begin{linenomath}
	\begin{equation}\label{eq:vph_M}
0\leq\varphi\leq\frac{\varepsilon^{3}}{120} \quad \text{\hl{and}}\quad M\geq\frac{1}{\omega^{-1}\left(\frac{\varepsilon}{4}\right)},
	\end{equation}
	\end{linenomath}
	at least one of the following statements holds: 
	\begin{enumerate}
		\item\label{1dimdich1} There exists a set $\Omega\subset[N-1]$ with $\left|\Omega\right|\geq (1-\varepsilon)(N-1)$ such that for all $i\in\Omega$ and for all $x\in\left[\frac{(i-1)c}{N},\frac{ic}{N}\right]$
		\begin{linenomath}
		\begin{equation*}
		\left\|h\left(x+\frac{c}{N}\right)-h(x)-\frac{1}{N}(h(c)-h(0))\right\|_{2}\leq\varepsilon \omega\left(\frac{c}{N}\right).
		\end{equation*}
		\end{linenomath}
		\item\label{1dimdich2} There exists $z\in\frac{c}{NM}\Z\cap[0,c-\frac{c}{NM}]$ such that
		\begin{linenomath}
		\begin{equation*}
		\frac{\lnorm{2}{h(z+\frac{c}{NM})-h(z)}}{\frac{c}{NM}}>(1+\varphi)\frac{\lnorm{2}{h(c)-h(0)}}{c}.
		\end{equation*}
		\end{linenomath}
	\end{enumerate}
\end{restatable}
\begin{proof}
	Let $M\in\N$ and $\varphi\in(0,1)$ be parameters \hl{satisfying \eqref{eq:vph_M}}. Let $n\in\N$ and $h\colon[0,c]\to\R^{n}$ be an $\omega$-mapping. The assertion of the Lemma holds for $h$ if and only if the assertion holds for $\rho\circ h$, where $\rho\colon\R^{n}\to\R^{n}$ is any distance preserving transformation. Therefore, we may assume that $h(0)=(0,0,\ldots,0)$ and $h(c)=(A,0,\ldots,0)$ where \hl{$A\geq 0$}.	
	
	Assume that the second statement does not hold for $h$. In other words we have that
	\begin{linenomath}
		\begin{equation}\label{eq:not2}
		\frac{\left\|h(x+\frac{c}{NM})-h(x)\right\|_{2}}{\frac{c}{NM}}\leq(1+\varphi)\frac{A}{c}
		\end{equation} 
	\end{linenomath}
	for all $x\in\frac{c}{NM}\Z\cap[0,c-\frac{c}{NM}]$. We complete the proof, by verifying that the first statement holds for $h$.	
	
	We distinguish two cases, namely $A=0$ and $A>0$. In the former we have \hl{$h(c)=\mb{0}$}. Using \eqref{eq:not2}, we get that \hl{$h(z)=\mb{0}$} for every $z\in\frac{c}{NM}\Z\cap [0,c]$.
	For any \hl{$x\in[0,c]$} we can find $z\in\frac{c}{NM}\Z\cap [0,c]$, $z\leq x$, such that $\abs{x-z}\leq\frac{c}{NM}$. This, however, implies that for any \hl{$x\in\left[0,c-\frac{c}{N}\right]$}
\begin{linenomath}
\begin{equation*}
\lnorm{2}{h\left(x+\frac{c}{N}\right)-h(x)-\frac{1}{N}(h(c)-h(0))}\leq 2\omega\left(\frac{c}{NM}\right).
\end{equation*}
\end{linenomath}
	\hl{Applying the condition on $M$ from \eqref{eq:vph_M}} and using the submultiplicativity of $\omega$, we verify that the last quantity is at most $\varepsilon\omega\left(\frac{c}{N}\right)$. Hence statement~\eqref{1dimdich1} holds with $\Omega:=[N-1]$.
	In the remainder of the proof, we assume the second case $A>0$.
\short
Let $S_{i}:=\left[\frac{(i-1)c}{N},\frac{ic}{N}\right]$ for $i\in[N]$, {\hl{$t:=t(\varepsilon)\in(\varphi,1)$}} be some parameter to be determined later in the proof,
	\begin{linenomath}
	\begin{equation*}
		P:=\left\{x\in \frac{c}{NM}\Z\cap\left[0,c-\frac{c}{N}\right]\colon h^{(1)}\left(x+\frac{c}{N}\right)-h^{(1)}(x)>\frac{(1-t)A}{N}\right\}
\end{equation*}
\end{linenomath}
and
\begin{linenomath}
\begin{equation*}
	\Omega:=\left\{i\in[N-1]\colon \frac{c}{NM}\Z\cap S_{i}\subseteq P\right\}.
\end{equation*}
\end{linenomath}
The next part of the proof does not use information about the modulus of continuity of $h$. Therefore, the argument of \cite[\hl{Lem.~3.2}, p.~634]{DKK2018}, which treats the Lipschitz case, may be applied here with absolutely no change to derive 
\begin{linenomath}
	\begin{equation}\label{eq:P}
	\left\|h\left(x+\frac{c}{N}\right)-h(x)-\frac{1}{N}(h(c)-h(0))\right\|_{2}\leq\frac{\sqrt{t^{2}+4t}A}{N}\leq\frac{\sqrt{5t}A}{N}\qquad \forall x\in P,
	\end{equation}
\end{linenomath}
and 
\begin{linenomath}
\begin{equation*}
\left|\Omega\right|\geq \br*{1-\frac{6\varphi}{\varphi+t}}(N-1).
\end{equation*}
\end{linenomath}
{\hl{The first of these inequalities corresponds to \cite[(B.3)]{DKK2018}.}} For any $i\in\Omega$ and $x\in S_{i}$, we can find $x'\in P$ with $\abs{x'-x}\leq c/\jl{(}NM\jl{)}$.
\fi\normalcolor
\arxiv	
	 The next passage of text (approximately one page) is from \cite[Lem.~3.2, p.~634]{DKK2018}. For later use, we point out that \eqref{eq:not2} implies
	\begin{linenomath}
		\begin{equation}\label{eq:lipbdgrid}
		\left\|h(b)-h(a)\right\|_{2}\leq (1+\varphi)\frac{A}{c}\left\|b-a\right\|_{2}
		\end{equation}
	\end{linenomath}
	whenever $a,b\in\frac{c}{NM}\Z\cap[0,c]$. Let $S_{i}=\left[\frac{(i-1)c}{N},\frac{ic}{N}\right]$ for $i\in[N]$, $t:=t(\varepsilon)\in(\varphi,1)$ be some parameter to be determined later in the proof and 
	\begin{linenomath}
		\begin{equation*}
		P:=\left\{x\in \frac{c}{NM}\Z\cap\left[0,c-\frac{c}{N}\right]\colon h^{(1)}\left(x+\frac{c}{N}\right)-h^{(1)}(x)>\frac{(1-t)A}{N}\right\}.
		\end{equation*}
	\end{linenomath}
	For $x\in P$ we have 
	\begin{linenomath}
		\begin{equation*}
		\left|h^{(1)}\left(x+\frac{c}{N}\right)-h^{(1)}(x)-\frac{A}{N}\right|\leq\frac{tA}{N}.
		\end{equation*}
	\end{linenomath}
	This inequality follows from the definition of $P$, the inequality \eqref{eq:lipbdgrid} and $t>\varphi$. For the remaining co-ordinate functions we have
	\begin{linenomath}
		\begin{equation*}
		\sum_{i=2}^{n}\left|h^{(i)}\left(x+\frac{c}{N}\right)-h^{(i)}(x)\right|^{2}\leq \frac{(1+\varphi)^{2}A^{2}}{N^{2}}-\frac{(1-t)^{2}A^{2}}{N^{2}}\leq \frac{4tA^{2}}{N^{2}}.
		\end{equation*}
	\end{linenomath}
	Combining the two inequalities above we deduce
	\begin{linenomath}
		\begin{equation}\label{eq:P}
		\left\|h\left(x+\frac{c}{N}\right)-h(x)-\frac{1}{N}(h(c)-h(0))\right\|_{2}\leq\frac{\sqrt{t^{2}+4t}A}{N}\leq\frac{\sqrt{5t}A}{N}\qquad \forall x\in P.
		\end{equation}
	\end{linenomath}
	Let $\Gamma\subset[0,1]$ be a maximal $c/N$-separated subset of $\frac{c}{NM}\Z\cap\left[0,c-\frac{c}{N}\right]\setminus P$ and let $x_{1},\ldots,x_{\left|\Gamma\right|}$ be the elements of $\Gamma$. Then the intervals $([x_{i},x_{i}+\frac{c}{N}])_{i=1}^{\left|\Gamma\right|}$ can only intersect in the endpoints.
	Therefore the set $[0,c]\setminus \bigcup_{i=1}^{\left|\Gamma\right|}[x_{i},x_{i}+\frac{c}{N}]$ is a finite union of intervals with endpoints in $\frac{c}{NM}\Z\cap[0,c]$ and with total length $c-\frac{\left|\Gamma\right|c}{N}$. Using $\Gamma\cap P=\emptyset$ and \eqref{eq:lipbdgrid} we deduce that
	\begin{linenomath}
		\begin{equation*}
		A=h^{(1)}(c)-h^{(1)}(0)\leq\left|\Gamma\right|\frac{(1-t)A}{N}+(1+\varphi)\frac{A}{c}\left(c-\frac{\left|\Gamma\right|c}{N}\right).
		\end{equation*}
	\end{linenomath}
	Since $A>0$, we may rearrange this inequality to obtain
	\begin{linenomath}
		\begin{equation*}
		\left|\Gamma\right|\leq\frac{\varphi}{\varphi+t}N\leq\frac{2\varphi}{\varphi+t}(N-1),
		\end{equation*}
	\end{linenomath}
	where, for the last inequality, we apply $N\geq 2$. It follows that the set $\frac{c}{NM}\Z\cap[0,c-\frac{c}{N}]\setminus P$ can intersect at most $\frac{6\varphi}{\varphi+t}(N-1)$ intervals $S_{i}$. Letting 
	\begin{linenomath}
		\begin{equation*}
		\Omega:=\left\{i\in[N-1]\colon \frac{c}{NM}\Z\cap S_{i}\subseteq P\right\}
		\end{equation*}
	\end{linenomath}
	we deduce that $\left|\Omega\right|\geq \br*{1-\frac{6\varphi}{\varphi+t}}(N-1)$. Moreover for any $i\in\Omega$ and $x\in S_{i}$, we can find $x'\in P$ with $\left|x'-x\right|\leq c/NM$.
\fi\normalcolor
This allows us to apply \eqref{eq:P} to get
\begin{linenomath}
	\begin{multline*}
		\left\|h\left(x+\frac{c}{N}\right)-h(x)-\frac{1}{N}(h(c)-h(0))\right\|_{2}\\
		\leq\left\|h\left(x'+\frac{c}{N}\right)-h(x')-\frac{1}{N}(h(c)-h(0))\right\|_{2}+2\omega\left(\frac{c}{NM}\right)
		\\ \leq\frac{\sqrt{5t}\omega(c)}{N}+2\omega\left(\frac{c}{NM}\right)
		\leq \left(\sqrt{5t}+2\omega\left(\frac{1}{M}\right)\right)\omega\left(\frac{c}{N}\right).
	\end{multline*}
\end{linenomath}
In the final step, we used the concavity of $\omega$.
Making the choice $t:=\frac{\varepsilon^{2}}{20}$ and applying the bound on $M$ from \eqref{eq:vph_M}, we obtain that the last quantity is at most $\varepsilon \omega\left(\frac{c}{N}\right)$. Finally, the bound on $\varphi$ from \eqref{eq:vph_M} and the choice of $t$ ensure that $\frac{6\varphi}{\varphi +t}<\varepsilon$.
\end{proof} 

\begin{lemma}
\label{lemma:dichotomy}
Let $d\in\N$, $\omega\in\M$ and \hl{$c,\varepsilon\in(0,a_{\omega})$}. Then there exist parameters
\begin{linenomath}
	\begin{equation*}
		\varphi=\varphi(d,\omega,\varepsilon)\hl{\in (0,1)},\quad N_0=N_0(d,\omega,\varepsilon,c)\hl{\geq 1}
	\end{equation*}
\end{linenomath}
such that for all $N\in\N$, $N\geq N_{0}$ there exists a parameter
\begin{linenomath}
\begin{equation*}
M=M(\hl{N},d,\omega,\varepsilon,c)\in\N
\end{equation*}
\end{linenomath}
such that for all $n\geq d$ and all bi-$\omega$-mappings 
\begin{linenomath}
\begin{equation*}
h\colon [0,c]\times[0,c/N]^{d-1}\to\R^{n}\qquad \text{with $\bilipomeg(h)\leq 1$}
\end{equation*} 
\end{linenomath}
at least one of the following statements holds:
\begin{enumerate}
\item\label{dich1} There exists a set $\Omega\subset[N-1]$ with $\left|\Omega\right|\geq(1-\varepsilon)(N-1)$ such that for all $i\in \Omega$ and for all $\mb{x}\in\left[\frac{(i-1)c}{N},\frac{ic}{N}\right]\times\left[0,\frac{c}{N}\right]^{d-1}$
\begin{linenomath}
\begin{equation}\label{eq:translation}
\left\|h\left(\mb{x}+\frac{c}{N}\mb{e}_{1}\right)-h(\mb{x})-\frac{1}{N}(h(c\mb{e}_{1})-h(\mb{0}))\right\|_{2}\leq\varepsilon\omega\left(\frac{c}{N}\right).
\end{equation}
\end{linenomath}
\item\label{dich2} There exists $\mb{z}\in\frac{c}{NM}\Z^{d}\cap([0,c-\frac{c}{NM}]\times[0,\frac{c}{N}-\frac{c}{NM}]^{d-1})$ such that 
\begin{linenomath}
\begin{equation*}
\frac{\left\|h(\mb{z}+\frac{c}{NM}\mb{e}_{1})-h(\mb{z})\right\|_{2}}{\frac{c}{NM}}>(1+\varphi)\frac{\left\|h(c\mb{e}_{1})-h(\mb{0})\right\|_{2}}{c}.
\end{equation*}
\end{linenomath}
\end{enumerate}
\end{lemma}
\begin{proof}	
	In this proof we will sometimes add the superscript $d$ or $d-1$ to objects such as the Lebesgue measure $\leb$ or vectors $\mb{e}_{i}$, $\mb{0}$ in order to emphasise the dimension of the Euclidean space to which they correspond. For $d\geq 2$, we will express points in $\R^{d}$ in the form $\mb{x}=(x_{1},x_{2},\ldots,x_{d})$. Given $\mb{x}=(x_{1},\ldots,x_{d})\in \R^{d}$ and $s\in\R$ we let
	\begin{linenomath}
		\begin{equation*}
		\mb{x}\wedge s:=(x_{1},\ldots,x_{d},s)
		\end{equation*}
	\end{linenomath}
	denote the point in $\R^{d+1}$ formed by concatenation of $\mb{x}$ and $s$.
	
	The case $d=1$ is dealt with by Lemma~\ref{lemma:1dimdich}. Let $d\geq 2$ and suppose that the statement of the lemma holds when $d$ is replaced with $d-1$.	
	We define an additional parameter $\theta:=\theta(d, \omega,\varepsilon)$ whose precise value will be specified later.
	Given $\omega$, $\varepsilon$ and $c>0$ we let $\varphi:=\varphi(d,\omega,\varepsilon)\in(0,1)$ and $N_{0}(d,\varepsilon):=N_{0}(d,\omega,\varepsilon,c)\in\N$ be parameters on which we impose various conditions in the course of the proof.
	For now, we just prescribe that $0<\varphi<\frac{1}{2}\varphi(d-1,\omega,\theta(d,\omega,\varepsilon))$
	and $N_{0}(d,\varepsilon)\geq N_{0}(d-1,\theta(d,\omega,\varepsilon))$.

	Let $N\geq N_{0}(d,\varepsilon)$ and $M:=M(N,d,\omega,\varepsilon,c)$ be a parameter to be determined later. For $M_{d-1}:=M(N,d-1,\omega,\theta(d,\omega,\varepsilon),c)$ we prescribe first that $M\in M_{d-1}\Z$, so that $\frac{c}{NM_{d-1}}\Z\subseteq\frac{c}{NM}\Z$.
	
	Let $n\geq d$ and $h\colon [0,c]\times [0,c/N]^{d-1}\to\R^{n}$ be a bi-$\omega$-mapping with $\bilipomeg(h)\leq 1$. For each $s\in[0,c/N]$ the mapping $h\wedge s\colon [0,c]\times[0,c/N]^{d-2}\to\R^{n}$ defined by
	\begin{linenomath}
		\begin{equation*}
		h\wedge s(\mb{x}):=h(\mb{x}\wedge s)=h(x_{1},x_{2},\ldots,x_{d-1},s).
		\end{equation*} 
	\end{linenomath}
\hl{is a bi-$\omega$-mapping with $\bilipomeg(h\wedge s)\leq \bilipomeg(h)\leq 1$. This is straightforward to verify.
	Thus, for each $s\in [0,c/N]$ we may apply the induction hypothesis to $h\wedge s$.}	
	For each $s\in[0,c/N]$ we get that at least one of the following statements holds:
	\begin{itemize}
		\item[(\mylabel{dich1ind}{1$_{s}$})]There exists a set $\Omega_{s}\subset[N-1]$ with $\left|\Omega_{s}\right|\geq (1-\theta)(N-1)$ such that for all $i\in\Omega_{s}$ and $\mb{x}\in \left[\frac{(i-1)c}{N},\frac{ic}{N}\right]\times\left[0,\frac{c}{N}\right]^{d-2}$ it holds that
		\begin{linenomath}
			\begin{multline*}
			\left\|h\wedge s\left(\mb{x}+\frac{c}{N}\mb{e}_{1}^{d-1}\right)-h\wedge s(\mb{x})-\frac{1}{N}\br*{h\wedge s\bigr(c\mb{e}_{1}^{d-1}\bigl)-h\wedge s\bigr(\mb{0}^{d-1}\bigl)}\right\|_{2}\\
			\leq\theta\omega\left(\frac{c}{N}\right).
			\end{multline*}
		\end{linenomath}
		\item[(\mylabel{dich2ind}{2$_{s}$})] There exists $\mb{z}_{s}\in\frac{c}{NM_{d-1}}\Z^{d-1}\cap\bigr([0,c-\frac{c}{NM_{d-1}}]\times[0,\frac{c}{N}-\frac{c}{NM_{d-1}}]^{d-2}\bigl)$ such that
		\begin{linenomath}
			\begin{multline*}
			\frac{\left\|h\wedge s\br*{\mb{z}_{s}+\frac{c}{NM_{d-1}}\mb{e}_{1}^{d-1}}-h\wedge s(\mb{z}_{s})\right\|_{2}}{\frac{c}{NM_{d-1}}}\\
			>(1+2\varphi)\frac{\left\|h\wedge s\br*{c\mb{e}_{1}^{d-1}}-h\wedge s\br*{\mb{0}^{d-1}}\right\|_{2}}{c}.
			\end{multline*}
		\end{linenomath}
	\end{itemize}

	Suppose first that statement~\eqref{dich2ind} holds for some $s\in [0,c/N]$. We will show that statement~\ref{dich2} holds for $h$. Choose a number $s'\in\frac{c}{NM}\Z\cap[0,\frac{c}{N}-\frac{c}{NM}]$ with $s'\leq s$ and $\left|s'-s\right|\leq \frac{c}{NM}$.
	Setting $\mb{w}=\mb{z}_{s}\wedge s'$ we note that $\mb{w}$ is an element of $\frac{c}{NM}\Z^{d}\cap\hl{\Bigl(}[0,c-\frac{c}{NM_{d-1}}]\times[0,\frac{c}{N}-\frac{c}{NM}]^{d-1}\hl{\Bigr)}$, $\left\|\mb{w}-\mb{z}_{s}\wedge s\right\|_{2}\leq\frac{c}{NM}$ and $\lnorm{2}{h\wedge s(c\mb{e}_{1}^{d-1})-h\wedge s(\mb{0}^{d-1})}\geq\lnorm{2}{h(c\mb{e}_{1}^{d})-h(\mb{0}^{d})}-2\omega\left(\frac{c}{N}\right)$. We use these inequalities and the inequality of~\eqref{dich2ind} to derive
	\begin{linenomath}
		\begin{align*}
		&\left\|h\left(\mb{w}+\frac{c}{NM_{d-1}}\mb{e}_{1}^{d}\right)-h(\mb{w})\right\|_{2}\\
		&\geq\left\|h\wedge s\left(\mb{z}_{s}+\frac{c}{NM_{d-1}}\mb{e}_{1}^{d-1}\right)-h\wedge s(\mb{z}_{s})\right\|_{2}-2\omega\left(\frac{c}{NM}\right)\\	
		&> (1+2\varphi)\left(\frac{\left\|h(c\mb{e}_{1}^{d})-h(\mb{0}^{d})\right\|_{2}-2\omega\left(\frac{c}{N}\right)}{c}\right)\frac{c}{NM_{d-1}}-2\omega\left(\frac{c}{NM}\right)\\
		&\geq\left(1+2\varphi-\frac{2(1+2\varphi)\omega\left(\frac{c}{N}\right)}{\lnorm{2}{h(c\mb{e}_{1}^{d})-h(\mb{0}^{d})}}-\frac{2\omega\left(\frac{c}{NM}\right)NM_{d-1}}{\lnorm{2}{h(c\mb{e}_{1}^{d})-h(\mb{0}^{d})}}\right)\frac{\left\|h(c\mb{e}_{1}^{d})-h(\mb{0}^{d})\right\|_{2}}{NM_{d-1}}\\
		&\geq\left(1+2\varphi-\frac{4\omega\left(\frac{c}{N_{0}(d,\varepsilon)}\right)}{\omega^{-1}(c)}-\frac{2\omega\left(\frac{c}{NM}\right)NM_{d-1}}{\omega^{-1}(c)}\right)\frac{\left\|h(c\mb{e}_{1}^{d})-h(\mb{0}^{d})\right\|_{2}}{NM_{d-1}}\\
		&>(1+\varphi)\frac{\left\|h(c\mb{e}_{1}^{d})-h(\mb{0}^{d})\right\|_{2}}{NM_{d-1}}. 	
		\end{align*}
	\end{linenomath}
	To deduce the fourth inequality in the sequence above we use \hl{that $\varphi\in (0,1/2)$ and} the $\omega^{-1}$-bound on $h$. In fact, this is the only place in the proof of Lemma~\ref{lemma:dichotomy} where we use that the mapping $h$ is bi-$\omega$-continuous and not just $\omega$-continuous. The final inequality is ensured by taking $N_{0}(d,\varepsilon)$ and $M$ sufficiently large. 
Specifically, using the submultiplicativity of $\omega$, it is sufficient to take
\begin{linenomath}
\begin{equation*}
N_{0}\geq \frac{1}{\omega^{-1}\left(\frac{\varphi\omega^{-1}(c)}{8\omega(c)}\right)},\qquad M\geq \frac{1}{\omega^{-1}\left(\frac{1}{NM_{d-1}}\right)}.
\end{equation*}
\end{linenomath}

	From the final inequality obtained for $\left\|h\left(\mb{w}+\frac{c}{NM_{d-1}}\mb{e}_{1}^{d}\right)-h(\mb{w})\right\|_{2}$ it follows that there exists $i\in\left[\frac{M}{M_{d-1}}\right]$ so that the point $\mb{z}:=\mb{w}+\frac{(i-1)c}{NM}\mb{e}_{1}^{d}$ verifies statement~\ref{dich2} for $h$.
	
	We may now assume that the first statement~\eqref{dich1ind} holds for all $s\in[0,c/N]$. We complete the proof by verifying statement \ref{dich1} for $h$. Whenever $\mb{x}\in[0,c]\times[0,c/N]^{d-2}$ and $s\in[0,c/N]$ satisfy the inequality of~\eqref{dich1ind} we have that
	\begin{linenomath}
		\begin{equation}\label{eq:gdpts}
		\left\|h\br*{(\hl{\mb{x}}\wedge s)+\frac{c}{N}\mb{e}_{1}}-h(\hl{\mb{x}}\wedge s)-\frac{1}{N}(h(c\mb{e}_{1})-h(\mb{0}))\right\|\leq\theta\omega\left(\frac{c}{N}\right)+\frac{2\omega\left(\frac{c}{N}\right)}{N}.
		\end{equation}
	\end{linenomath}

\short
Let $R:=\left[0,c-\frac{c}{N}\right]\times\left[0,\frac{c}{N}\right]^{d-1}$,
\begin{linenomath}
	\begin{equation*}
		A:=\left\{\mb{x}\in R\colon \mb{x}\text{ satisfies \eqref{eq:translation} with $\varepsilon=\theta+\frac{2}{N}$}\right\},
	\end{equation*}
\end{linenomath}
$S_{i}:=\left[\frac{(i-1)c}{N},\frac{ic}{N}\right]\times\left[0,\frac{c}{N}\right]^{d-1}$ for $i\in\hl{[N-1]}$ and
\begin{linenomath}
\begin{equation*}
	\Omega:=\left\{i\in[N-1]\colon \leb^{d}(A\cap S_{i})\geq (1-\sqrt{\theta})\leb^{d}(S_{i})\right\}.
\end{equation*}
\end{linenomath}
The next part of the proof uses a Fubini theorem argument which does not require any information about the modulus of continuity of the mapping $h$. Therefore, this part of the proof is identical to the Lipschitz case and the Fubini theorem argument of \cite[p.~637]{DKK2018} now applies without any modification to obtain
\begin{linenomath}
	\begin{equation*}
		\frac{\left|\Omega\right|}{N-1}\geq (1-\sqrt{\theta})\geq 1-\varepsilon,
	\end{equation*}
\end{linenomath}
where for the last inequality we require $\theta\leq\varepsilon^{2}$. 

Fix $i\in\Omega$ and $\mb{x}\in S_{i}$. Then for any cube $Q\subseteq S_{i}$ with sidelength $(2\sqrt{\theta}\leb^{d}(S_{i}))^{\frac{1}{d}}$ we have $A\cap Q\neq \emptyset$. Therefore, 
\fi\normalcolor
\arxiv
	The next passage of text (approximately one page) is from \cite[p.~636--637]{DKK2018}. Let $R:=\left[0,c-\frac{c}{N}\right]\times\left[0,\frac{c}{N}\right]^{d-1}$ and
	\begin{linenomath}
		\begin{equation*}
		A:=\left\{\mb{x}\in R\colon \mb{x}\text{ satisfies \eqref{eq:translation} with $\varepsilon=\theta+\frac{2}{N}$}\right\}.
		\end{equation*}
	\end{linenomath}
	Using \eqref{eq:gdpts} and the fact that statement~\eqref{dich1ind} holds for every $s\in[0,c/N]$ we deduce	
	\begin{linenomath}
		\begin{equation*}
		\leb^{d-1}(A\cap \left\{\mb{x}\colon x_{d}=s\right\})\geq (1-\theta)\leb^{d-1}(R\cap \left\{\mb{x}\colon x_{d}=s\right\})\qquad \text{for all $s\in[0,c/N]$}.
		\end{equation*}
	\end{linenomath}
	Therefore, by Fubini's theorem,
	\begin{linenomath}
		\begin{equation*}
		\leb^{d}(A)\geq(1-\theta)\leb^{d}(R).
		\end{equation*}
	\end{linenomath}
	For each $i\in[N-1]$ we let $S_{i}:=\left[\frac{(i-1)c}{N},\frac{ic}{N}\right]\times\left[0,\frac{c}{N}\right]^{d-1}$, define
	\begin{linenomath}
		\begin{equation*}
		\Omega:=\left\{i\in[N-1]\colon \leb^{d}(A\cap S_{i})\geq (1-\sqrt{\theta})\leb^{d}(S_{i})\right\}
		\end{equation*}
	\end{linenomath}
	and observe that
	\begin{linenomath}
		\begin{align*}
		\leb^{d}(A)\leq \left|\Omega\right|\frac{\leb^{d}(R)}{N-1}+(N-1-\left|\Omega\right|)(1-\sqrt{\theta})\frac{\leb^{d}(R)}{N-1}.
		\end{align*}
	\end{linenomath}
	Combining the two inequalities derived above for $\leb^{d}(A)$ and requiring $\theta\leq\varepsilon^2$, we deduce
	\begin{linenomath}
		\begin{equation*}
		\frac{\left|\Omega\right|}{N-1}\geq (1-\sqrt{\theta})\geq 1-\varepsilon. 
		\end{equation*}
	\end{linenomath}
	Moreover, for any $i\in\Omega$ and any cube $Q\subseteq S_{i}$ with sidelength $(2\sqrt{\theta}\leb^{d}(S_{i}))^{\frac{1}{d}}$ we have $A\cap Q\neq \emptyset$. Therefore, for any $i\in\Omega$ and any $\mb{x}\in S_{i}$
\fi\normalcolor
we can find $\mb{x}'\in A\cap S_{i}$ with
\begin{linenomath}
	\begin{equation*}
		\left\|\mb{x}'-\mb{x}\right\|_{2}\leq \sqrt{d}(2\sqrt{\theta}\leb^{d}(S_{i}))^{\frac{1}{d}}\leq\frac{2\sqrt{d}\theta^{1/2d}c}{N}.
	\end{equation*}
\end{linenomath}
Using this approximation, we obtain
\begin{linenomath}
	\begin{multline*}
		\left\|h\left(\mb{x}+\frac{c}{N}\mb{e}_{1}\right)-h(\mb{x})-\frac{1}{N}(h(c\mb{e}_{1})-h(\mb{0}))\right\|_{2}\\
		\leq \left\|h\left(\mb{x}+\frac{c}{N}\mb{e}_{1}\right)-h\left(\mb{x}'+\frac{c}{N}\mb{e}_{1}\right)\right\|_{2}\\+\left\|h\left(\mb{x}'+\frac{c}{N}\mb{e}_{1}\right)-h(\mb{x}')-\frac{1}{N}(h(c\mb{e}_{1})-h(\mb{0}))\right\|_{2}+\left\|h(\mb{x}')-h(\mb{x})\right\|_{2}\\
		\leq 2\omega\left(\frac{2\sqrt{d}\theta^{1/2d}c}{N}\right)+\left(\theta+\frac{2}{N}\right)\omega\left(\frac{c}{N}\right)\\
		\leq\left(2\omega\left(2\sqrt{d}\theta^{1/2d}\right)+\theta+\frac{2}{N_{0}(d,\varepsilon)}\right)\omega\left(\frac{c}{N}\right)\leq \varepsilon\omega\left(\frac{c}{N}\right),
	\end{multline*}
\end{linenomath}
where the final inequality is satisfied by taking $N_0(d,\varepsilon)\geq\frac{6}{\varepsilon}$
and
\begin{linenomath}
\begin{equation}\label{eq:theta_bound}
	\theta(d,\omega, \varepsilon)\leq \left(\frac{\omega^{-1}\br*{\frac{1}{6}}\omega^{-1}\br*{\varepsilon}}{2\sqrt{d}}\right)^{2d}
\end{equation}
\end{linenomath}
This choice of $\theta$ also satisfies the requirement $\theta\leq\varepsilon^2$ imposed before.
\end{proof}

\subsection{Iterating Lemma~\ref{lemma:dichotomy}}
In this subsection we identify a certain subfamily $\M_{0}\subseteq\M$ with the property that for any modulus $\omega\in\M_{0}$ we may iterate Lemma~\ref{lemma:dichotomy} a controlled number of times in order to eliminate conclusion~\ref{dich2} of the dichotomy. In the next subsection we verify that the subfamily $\M_{0}$ contains all moduli of the form $\omega(t)=t\left(\log\frac{1}{t}\right)^{\alpha}$.

\begin{define}[The family $\mc{M}_0$]\label{def_par_subfam}
We use Lemma~\ref{lemma:dichotomy} to generate sequences of parameters. Given $d\in\N$, $\omega\in\mc{M}$ and \hl{$c,\varepsilon\in (0,a_{\omega})$} we define sequences $(N_{i})_{i=1}^{\infty}$, $(M_{i})_{i=1}^{\infty}$ and $(c_{i})_{i=1}^{\infty}$ by 
\begin{linenomath}
\begin{multline}\label{eq:def_parameters}
N_{1}:= N_{0}(d,\omega,\varepsilon,c),\qquad N_{i}:= N_{0}\left(d,\omega,\varepsilon,c_{i}\right),\quad i\geq 2,\\
c_{1}:=c,\qquad c_{i}:=\frac{c_{i-1}}{N_{i-1}M_{i-1}},\quad i\geq 2,\quad \text{and }\\ M_{\hl{i}}:=M_{0}(d,\omega,\varepsilon,c_{\hl{i}}):=M(N_{i},d,\omega,\varepsilon,c_{\hl{i}})\in\N,\qquad \hl{i}\geq 1.
\end{multline}
\end{linenomath}

Let $\M_{0}$ be defined as the family of all moduli $\omega\in\M$ for which the following condition holds: \jl{f}or any $d\in\N$, $c\in(0,a_{\omega})$ and $\varepsilon\in(0,1)$ there exists $r:=r(d,\omega,\varepsilon,c)\in\N$ such that for the parameter $\varphi=\varphi(d,\omega,\varepsilon)$ of Lemma~\ref{lemma:dichotomy} and the parameter sequence $(c_{i})_{i=1}^{\infty}$ defined in the paragraph above, we have
\begin{linenomath}
\begin{equation}\label{eq:num_iter}
\frac{(1+\varphi)^{r}\omega^{-1}(c)}{c}\geq \frac{\omega(c_{r+1})}{c_{r+1}}.
\end{equation}
\end{linenomath}
\end{define}

	\begin{lemma}\label{lemma:terminate}
    Let $d\in\N$, $\omega\in\M_{0}$, \hl{$c,\varepsilon\in(0,a_{\omega})$}, $n\geq d$ and $g\colon[0,c]\times[0,c/N]^{d-1}\to\R^{n}$ be a bi-$\omega$-mapping with $\bilipomeg(g)\leq 1$. Let the parameters $(c_{i})_{i=1}^{\infty}$, $(N_{i})_{i=1}^{\infty}$, $(M_{i})_{i=1}^{\infty}$ and $r$ be defined according to Definition~\ref{def_par_subfam}. Then there exists $p\in[r]$ and
	\begin{linenomath}
		\begin{equation*}
		\mb{z}_{1}=\mb{0},\quad \mb{z}_{i+1}\in c_{i+1}\Z^{d}\cap\hl{\Biggl(}[0,c_{i}-c_{i+1}]\times\left[0,\frac{c_{i}}{N_{\hl{i}}}-c_{i+1}\right]^{d-1}\hl{\Biggr)}\hl{,}\quad i\in[p-1],
		\end{equation*} 
	\end{linenomath}
	such that statement~\ref{dich1} of Lemma~\ref{lemma:dichotomy} is valid for the mapping $g_{p}\colon[0,c_{p}]\times[0,c_{p}/N_{p}]^{d-1}\to\R^{n}$ defined by
	\begin{linenomath}
		\begin{equation}\label{eq:g_p}
		g_{p}(\mb{x}):=g\br*{\mb{x}+\sum_{i=1}^{p}\mb{z}_{i}}.
		\end{equation}	
	\end{linenomath}
	\end{lemma}
\short \begin{proof} To prove Lemma~\ref{lemma:terminate} we first apply Lemma~\ref{lemma:dichotomy} to $g$. If statement~\ref{dich1} of the dichotomy holds we stop. Otherwise we find via statement~\ref{dich2} of the dichotomy a $\frac{1}{N_{1}M_{1}}$-times smaller copy of the original cuboid $[0,c]\times[0,c/N_{1}]^{d-1}$ on which $g$ stretches the basepoints corresponding to $c\mb{e}_{1}$ and $\mb{0}$ by a factor $(1+\varphi)$-times more than their analogues in the original cuboid. We now apply Lemma~\ref{lemma:dichotomy} to $g$ restricted to this smaller cuboid, that is, to the mapping denoted by $g_{2}$ in the statement of the present lemma, and repeat the procedure described above. We iterate this procedure as many times as possible. Now we argue that this algorithm terminates after at most a finite number $r$ of iterations, where $r\in\N$ is given by Definition~\ref{def_par_subfam}. Otherwise, after $(r+1)$ iterations we would find a cuboid of length $c_{r+1}$ for which $g$ stretches the basepoints by a factor $(1+\varphi)^{r}$ more than it stretches $\mb{0}$ and $c\mb{e}_{1}$. This however is incompatible with $\lipomeg(g)\leq 1$ and \eqref{eq:num_iter}.
	
	We do not provide a more formal proof of Lemma~\ref{lemma:terminate} here because the argument is a simple modification of \cite[Proof of Lem.~3.5, p.~638]{DKK2018}. Algorithm~B.1 of \cite[p.~638]{DKK2018} may be applied here exactly as it is written and then the remaining argument of \cite[p.~638]{DKK2018} requires only trivial changes such as replacing Lipschitz estimates with $\omega$-continuity estimates.\end{proof}
	\fi\normalcolor
	\arxiv\begin{proof}
	The proof is a simple modification of \cite[Proof of Lem.~3.5, p.~638]{DKK2018}. 
	 Let $\varphi:=\varphi(d,\omega,\varepsilon)$ be given by the conclusion of Lemma~\ref{lemma:dichotomy}. We implement the following algorithm.
	\begin{algorithm}[{\cite[Alg.~B.1]{DKK2018}}]\label{algorithm}
		Set $i=1$, $\mb{z}_{1}=\mb{0}$ and $g_{1}=g$.
		\begin{enumerate}
			\item\label{step1} If statement $1$ of Lemma~\ref{lemma:dichotomy} holds for $h=g_{i}$ and $c=c_{i}$ then stop. If not proceed to step 2.
			\item\label{step2} Choose $\mb{z}_{i+1}\in c_{i+1}\Z^{d}\cap[0,c_{i}-c_{i+1}]\times[0,\frac{c_{i}}{N}-c_{i+1}]^{d-1}$ such that
			\begin{linenomath}
				\begin{equation}\label{eq:2i}
				\frac{\left\|g_{i}(\mb{z}_{i+1}+c_{i+1}\mb{e}_{1})-g_{i}(\mb{z}_{i+1})\right\|_{2}}{c_{i+1}}>(1+\varphi)\frac{\left\|g_{i}(c_{i}\mb{e}_{1})-g_{i}(\mb{0})\right\|_{2}}{c_{i}}
				\end{equation}	 
			\end{linenomath}
			and define $g_{i+1}\colon [0,c_{i+1}]\times[0,c_{i+1}/N]^{d-1}\to\R^{kd}$ by
			\begin{linenomath}
				\begin{equation*}
				g_{i+1}(\mb{x}):=g_{i}(\mb{x}+\mb{z}_{i+1})=g\left(\mb{x}+\sum_{j=1}^{i+1}\mb{z}_{j}\right).
				\end{equation*}
			\end{linenomath}
			\item Set $i=i+1$ and return to step 1.
		\end{enumerate}
	\end{algorithm}
	At each potential iteration $i\geq 1$ of Algorithm~\ref{algorithm}, the conditions of Lemma~\ref{lemma:dichotomy} are satisfied for $d$, $\omega$, $\varepsilon$, $M$, $\varphi$, $N_{0}$, $c=c_{i}$, $n$, $N$ and $h=g_{i}\colon [0,c_{i}]\times [0,c_{i}/N]^{d-1}\to \R^{n}$. Therefore, whenever the algorithm does not terminate in step~\ref{step1}, we have that such a point $\mb{z}_{i+1}$ required by step~\ref{step2} exists by Lemma~\ref{lemma:dichotomy}. 
	
	To complete the proof, it suffices to verify that Algorithm~\ref{algorithm} terminates after at most $r$ iterations. This is clear, after rewriting \eqref{eq:2i} in the form
	\begin{linenomath}
		\begin{equation*}
		\frac{\lnorm{2}{g_{i+1}(c_{i+1}\mb{e}_{1})-g_{i+1}(\mb{0})}}{c_{i+1}}>(1+\varphi)\frac{\left\|g_{i}(c_{i}\mb{e}_{1})-g_{i}(\mb{0})\right\|_{2}}{c_{i}}>\frac{(1+\varphi)^{i}\omega^{-1}(c)}{c},
		\end{equation*}
	\end{linenomath}
	where the latter inequality follows by induction and the $\omega$-continuity of $g^{-1}$. If Algorithm~\ref{algorithm} completed $r+1$ iterations then, the inequality above for $i=r$ provides, in light of \eqref{eq:num_iter}, a contradiction to the $\omega$-continuity of $g_{r+1}$. 
\end{proof}\fi\normalcolor

\subsection{Largeness of the subfamily \texorpdfstring{$\M_{0}$}{\unichar{"02133}\textzeroinferior}.}\label{subsec:largeness}
The objective of this subsection is to prove that any $\omega\in\M$ satisfying
\begin{linenomath}
\begin{equation*}
\omega(t)\leq Lt\left(\log\frac{1}{t}\right)^{\alpha},\qquad \text{for all $t\in(0,a_{\omega})$,}
\end{equation*}
\end{linenomath}
for some $\alpha>0$ and $L\geq 1$ belongs to the family $\M_{0}$ of Definition~\ref{def_par_subfam}. The proof relies on establishing sufficiently good bounds on the parameters of Lemma~\ref{lemma:dichotomy} and Definition~\ref{def_par_subfam}.

Throughout the work, the parameter $c$ is usually treated as a constant. In this subsection, however, we are making the dependence on $c$ explicit. The first reason for that is that we are going to apply the bounds derived here to sequences of parameters generated in Definition~\ref{def_par_subfam}, that is, with $c_i$ in place of $c$. Another reason is that we want to make sure that the value of $c$ does not influence the powers of $\varepsilon$ in various bounds of the form $\poly{\varepsilon}$ below. This is to ensure that the value of $\alpha_0(d)$ from Lemma~\ref{lemma:M1} is independent of $c$.

For technical reasons, we need to make sure that the bounds on various parameters established for the modulus $Lt(\log(1/t))^\gamma$ are also valid bounds for the values of the same parameters with respect to all moduli $\omega(t)\leq Lt(\log(1/t))^\gamma$.
\begin{lemma}\label{lemma:parameters}
	Let $\gamma>0$, $L\geq 1$ and $\omega\in \mc{M}$ satisfy
	\begin{linenomath}
	\begin{equation}\label{eq:omega_upper_bound}
	\omega(t)\leq  Lt\left(\log\frac{1}{t}\right)^\gamma,\qquad \text{for all $t\in(0,a_{\omega})$}.
	\end{equation}
	\end{linenomath}
 \hl{Let $d\in\N$ and $\varepsilon,c\in(0,a_{\omega})$.} Then in addition to the conclusion of Lemma~\ref{lemma:dichotomy} the parameters 
	\begin{linenomath}
	\begin{multline*}
		\varphi=\varphi(d,\omega,\varepsilon),\qquad N_{0}=N_{0}(d,\omega,\varepsilon,c),\\
		 M=M(N,d,\omega,\varepsilon,c),\qquad M_{0}=M_{0}(d,\omega,\varepsilon,c):=M(N_{0},d,\omega,\varepsilon,c),
	\end{multline*}
	\end{linenomath}
	may be taken of the form
	\hl{\begin{linenomath}
	\begin{multline}\label{eq:parameters}
	\varphi=\pl{d,\gamma}{L}{\varepsilon},\qquad N_{0}=\frac{\pl{d,\gamma}{L}{\log \frac{1}{c}}}{\pl{d,\gamma}{L}{\varepsilon}},\\
	M=\frac{\pl{d,\gamma}{L}{N}}{\pl{d,\gamma}{L}{\varepsilon}},\qquad M_{0}=\frac{\pl{d,\gamma}{L}{\log\frac{1}{c}}}{\pl{d,\gamma}{L}{\varepsilon}}.
	\end{multline}
	\end{linenomath}}
\end{lemma}

\begin{proof}
	The proof relies on estimating $\omega^{-1}$ from below. We observe the bound
	\begin{equation}\label{eq:strong_lb_omeg_inv}
	\omega^{-1}(s)\geq \Lambda(\gamma,L)\cdot \frac{s}{\left(\log\frac{1}{s}\right)^{\gamma}}, \qquad s\in \left(0,\min\set{\frac{1}{2},\omega(a_{\omega})}\right),
	\end{equation}
	which may be derived as follows: \jl{f}irstly note that
	\begin{equation}\label{eq:weak_lb_omeg_inv}
	t\left(\log\frac{1}{t}\right)^{\gamma}\leq \sup_{r\in (0,1/2)}r^{1/2}\left(\log\frac{1}{r}\right)^{\gamma}\cdot t^{1/2}=\Lambda(\gamma)t^{1/2}, \qquad t\in (0,a_{\omega}).
	\end{equation}
	Here we used that $a_{\omega}\leq 1/2$; see Definition~\ref{def:M}. Let $s\in \left(0,\min\set{\frac{1}{2},\omega(a_{\omega})}\right)$ and $t\in (0,a_{\omega})$ be such that $s=\omega(t)$. Then we may combine \eqref{eq:omega_upper_bound} and \eqref{eq:weak_lb_omeg_inv} to get
	\[
	\log\frac{1}{s}\geq\log\frac{1}{L\Lambda(\gamma)t^{1/2}}=\log\frac{1}{L\Lambda(\gamma)}+\frac{1}{2}\log		\frac{1}{t}\geq\frac{1}{4}\log\frac{1}{t},
	\]
	which is valid for all $t\in\br*{0, (L\Lambda(\gamma))^{-4}}$. Thus, we can write that
	\[
	\log\frac{1}{s}\geq\Lambda(\gamma, L)\log\frac{1}{t}, \qquad t\in(0,a_\omega)
	\]
	for an appropriate choice of the constant on the right-hand side. This inequality and \eqref{eq:omega_upper_bound} imply that
	\begin{linenomath}	
	\begin{equation*}
	\omega^{-1}(s)=t\geq \frac{s}{L\left(\log\frac{1}{t}\right)^{\gamma}}\geq \Lambda(\gamma,L)\frac{s}{\left(\log\frac{1}{s}\right)^{\gamma}}.
	\end{equation*}	
	\end{linenomath}
From \eqref{eq:strong_lb_omeg_inv} and \eqref{eq:weak_lb_omeg_inv} it follows that
	\begin{equation}\label{eq:omeg_inv_poly}
	\hl{
	\omega^{-1}(s)\geq \ply{\gamma,L}{s},\qquad s\in(0,\omega(a_{\omega})).
	}
	\end{equation}
From this point on, all expressions of the form $\poly{\cdot}$ should be read as $\pl{d,\gamma}{L}{\cdot}$.

We prove the lemma by induction on the dimension $d$.
 The case $d=1$ comes immediately from Lemma~\ref{lemma:1dimdich} and \eqref{eq:omeg_inv_poly}\jl{; we just note that although $N_0$ does not appear explicitly in the lemma, it may be taken equal to $2$ there, since the lemma applies to any $N\geq 2$}.
 Assume now that $d\geq 2$ and that the statement of the lemma is valid for all smaller dimensions. For a parameter $\theta=\theta(d,\omega,\varepsilon)$, which in view of \eqref{eq:theta_bound} \jl{and \eqref{eq:omeg_inv_poly}} may be taken of the form $\theta=\poly{\varepsilon}$, the proof of Lemma~\ref{lemma:dichotomy} establishes the following sufficient conditions on the parameters $\varphi$, $N_{0}$ and $M$:
	\begin{linenomath}
	\begin{align}
	&0<\varphi<\frac{1}{2}\varphi(d-1,\omega,\theta)\label{eq:varphi}\\
	&N_{0}\geq N_{0}(d-1,\omega,\theta,c),\quad N_{0}\geq \frac{1}{\omega^{-1}\left(\frac{\varphi\omega^{-1}(c)}{8\omega(c)}\right)}, \quad \hl{N_{0}\geq \frac{6}{\varepsilon},}\label{eq:N}\\
	&M\in M_{d-1}\Z\qquad M\geq\frac{1}{\omega^{-1}\left(\frac{1}{NM_{d-1}}\right)},\quad\text{with $M_{d-1}:=M(N,d-1,\omega,\theta,c)$.}\label{eq:M}
	\end{align}
	\end{linenomath}
	We argue that these conditions are satisfied for a choice of $\varphi$, $N_{0}$, $M$ and $M_{0}$ of the form~\eqref{eq:parameters}. From the induction hypothesis and $\theta=\poly{\varepsilon}$, it is clear that \eqref{eq:varphi} is satisfied for an appropriate choice of $\varphi\jl{=\varphi(d,\omega,\theta)}=\poly{\varepsilon}$.	
	We fix $\varphi$ accordingly. Similarly, the induction hypothesis and $\theta=\poly\varepsilon$ ensure that the first inequality of \eqref{eq:N} may be satisfied by a choice of $N_{0}$ of the form of \eqref{eq:parameters}. We verify that such a choice may additionally satisfy the second \hl{and the third} inequality of \eqref{eq:N}.
To this end, \hl{we use \eqref{eq:omega_upper_bound} and \eqref{eq:strong_lb_omeg_inv} to derive}
\begin{linenomath}
	\begin{equation}\label{eq:w_inv_over_w}
	\frac{\omega^{-1}(c)}{\omega(c)}\geq \frac{1}{\poly{\log\frac{1}{c}}}.
	\end{equation}
\end{linenomath}
	We apply this bound, \hl{$\varphi=\poly{\varepsilon}$ and \eqref{eq:omeg_inv_poly}} to derive
	\begin{linenomath}
	\begin{equation*}
	\frac{1}{\omega^{-1}\left(\frac{\varphi\omega^{-1}(c)}{8\omega(c)}\right)}\leq \frac{\poly{\log\frac{1}{c}}}{\poly{\varepsilon}}.
	\end{equation*}
	\end{linenomath}
	Hence, an appropriate choice of $N_{0}$ of the form~\eqref{eq:parameters} satisfies \hl{all three} inequalities of \eqref{eq:N}. We fix such an $N_{0}$ and show now that the choice of $M_{0}$ is possible. It is first necessary to consider the parameter $M$. By the induction hypothesis, the first condition of \eqref{eq:M} may clearly be satisfied by a choice of $M$ of the form \eqref{eq:parameters}. To justify that the second part of \eqref{eq:M} may also be satisfied by such a choice, we note the bound	
	\begin{linenomath}
	\begin{equation*}
	\frac{1}{\omega^{-1}\left(\frac{1}{NM_{d-1}}\right)}\leq \frac{1}{\omega^{-1}\left(\frac{\poly{\varepsilon}}{\poly{N}}\right)}\leq \frac{\poly{N}}{\poly{\varepsilon}}.
	\end{equation*}
	\end{linenomath}
Finally, given that $M$ may be chosen of the form $M=\frac{\poly{N}}{\poly{\varepsilon}}$, it follows that a sufficient condition on $M_{0}$ is given by $M_{0}\geq \frac{\poly{N_{0}}}{\poly{\varepsilon}}$. The choice of $N_{0}$ then allows us to take $M_{0}=\frac{\poly{\log\frac{1}{c}}}{\poly{\varepsilon}}$. 
\end{proof}

\begin{lemma}\label{lemma:c_i_bound}
	Let $\gamma>0$, $L\geq 1$ and $\omega\in \mc{M}$ satisfy
	\begin{linenomath}
	\begin{equation*}
	\omega(t)\leq  Lt\left(\log\frac{1}{t}\right)^\gamma,\qquad \text{for all $t\in(0,a_{\omega})$}.
	\end{equation*}
	\end{linenomath}
	\hl{Let $d\in\N$ and $c,\varepsilon\in(0,a_{\omega})$.}
	Then the following inequality holds for all $i\geq 1$:
	\begin{linenomath}
	\begin{equation}\label{eq:strong_ci_bound}
	c_{i}\geq \br*{\hl{\pl{d,\gamma}{L}{\varepsilon}}c}^{i^{2}},
	\end{equation}
	\end{linenomath}
\hl{where $c_{i}$ is defined according to Definition~\ref{def_par_subfam}}. In particular, the function \hl{$\pl{d,\gamma}{L}{\varepsilon}$} is independent of $i$. 
\end{lemma}
\begin{proof}
	 By Lemma~\ref{lemma:parameters}, there are functions $\beta(\varepsilon),\overline{\beta}(\varepsilon)$, \hl{ both having the form $\pl{d,\gamma}{L}{\varepsilon}$,}
	 and a number $q>0$, depending only on $d, L$ and $\gamma$ such that
\begin{linenomath}
\begin{equation}\label{eq:ci_recursion}
c_{i+1}=\frac{c_{i}}{N_{i}M_{i}}\geq\frac{\overline{\beta}(\varepsilon)c_{i}}{\left(\log\frac{1}{c_{i}}\right)^{q}}\geq \beta(\varepsilon)c_{i}^{2},\qquad i\geq 0,
\end{equation}
\end{linenomath}
    where the last inequality is achieved by choosing the polynomial $\beta$ carefully enough.
    A sufficient choice is to set $\beta(\varepsilon):=\min\set{1,\br*{\frac{e}{q}}^q}\cl{\beta}(\varepsilon)$. Additionally, $\cl{\beta}$ is chosen so that $\cl{\beta}(\varepsilon)\in(0,1)$ for every $\varepsilon\in(0,a_\omega)\subseteq(0,\frac{1}{2})$. Applying \eqref{eq:ci_recursion} inductively
    yields the bound 
	\begin{linenomath}
	\begin{equation}\label{eq:weak_ci_bound}
	c_{i}\geq \beta(\varepsilon)^{2^{i-1}-1}c^{2^{i-1}}, \qquad\hl{i\geq 1.}
	\end{equation}
	\end{linenomath}    
    We use the recursion \eqref{eq:ci_recursion} to derive
	\begin{linenomath}
	\begin{equation*}
	c_{i}\geq\frac{\overline{\beta}(\varepsilon)^{i-1}c}{\prod_{j=1}^{i-1}\left(\log\frac{1}{c_{j}}\right)^{q}},\qquad i\geq 1.
	\end{equation*}
	\end{linenomath}
	Bounding each $c_{j}$ term in the denominator below by $c_{i}$, and applying a weaker form of the inequality of \eqref{eq:weak_ci_bound}, namely
	\begin{linenomath}
	\begin{equation*}
	c_{i}\geq (\beta(\varepsilon)c)^{2^{i}},\qquad i\geq 1,
	\end{equation*}
	\end{linenomath}
	we obtain
	\begin{linenomath}
	\begin{equation*}
	c_{i}\geq\frac{\overline{\beta}(\varepsilon)^{i-1}\jl{c}}{\left(\log\frac{1}{c_{i}}\right)^{qi}}\geq  \frac{\overline{\beta}(\varepsilon)^{i-1}c}{\left(2^{i}\log\left(\frac{1}{\beta(\varepsilon)c}\right)\right)^{qi}},\qquad i\geq 1.
	\end{equation*}
	\end{linenomath}
	\hl{Observe that there is $\Lambda(q)>0$ such that $\left(\log\frac{1}{t}\right)\leq \Lambda(q) t^{-\frac{1}{q}}$ for every $t\in (0,\infty)$. Together with the inequality above, this implies \eqref{eq:strong_ci_bound}.}
\end{proof}
For $i\in[r]$ the tiled family of cubes $\Sq_{i}$ fulfilling the assertions of Lemma~\ref{lemma:geometric} will be defined as a subfamily of $\mc{Q}_{c_{i}/N_{i}}$. From the previous lemma we immediately obtain a lower bound on the sidelength $\frac{c_{i}}{N_{i}}$ of cubes in $\mc{Q}_{c_i/N_i}$.
\begin{cor}\label{cor:sidelength}
Let $\gamma>0$, $L\geq 1$ and $\omega\in \mc{M}$ satisfy
\begin{linenomath}
\begin{equation*}
\omega(t)\leq  Lt\left(\log\frac{1}{t}\right)^\gamma,\qquad \text{for all $t\in(0,a_{\omega})$}.
\end{equation*}
\end{linenomath}
\hl{Let $d\in\N$, $c,\varepsilon\in(0,a_{\omega})$ and the parameters $(c_{i})_{i=1}^{\infty}$, $(N_{i})_{i=1}^{\infty}$ be given by Definition~\ref{def_par_subfam}.} Then 
\begin{linenomath}
$$
\operatorname{sidelength}(S)=\frac{c_i}{N_i}\geq c_{i+1}\geq \left(\hl{\pl{d,\gamma}{L}{\varepsilon}}\cdot c\right)^{(i+1)^2}
$$
\end{linenomath}
for all cubes $S\in\Sq_{i}\subseteq \mc{Q}_{c_{i}/N_{i}}$.
\end{cor}
\begin{lemma}\label{lemma:rbound}
	Let $\gamma>0$, $L\geq 1$ and $\omega\in \mc{M}$ satisfy
	\begin{linenomath}
	\begin{equation*}
	\omega(t)\leq  Lt\left(\log\frac{1}{t}\right)^\gamma,\qquad \text{for all $t\in(0,a_{\omega})$}.
	\end{equation*}
	\end{linenomath}
	\hl{Let $d\in\N$ and $c,\varepsilon\in(0,a_{\omega})$.}
	Then $\omega\in\M_{0}$, and moreover, the parameter $r(d,\omega,\varepsilon,c)$ of Definition~\ref{def_par_subfam} witnessing this may be taken of the form
	\begin{linenomath}
	\begin{equation*}
	r(d,\omega,\varepsilon,c)=\frac{1}{c\cdot \hl{\pl{d,\gamma}{L}{\varepsilon}}}.
	\end{equation*}
	\end{linenomath}
\end{lemma}
\begin{proof}
We consider $d, L$ and $\gamma$ fixed. All terms in the subsequent calculation will depend implicitly on them and this dependence will no longer be mentioned explicitly. \hl{Moreover, all occurences of $\poly{\cdot}$ in the present proof stand for $\pl{d,\gamma}{L}{\cdot}$.}
Thus, as shown in Lemma~\ref{lemma:parameters}, we may write $\varphi=\poly{\varepsilon}$. 

We wish to find minimal $r:=r(d,\omega,\varepsilon,c)$ such that for every $i\geq r$ the following holds:
\begin{linenomath}
\begin{align}\label{eq:rreq}
\begin{split}
\frac{(1+\varphi)^{i}\omega^{-1}(c)}{c}&=\frac{\omega^{-1}(c)}{c}\br{1+\poly{\varepsilon}}^i\geq \frac{\omega(c_{i+1})}{c_{i+1}}
\end{split}
\end{align}
\end{linenomath}
By Lemma~\ref{lemma:c_i_bound}, we can bound $c_{i+1}\geq \br*{\poly{\varepsilon}c}^{\hl{(i+1)}^2}$. 
We emphasise particularly that the expression $\poly{\varepsilon}$ is independent of $i$. Using this bound together with the bound $\frac{\omega^{-1}(c)}{c}\geq\frac{1}{\poly{\log(1/c)}}$ of \eqref{eq:w_inv_over_w} and the hypothesis $\frac{\omega(t)}{t}\leq L\left(\log\frac{1}{t}\right)^{\gamma}$ we see that \eqref{eq:rreq} is implied by the inequality
\begin{linenomath}
\begin{equation*}
\frac{1}{\poly{\log\frac{1}{c}}}\br{1+\poly{\varepsilon}}^i\geq \br*{\log\br*{(c\poly{\varepsilon})^{-\hl{(i+1)}^2}}}^\gamma.
\end{equation*}
\end{linenomath}
This inequality can be rewritten as
\begin{linenomath}
\begin{equation*}
i\log(1+\poly{\varepsilon})\geq 2\gamma\log \hl{(i+1)}+\gamma\log\log\frac{1}{c\poly{\varepsilon}}+\Lambda{\log\log\frac{1}{c}}
\end{equation*}
\end{linenomath}
and using the bound $2\gamma\log \hl{(i+1)}\leq \Lambda\sqrt i$ one can easily see that $i$ of size at least $\frac{\Lambda}{c\cdot \br*{\log\br*{1+\poly{\varepsilon}}}^2}$  satisfies the inequality.

We note that $\log\br{1+\poly{\varepsilon}}$ behaves as $\poly{\varepsilon}$ as $\varepsilon$ goes to zero. This yields the desired upper bound on $r$ of the form
\begin{linenomath}
$$
r\leq \frac{1}{c\cdot \poly{\varepsilon}}.\eqno\qedhere
$$
\end{linenomath}
\end{proof}

\subsection{A volume bound.}
\short In the present subsection we state a volume bound on the difference of images of two bi-$\omega$-mappings which are close with respect to the $\lnorm{\infty}{-}$ distance.
\fi\normalcolor
\arxiv The present subsection is devoted to establishing a volume bound on the difference of images of two bi-$\omega$-mappings which are close with respect to the $\lnorm{\infty}{-}$ distance. 
\fi\normalcolor
This will allow us to derive statement~\ref{lemma:geometric2} of Lemma~\ref{lemma:geometric} from the first conclusion~\ref{dich1} of the dichotomy of Lemma~\ref{lemma:dichotomy}. 

\begin{lemma}\label{lemma:volume}
	Let $d\in \N$, $\omega\in \mc{M}$, $c,\varepsilon\in(0,a_{\omega})$, $N\in\N$, $i\in[N-1]$ and $h\colon [0,c]\times[0,c/N]^{d-1}\to\R^{d}$ be a bi-$\omega$-mapping with $\bilipomeg(h)\leq 1$. Suppose that $h$ satisfies inequality \eqref{eq:translation} on $S_{i}:=\left[\frac{(i-1)c}{N},\frac{ic}{N}\right]\times\left[0,\frac{c}{N}\right]^{d-1}$. Then
	\begin{linenomath}
		\begin{equation*}
			\left|\leb (h(S_{i}))-\leb (h(S_{i+1}))\right|\leq \Lambda(d)\frac{\omega(\omega(\varepsilon\omega(\ell(S_i))))}{\ell(S_i)}\left(\frac{\omega(\varepsilon)\omega(\ell(S_i))}{\varepsilon\ell(S_i)}\right)^{d-1}\leb(S_i),
		\end{equation*}
	\end{linenomath}
	where $\Lambda(d)>0$ is a constant depending on $d$ only.
\end{lemma}
The appearance of $\Lambda(d)$ above should be interpreted as explained in the Notation paragraph at the beginning of Section~\ref{section:geometric}; it is a general purpose constant whose exact value is irrelevant for the rest of the paper.

\short
\begin{proof}
	If, in addition to all the conditions in the statement, it holds that $\ell(S_i)>2\omega(\varepsilon\omega(\ell(S_i)))$, it is enough to follow \cite[Proof of Lem.~3.4 and B.2, p.~611, p.~639]{DKK2018} only replacing all Lipschitz estimates in \cite{DKK2018} with $\omega$-continuity estimates; thus, we omit the details. This yields the bound
	\begin{linenomath}
		\begin{equation*}
			\left|\leb (h(S_{i}))-\leb (h(S_{i+1}))\right|\leq \Lambda(d)\left(\frac{\omega(\omega(\varepsilon\omega(\ell(S_i))))^{d}}{\ell(S_i)\omega(\varepsilon\omega(\ell(S_i)))^{d-1}}\right)\leb(S_i).
		\end{equation*}
	\end{linenomath}
	This bound is stronger than the bound claimed in the statement of Lemma~\ref{lemma:volume}, which follows from the inequalities
	\[
	\frac{\omega(\omega(\varepsilon\omega(\ell(S_i))))}{\omega(\varepsilon\omega(\ell(S_i)))}\leq\frac{\omega(\varepsilon\omega(\ell(S_i)))}{\varepsilon\omega(\ell(S_i))}\leq\frac{\omega(\varepsilon)\omega(\omega(\ell(S_i)))}{\varepsilon\omega(\ell(S_i))}\leq\frac{\omega(\varepsilon)\omega(\ell(S_i))}{\varepsilon\ell(S_i)}.
	\]
	The first and the third inequalities in the chain above are applications of the inequality $\omega(\omega(t))=\omega\left(\frac{\omega(t)}{t}\cdot t\right)\leq\frac{\omega(t)}{t}\cdot\omega(t)$, which holds due to the concavity of $\omega$ and the fact that $\omega(t)\geq t$. The second inequality in the chain is due to the submultiplicativity of $\omega$.
	
	In the case that $\ell(S_i)\leq 2\omega(\varepsilon\omega(\ell(S_i)))$, it is enough to use the trivial bound
	\begin{linenomath}
		\begin{equation*}
			\left|\leb (h(S_{i}))-\leb (h(S_{i+1}))\right|\leq \Lambda(d)\omega(\ell(S_i))^d=\Lambda(d)\left(\frac{\omega(\ell(S_i))}{\ell(S_i)}\right)^d\leb(S_i),
		\end{equation*}
	\end{linenomath}
	the inequality $\omega(\varepsilon)\geq \varepsilon$ and the fact that, in the present case, $\omega(\ell(S_i))\leq 2\omega(\omega(\varepsilon\omega(\ell(S_i))))$.
\end{proof}
\fi\normalcolor
\arxiv

The proof of Lemma~\ref{lemma:volume} is a simple modification of that of \cite[Lem.~3.4]{DKK2018}. It is based on the following basic fact.
\begin{lemma}\label{lemma:bilcubvol}
	Let $\omega\in\mc{M}$, $0<\lambda<a_{\omega}$, $S\in \mathcal{Q}^{d}_{\lambda}$ and $f_{1},f_{2}\colon S\to \R^{d}$ be bi-$\omega$-mappings with $\bilipomeg(f_{i})\leq 1$ for $i=1,2$. Let $\varepsilon\in (0,1)$ and suppose that $2\omega\left(\varepsilon\omega(\lambda)\right)<\lambda$ and
	\begin{linenomath}
		\begin{equation}\label{eq:translation2}
			\left\|f_{2}(\mb{x})-f_{1}(\mb{x})\right\|_{\infty}\leq\varepsilon\omega(\lambda).
		\end{equation}
	\end{linenomath}
	Then
	\begin{linenomath}
		\begin{equation*}
			\abs{\leb (f_{1}(S))-\leb (f_{2}(S))}\leq \Lambda(d)\left(\frac{\omega(\omega(\varepsilon\omega(\lambda)))^{d}}{\lambda\omega(\varepsilon\omega(\lambda))^{d-1}}\right)\leb(S).
		\end{equation*}
	\end{linenomath}
\end{lemma}
\begin{proof}
	For a set $A\subseteq \R^{d}$ and $t>0$ we introduce the set
	\begin{linenomath}
		\begin{equation*}
			[A]_{t}:=\left\{\mb{x}\in A\colon \dist(\mb{x},\partial A)\geq t\right\}
		\end{equation*}
	\end{linenomath}
	of all points in the interior of $A$, whose distance to the boundary of $A$ is at least $t$.
	Using \eqref{eq:translation2} and the $\omega^{-1}$ bound on $f_{2}$ we deduce that
	\begin{linenomath}
		\begin{equation*}
			f_{1}([S]_{t})\subseteq\overline{B}(f_{2}([S]_{t}),\varepsilon\omega(\lambda))\subseteq \overline{B}([f_{2}(S)]_{\omega^{-1}(t)},\varepsilon\omega(\lambda))
		\end{equation*}
	\end{linenomath}
	for all $t>0$. For the second inclusion, we use Brouwer's Invariance of Domain~\cite[Thm.\ 2B.3]{Hatcher} in order to prove $f_{2}([S]_{t})\subseteq [f_{2}(S)]_{\omega^{-1}(t)}$. It follows that
	\begin{linenomath}
		\begin{equation*}\label{eq:ginc}
			f_{1}([S]_{\omega(\varepsilon\omega(\lambda))})\subseteq f_{2}(S).
		\end{equation*}
	\end{linenomath}
	Therefore
	\begin{linenomath}
		\begin{align*}
			\leb (f_{1}(S))-\leb (f_{2}(S))\leq \leb (f_{1}(S\setminus [S]_{\omega(\varepsilon\omega(\lambda))}))
		\end{align*}
	\end{linenomath}
	Because of the inequality $2\omega\left(\varepsilon\omega(\lambda)\right)<\lambda$, the set $[S]_{\omega(\varepsilon\omega(\lambda))}$ is non-empty and the set $S\setminus[S]_{\omega(\varepsilon\omega(\lambda))}$ can be covered by $\frac{\Lambda(d)\lambda^{d-1}}{\omega(\varepsilon\omega(\lambda))^{d-1}}$ cubes of side length $\omega(\varepsilon\omega(\lambda))$. Using the concavity of $\omega$, we get that the image of each of these cubes under $f_{1}$ is contained in a ball of radius $\sqrt{d}\omega(\omega(\varepsilon\omega(\lambda)))$.
	Thus, the total measure of $f_{1}(S\setminus [S]_{\omega(\varepsilon\omega(\lambda))})$ is at most
	\begin{linenomath}
		\begin{equation*}
			\Lambda(d)\left(\frac{\lambda}{\omega(\varepsilon\omega(\lambda))}\right)^{d-1}\omega(\omega(\varepsilon\omega(\lambda)))^{d}.
		\end{equation*}
	\end{linenomath}
	We conclude that
	\begin{linenomath}
		\begin{multline*}
			\leb (f_{1}(S))-\leb (f_{2}(S))\leq \Lambda(d)\left(\frac{\lambda}{\omega(\varepsilon\omega(\lambda))}\right)^{d-1}\omega(\omega(\varepsilon\omega(\lambda)))^{d}\\
			=\Lambda(d)\left(\frac{\omega(\omega(\varepsilon\omega(\lambda)))^{d}}{\lambda\omega(\varepsilon\omega(\lambda))^{d-1}}\right)\leb(S).
		\end{multline*}
	\end{linenomath}
	Since the above argument is completely symmetric with respect to $f_{1}$ and $f_{2}$, we also have
	\begin{linenomath}
		\begin{equation*}
			\leb (f_{2}(S))-\leb (f_{1}(S))\leq \Lambda(d)\left(\frac{\omega(\omega(\varepsilon\omega(\lambda)))^{d}}{\lambda\omega(\varepsilon\omega(\lambda))^{d-1}}\right)\leb(S).\qedhere
		\end{equation*}
	\end{linenomath}
\end{proof}

We can now prove Lemma~\ref{lemma:volume}.
\begin{proof}[Proof of Lemma~\ref{lemma:volume}]
	In the case that $\ell(S_i)\leq 2\omega(\varepsilon\omega(\ell(S_i)))$, it is enough to use the trivial bound
	\begin{linenomath}
		\begin{equation*}
			\left|\leb (h(S_{i}))-\leb (h(S_{i+1}))\right|\leq \Lambda(d)\omega(\ell(S_i))^d=\Lambda(d)\left(\frac{\omega(\ell(S_i))}{\ell(S_i)}\right)^d\leb(S_i),
		\end{equation*}
	\end{linenomath}
	the inequality $\omega(\varepsilon)\geq \varepsilon$ and the fact that, in the present case, $\omega(\ell(S_i))\leq 2\omega(\omega(\varepsilon\omega(\ell(S_i))))$.
	
	In the remaining case we have $\ell(S_i)> 2\omega(\varepsilon\omega(\ell(S_i)))$. To proceed, we define a translation $\phi\colon h([0,c]\times[0,c/N]^{d-1})\to\R^{d}$ by
	\begin{linenomath}
		\begin{equation*}
			\phi(h(\mb{x})):=h(\mb{x})+\frac{1}{N}(h(c\mb{e}_{1})-h(\mb{0})), \qquad \mb{x}\in[0,c]\times[0,c/N]^{d-1}.
		\end{equation*}
	\end{linenomath}
	Let the mappings $f_{1}\colon S_{i}\to\R^{d}$, $f_{2}\colon S_{i}\to\R^{d}$ be defined by $f_{1}:=\phi\circ h$ and $f_{2}(\mb{x}):=h(\mb{x}+\frac{c}{N}\mb{e}_{1})$. Then $f_{1},f_{2}$ are both bi-$\omega$-mappings of the cube $S_{i}\in\mathcal{Q}_{c/N}^{d}$ which satisfy $\lnorm{\infty}{f_{1}-f_{2}}\leq\varepsilon \omega\left(\frac{c}{N}\right)$, due to \eqref{eq:translation}. Moreover, the inequality of the present case is precisely the condition $2\omega\left(\varepsilon\omega(\lambda)\right)<\lambda$ of Lemma~\ref{lemma:bilcubvol} for $\lambda=\frac{c}{N}=\ell(S_{i})$. Applying Lemma~\ref{lemma:bilcubvol} and the identities $\leb(f_{1}(S_{i}))=\leb(h(S_{i}))$ and $f_{2}(S_{i})=h(S_{i+1})$, we get that
	\begin{linenomath}
		\begin{equation*}\label{eq:volumebound1}
			\left|\leb (\cancel{\phi(h(S_{i}))}h(S_{i}))-\leb (h(S_{i+1}))\right|\leq \Lambda(d)\left(\frac{\omega(\omega(\varepsilon\omega(\ell(S_i))))^{d}}{\ell(S_i)\omega(\varepsilon\omega(\ell(S_i)))^{d-1}}\right)\leb(S_i).
		\end{equation*}
	\end{linenomath}
	This bound is stronger than the bound claimed in the statement of Lemma~\ref{lemma:volume}, which follows from the inequalities
	\[
	\frac{\omega(\omega(\varepsilon\omega(\ell(S_i))))}{\omega(\varepsilon\omega(\ell(S_i)))}\leq\frac{\omega(\varepsilon\omega(\ell(S_i)))}{\varepsilon\omega(\ell(S_i))}\leq\frac{\omega(\varepsilon)\omega(\omega(\ell(S_i)))}{\varepsilon\omega(\ell(S_i))}\leq\frac{\omega(\varepsilon)\omega(\ell(S_i))}{\varepsilon\ell(S_i)}.
	\]
	The first and the third inequalities in the chain above are applications of the inequality $\omega(\omega(t))=\omega\left(\frac{\omega(t)}{t}\cdot t\right)\leq\frac{\omega(t)}{t}\cdot\omega(t)$, which holds due to the concavity of $\omega$ and the fact that $\omega(t)\geq t$. The second inequality in the chain is due to the submultiplicativity of $\omega$.
\end{proof}
\fi\normalcolor

\subsection*{Proof of Lemma~\ref{lemma:geometric_amalg}.}

Finally, we collect together the results of the present section to give a proof of Lemma~\ref{lemma:geometric_amalg}. For convenience, we divide Lemma~\ref{lemma:geometric_amalg} into two lemmas (\ref{lemma:geometric} and \ref{lemma:M1}), which we prove separately.
\begin{lemma}\label{lemma:geometric}
	Let $\alpha>0$ and $\omega\in\M$ satisfy
	\begin{linenomath}
		\begin{equation*}
		\omega(t)\leq t\left(\log\frac{1}{t}\right)^{\alpha},\qquad \text{for all $t\in(0,a_{\omega})$}. 
		\end{equation*}
	\end{linenomath}
	Let $d,k\in\N$, $d\geq 2$, \hl{$c,\varepsilon\in(0,a_{\omega})$} and $L\geq 1$. Then there exists $r=r(d,L\sqrt{k}\omega,\varepsilon,c)\in\N$ such that for every non-empty open ball $U\subseteq\R^{d}$ of radius at least $2c\sqrt{d}$ there exist finite tiled families $\Sq_{1},\Sq_{2},\ldots,\Sq_{r}$ of cubes contained in $U$ with the following properties:
	\begin{enumerate}
		\item\label{lemma:geometric1} For each $1\leq i<r$ and each cube $S\in\Sq_{i}$
		\begin{linenomath}
			\begin{equation*}
			\leb \Biggl(S\cap\bigcup_{j=i+1}^{r}\bigcup\Sq_{j}\Biggl)\leq \poly{\varepsilon}\leb (S). 
			\end{equation*}
		\end{linenomath}
		\item\label{lemma:geometric2} For any $k$-tuple $(h_{1},\ldots,h_{k})$ of bi-$\omega$-mappings $h_{j}\colon U\to\R^{d}$ for which $\max \mathfrak{L}_{\omega}(h_{j})\leq L$ there exist $i\in[r]$ and $\mb{e}_{1}$-adjacent cubes $S,S'\in \Sq_{i}$ such that
		\begin{linenomath}
			\begin{equation*}
			\frac{\abs{\leb(h_j(S))-\leb(h_j(S'))}}{\leb(S)}\leq \upsilon(d,\omega,L,k,\varepsilon,\ell(S))
			\end{equation*}
		\end{linenomath}
		for all $j\in[k]$, where 
		\begin{linenomath}		
			\begin{equation*}
			\upsilon(d,\omega,L,k,\varepsilon,\ell(S)):=
			\Lambda(d,L,k)\frac{\omega\br*{\omega\br*{\varepsilon\omega(\ell(S))}}}{\ell(S)}\left(\frac{\omega(\varepsilon)\omega(\ell(S))}{\varepsilon\ell(S)}\right)^{d-1}
			\end{equation*}
		\end{linenomath}
	and $\Lambda(d,L,k)>0$ is a constant depending only on $d,L$ and $k$.
	\end{enumerate}
\end{lemma}
The behaviour of the right-hand side of the inequality in statement~\ref{lemma:geometric2} depends on $\omega$; the statement is most powerful for those moduli $\omega$, for which the expression goes to zero with $\varepsilon$. The work of \cite{BK1} (see also \cite{DKK2018}) implies that for Lipschitz moduli, i.e., those that satisfy $\omega(x)\leq Lx$, $L\geq 1$, it indeed goes to zero. On the other hand, it follows from the work of Rivi\`ere and Ye~\cite[Thm.~1]{RY}
that for any $\alpha<1$ the expression \textit{cannot} go to zero for any $\omega(x)\geq x^\alpha$, i.e., for Hölder moduli of continuity. This is because otherwise one could use a~construction similar to that of Theorem~\ref{thm:non_rlz_dsty} to construct continuous Hölder non-realisable densities, which, however, do not exist by \cite{RY} (see also McMullen~\cite[Sec.~5]{McM}). 
We will show that the right-hand side of the inequality in statement~\ref{lemma:geometric2} converges to zero for some moduli lying strictly between the Lipschitz and the Hölder moduli of continuity. 

Note that in the parameter $r$ of Lemma~\ref{lemma:geometric} we consider the modulus $L\sqrt{k}\omega$ instead of $\omega$. This is because we will view the $k$-tuple $h_{1},\ldots,h_{k}$ as a single mapping $(h_{1},\ldots,h_{k})\colon \R^{d}\to\R^{kd}$ and this single mapping has modulus of continuity $L\sqrt{k}\omega$.

The following proof is an easy adaptation of the proof of \cite[Lem.~3.1]{DKK2018}.
\begin{proof}[Proof of Lemma~\ref{lemma:geometric}]
	Let $\overline{\omega}(t):=L\sqrt{k}\omega(t)$ with $a_{\overline{\omega}}=a_\omega$. Then $\overline{\omega}$ clearly belongs to $\mc{M}$. Moreover, by Lemma~\ref{lemma:rbound} we have $\overline{\omega}\in\mc{M}_{0}$.
	Let the sequences $(N_{i})_{i=1}^{\infty}$, $(M_{i})_{i=1}^{\infty}$, $(c_{i})_{i=1}^{\infty}$ and the number $r=r(d,\overline{\omega},\varepsilon,c)\in\N$ be defined according to Definition~\ref{def_par_subfam}, with all these values using $\overline{\omega}$ instead of $\omega$. Let $U\subseteq \R^{d}$ be an open ball of radius at least $2c\sqrt{d}$. Since the conclusion of Lemma~\ref{lemma:geometric} is invariant under translation of the set $U\subseteq \R^{d}$, we may assume that $B(\mb{0},2c\sqrt{d})\subseteq U$ so that
	\begin{linenomath}
		\begin{equation*}
		[0,c]\times[0,c/N_{1}]^{d-1}\subseteq U.
		\end{equation*} 
	\end{linenomath}
	We are now ready to define the families of cubes $\Sq_{1},\ldots,\Sq_{r}$, making use of the sequences $(N_{i})_{i=1}^{\infty}$ and $(c_{i})_{i=1}^{\infty}$.
	\begin{define}\label{construction}
		For each $i\in[r]$ we define the family $\Sq_{i}\subseteq \mathcal{Q}_{c_{i}/N_{i}}$ as the collection of all cubes of the form
		\begin{linenomath}
			\begin{equation*}
			\br*{\sum_{j=1}^{i}\mb{z}_{j}}+\hl{\Biggl(}\left[\frac{(l-1)c_{i}}{N_{i}},\frac{lc_{i}}{N_{i}}\right]\times\left[0,\frac{c_{i}}{N_{i}}\right]^{d-1}\hl{\Biggr)}
			\end{equation*}
		\end{linenomath}
		where $\mb{z}_{1}:=\mb{0}$ and $\mb{z}_{j+1}\in c_{j+1}\Z^{d}\cap\hl{\Big(}[0,c_{j}-c_{j+1}]\times[0,\frac{c_{j}}{N_{j}}-c_{j+1}]^{d-1}\hl{\Big)}$ for each $j\geq 1$ and \hl{$l\in[N_{i}]$}. 
	\end{define}

	Let us verify that the above defined families $\Sq_{1},\ldots,\Sq_{r}$ satisfy condition~\ref{lemma:geometric1} in the statement of Lemma~\ref{lemma:geometric}. 
	It is immediate from Definition~\ref{construction} \hl{and the definiton of $c_{i}$ in Definition~\ref{def_par_subfam}} that
	\begin{linenomath}
		\begin{equation*}\label{eq:nested}
		\bigcup \Sq_{r}\subseteq \bigcup \Sq_{r-1}\subseteq\ldots\subseteq \bigcup\Sq_{1}.
		\end{equation*}
	\end{linenomath}
	Thus, given $1\leq i<r$ and $S\in\Sq_{i}$, we have that 
	\begin{linenomath}
		\begin{equation*}
		S\cap\bigcup_{j=i+1}^{r}\bigcup\Sq_{j}\subseteq S\cap \bigcup\Sq_{i+1}.
		\end{equation*}
	\end{linenomath}
	Therefore, computing the volume of the latter set comes down to counting the number of cubes in $\Sq_{i+1}$ that intersect $S$. For the simple counting argument required, we refer the reader to \cite[p.~613--614]{DKK2018}. \jl{There} the argument is given in a less general situation where $N_{i}=N_{i+1}=N$. It gives the bound
	\begin{linenomath}
		\begin{equation*}
		\leb \br*{S\cap \bigcup_{j=i+1}^{r}\bigcup\Sq_{j}}\leq\frac{(M_{i}+1)^{d}}{M_{i}^{d}N_{i+1}^{d-1}}\left(\frac{c_{i}}{N_{i}}\right)^{d}
		\leq 2^{d}\frac{1}{N_{i+1}}\leb(S)\leq\poly{\varepsilon}\leb (S),
		\end{equation*}
	\end{linenomath}
	where, for the latter two inequalities, we use $\leb (S)=(c_{i}/N_i)^{d}$ and the bound $N_{i+1}\geq\hl{\frac{1}{\poly{\varepsilon}}}$,
	\hl{which comes from Lemma~\ref{lemma:parameters} and $c_{i+1}\in (0,a_{\omega})\subseteq (0,\frac{1}{2})$}.
	Thus, statement~\ref{lemma:geometric1} is satisfied.
	
	Turning now to statement~\ref{lemma:geometric2}, we consider a $k$-tuple $(h_{1},\ldots,h_{k})$ of bi-$\omega$-mappings $h_{i}\colon U\to\R^{d}$ with $\max_{i\in[k]}\bilipomeg(h_{i})\leq L$.	
	We define a combined mapping $g\colon U\to \R^{kd}$ co-ordinate-wise by
	\begin{linenomath}
		\begin{equation*}
		g^{((i-1)d+j)}(\mb{x}):=h_{i}^{(j)}(\mb{x})
		\end{equation*}
	\end{linenomath}
	for $i\in[k]$ and $j\in[d]$. It is straightforward to verify that $g$ is a bi-$\omega$-mapping with $\bilipomeg(g)\leq L\sqrt{k}$.
	This, in turn, implies that $g$ is a bi-$\overline{\omega}$-mapping with $\bilipgen{\overline{\omega}}(g)\leq 1$.
	
	The conditions of Lemma~\ref{lemma:terminate} are now satisfied for $d$, $\overline{\omega}$, $\varepsilon$, $(N_{i})_{i=1}^{\infty}$, $(M_{i})_{i=1}^{\infty}$, $(c_{i})_{i=1}^{\infty}$, $n=kd$, $g\colon [0,c_1]\times[0,c_1/N_{1}]^{d-1}\to\R^{kd}$ and $r=r(d,\overline{\omega},\varepsilon,c)$.
	Let $p\in [r]$ and $\mb{z}_{1},\ldots,\mb{z}_{p}\in\R^{d}$ be given by the conclusion of Lemma~\ref{lemma:terminate}. Then statement~\ref{dich1} of Lemma~\ref{lemma:dichotomy} holds for the mapping $g_{p}\colon [0,c_{p}]\times[0,c_{p}/N_{p}]^{d-1}\to\R^{kd}$ defined by~\eqref{eq:g_p}. Let $\Omega\subseteq[N_{p}-1]$ be given by the assertion of Lemma~\ref{lemma:dichotomy}, statement~\ref{dich1} for $g_{p}$. The co-ordinate functions of the mapping $g_{p}\colon [0,c_{p}]\times[0,c_{p}/N_{p}]^{d-1}\to\R^{kd}$ are defined by
	\begin{linenomath}
		\begin{equation*}
		g_{p}^{((t-1)d+s)}(\mb{x})=g^{((t-1)d+s)}\br*{\mb{x}+\sum_{j=1}^{p}\mb{z}_{j}}=h_{t}^{(s)}\br*{\mb{x}+\sum_{j=1}^{p}\mb{z}_{j}}
		\end{equation*}
	\end{linenomath}
	for $t\in[k]$, $s\in[d]$. Therefore for each $i\in\Omega$ and each $h_{t,p}\colon [0,c_{p}]\times[0,c_{p}/N_{p}]^{d-1}\to\R^{d}$ defined by $h_{t,p}(\mb{x}):=h_{t}(\mb{x}+\sum_{j=1}^{p}\mb{z}_{j})$ for $t\in[k]$, we have that $h=h_{t,p}$ satisfies inequality~\eqref{eq:translation} on $S_{i}:=\left[\frac{(i-1)c_{p}}{N_{p}},\frac{ic_{p}}{N_{p}}\right]\times\left[0,\frac{c_{p}}{N_{p}}\right]^{d-1}$.
	
	We fix $i\in\Omega$. Then the conditions of Lemma~\ref{lemma:volume} are satisfied for $\overline{\omega}$, $\varepsilon$, $d$, $N=N_{p}$, $c=c_{p}$, $h=h_{t,p}$ for each $t\in[k]$ and $i$. Hence, for $\lambda_{p}:=\ell(S_{i})=\frac{c_{p}}{N_{p}}$ we have
\begin{linenomath}
	\begin{align}\label{eq:omega_bar_volume}
		\left|\leb (h_{t,p}(S_{i}))-\leb (h_{t,p}(S_{i+1}))\right|&\leq \Lambda(d)\frac{\overline{\omega}(\overline{\omega}(\varepsilon\overline{\omega}(\lambda_p)))}{\lambda_p}\left(\frac{\overline{\omega}(\varepsilon)\overline{\omega}(\lambda_p)}{\varepsilon\lambda_p}\right)^{d-1}\leb(S_i),
	\end{align}
\end{linenomath}
	which can be bounded above using the concavity of $\overline{\omega}$ by
	\begin{linenomath}
		$$
		\Lambda(d,L,k)\left(\frac{\omega(\omega(\varepsilon\omega(\lambda_p)))}{\lambda_p}\right)\left(\frac{\omega(\varepsilon)\omega(\lambda_p)}{\varepsilon\lambda_p}\right)^{d-1}\leb(S_i).
		$$
	\end{linenomath}
	Set $S=\sum_{j=1}^{p}\mb{z}_{j}+S_{i}$ and $S'=\sum_{j=1}^{p}\mb{z}_{j}+S_{i+1}$. It is clear upon reference to Definition~\ref{construction} that $S$ and $S'$ are $\mb{e}_{1}$-adjacent cubes belonging to the family $\Sq_{p}$. Moreover, we have $h_{t}(S)=h_{t,p}(S_{i})$ and $h_{t}(S')=h_{t,p}(S_{i+1})$ for all $t\in[k]$. Therefore $S$ and $S'$ verify statement~\ref{lemma:geometric2} of Lemma~\ref{lemma:geometric} for the $k$-tuple $(h_{1},\ldots,h_{k})$. This completes the proof of Lemma~\ref{lemma:geometric}.
\end{proof}

\begin{sloppypar}
The final lemma refers to the notation of Lemma~\ref{lemma:geometric}.
The parameter $\kappa$ below represents the weakest possible upper bound of statement~\ref{lemma:geometric2} in Lemma~\ref{lemma:geometric}.
We emphasise that Lemma~\ref{lemma:geometric} and Lemma~\ref{lemma:M1} together imply Lemma~\ref{lemma:geometric_amalg}.
\end{sloppypar}
\begin{lemma}\label{lemma:M1}
	Let $d\geq 2$. Then there is $\alpha_0=\alpha_0(d)>0$	
	such that for $\omega\in\M$ of the form
	\begin{linenomath}
		\begin{equation*}
		\omega(t)\leq t\left(\log\frac{1}{t}\right)^{\alpha_0},\qquad \text{for all $t\in(0,a_{\omega})$},
		\end{equation*}
	\end{linenomath}
and
	\begin{linenomath}
		\begin{equation*}
		\kappa(\varepsilon):=\kappa(d,\omega,L,k,\varepsilon, c):=\sup_{i\in [r],\,S\in\mc{S}_{i}} \upsilon(d,\omega,L,k,\varepsilon,\ell(S))
		\end{equation*}
	\end{linenomath}
	we have
	\begin{linenomath}
		\begin{equation*}
		\lim_{\varepsilon\to 0}\kappa(\varepsilon)=0.
		\end{equation*}
	\end{linenomath}
\end{lemma}
\begin{proof}
	Let $\alpha>0$ and $\omega\in\M$ satisfy
	\begin{linenomath}
		\begin{equation*}
		\omega(t)\leq t\left(\log\frac{1}{t}\right)^{\alpha},\qquad \text{for all $t\in(0,a_{\omega})$}.
		\end{equation*}
	\end{linenomath}
	It suffices to show that there is some threshold $\alpha_{0}(d)>0$ so that whenever $\alpha\leq \alpha_{0}(d)$ the expression
	$\kappa(d,\omega,L,k,\varepsilon,c)$ is bounded above by $\poly{\varepsilon}$, and thus, goes to zero with $\varepsilon$.
	The right-hand side of the inequality in statement~\ref{lemma:geometric2} of Lemma~\ref{lemma:geometric} reads asymptotically as
	\[
	\frac{\omega(\omega(\varepsilon\omega(\ell(S_i))))}{\ell(S_i)}\left(\frac{\omega(\varepsilon)\omega(\ell(S_i))}{\varepsilon\ell(S_i)}\right)^{d-1}.
	\]

	We start with the first term.
	\begin{linenomath}
		\begin{align*}
			\begin{split}
				\frac{\omega(\omega(\varepsilon\omega(\ell(S_i))))}{\ell(S_i)}&=\frac{\omega(\omega(\varepsilon\omega(\ell(S_i))))}{\omega(\varepsilon\omega(\ell(S_i)))}\cdot\frac{\omega(\varepsilon\omega(\ell(S_i)))}{\varepsilon\omega(\ell(S_i))}\cdot\frac{\varepsilon\omega(\ell(S_i))}{\ell(S_i)}\\
				&\leq\varepsilon\left(\log\frac{1}{\omega(\varepsilon\omega(\ell(S_i)))}\right)^\alpha\left(\log\frac{1}{\varepsilon\omega(\ell(S_i))}\right)^\alpha\left(\log\frac{1}{\ell(S_i)}\right)^\alpha.
			\end{split}
		\end{align*}
	\end{linenomath}
	Each of the logarithms is at most $\log\frac{1}{\varepsilon\ell(S_i)}$.
	
	The second term can be bounded above as
	\[
	\left(\frac{\omega(\varepsilon)\omega(\ell(S_i))}{\varepsilon\ell(S_i)}\right)^{d-1}\leq \left(\left(\log\frac{1}{\varepsilon}\right)^\alpha\left(\log\frac{1}{\ell(S_i)}\right)^\alpha\right)^{d-1}
	\leq\left(\log\frac{1}{\varepsilon\ell(S_i)}\right)^{2\alpha(d-1)}.
	\]
	Combining the two bounds above, we infer
	\begin{linenomath}
		\begin{align}\label{eq:bound1}
		\begin{split}
		\frac{\omega(\omega(\varepsilon\omega(\ell(S_i))))}{\ell(S_i)}\left(\frac{\omega(\varepsilon)\omega(\ell(S_i))}{\varepsilon\ell(S_i)}\right)^{d-1}\leq
		\varepsilon\br*{\log\frac{1}{\varepsilon\ell(S)}}^{(2d+1)\alpha}.
		\end{split}
		\end{align}
	\end{linenomath}
	
	Set $\gamma=1$ and take $\alpha\in(0,1]$. By Corollary~\ref{cor:sidelength}, every cube $S\in\bigcup_{i=1}^r\mc{S}_i$ satisfies
	\begin{linenomath}$$
		\ell(S)\geq \br*{c\hl{\pl{d}{L,k}{\varepsilon}}}^{(r+1)^2}.
		$$
	\end{linenomath}
	This means that $\varepsilon\ell(S)$ can be bounded below \hl{by} $\br*{c\hl{\pl{d}{L,k}{\varepsilon}}}^{(r+1)^2}$, too. By Lemma~\ref{lemma:rbound}, we have that 
	\begin{linenomath}
		\begin{equation*}
		r:=r(d,L\sqrt{k}\omega,\varepsilon,c)\leq\frac{1}{c\hl{\pl{d}{L,k}{\varepsilon}}}.
		\end{equation*}
	\end{linenomath}
	We emphasise that the $\hl{\pl{d}{L,k}{\varepsilon}}$ expressions above are independent of $\alpha\in(0,1]$ and $c$.
	Plugging these two bounds in the inequality~\eqref{eq:bound1}, we get that
	\begin{linenomath}
		\begin{align*}
		\frac{\omega(\omega(\varepsilon\omega(\ell(S_i))))}{\ell(S_i)}&\left(\frac{\omega(\varepsilon)\omega(\ell(S_i))}{\varepsilon\ell(S_i)}\right)^{d-1}\\
		&\leq\varepsilon\br*{\log\br*{\br*{c\hl{\pl{d}{L,k}{\varepsilon}}}^{-(r+1)^2}}}^{(2d+1)\alpha}\\
		&\leq\varepsilon\br*{(r+1)^2\log\br*{\frac{1}{c\hl{\pl{d}{L,k}{\varepsilon}}}}}^{(2d+1)\alpha}\\
		&\leq\Lambda(d,c)\varepsilon\br*{\hl{\pl{d}{L,k}{\varepsilon}}^{-1}}^{(2d+1)\alpha}
		\end{align*}
	\end{linenomath}
	for $\varepsilon>0$ small enough. The last expression vanishes as $\varepsilon>0$ goes to zero provided $\alpha>0$ is chosen smaller than some threshold determined solely by $d$.
\end{proof}

\begin{remark}
	In the present work we have not required the explicit dependence on $k$ of the quantities \hl{$\upsilon$} and $\kappa$ of Lemmas~\ref{lemma:geometric} and \ref{lemma:M1}. However, the authors envisage potential future applications in which this dependence becomes relevant, particularly in relation to the open question concerning the `Feige sequence' which we discuss in the next section. Therefore we wish to place on record in this remark how the quantities \hl{$\upsilon$} and $\kappa$ depend on $k$ and comment on the necessary modifications of the proof needed to extract this dependence.
	
	For the quantity \hl{$\upsilon$} of Lemma~\ref{lemma:geometric} we have
	\begin{linenomath}
	\begin{equation*}
	\hl{\upsilon}\leq \Lambda(d)\hl{\pl{d}{}{L\sqrt{k}}}\frac{\omega(\omega(\varepsilon\omega(\ell(S))))}{\ell(S)}\left(\frac{\omega(\varepsilon)\omega(\ell(S))}{\varepsilon\ell(S)}\right)^{d-1}
	\end{equation*}
	\end{linenomath}
	and for $\omega(t)\leq t\left(\log\frac{1}{t}\right)^{\alpha}$ with $\alpha\in (0,1]$ the quantity $\kappa$ of Lemma~\ref{lemma:M1} satisfies
	\begin{linenomath}
	\begin{equation*}
	\kappa\leq \Lambda(d,c)\hl{\pl{d}{}{L\sqrt{k}}}\varepsilon\br*{\ply{d}{\varepsilon}^{-1}}^{(2d+1)\alpha}.
	\end{equation*}
	\end{linenomath}
	
	To get the bound on \hl{$\upsilon$}, we only have to use the inequality $\omega(L\sqrt{k}t)\leq L\sqrt{k}\omega(t)$, which comes from the concavity of $\omega$, to extract $L\sqrt{k}$ from the argument of $\omega$ in \eqref{eq:omega_bar_volume}. Note that in \eqref{eq:omega_bar_volume} we have $\overline{\omega}=L\sqrt{k}\omega$.
	
	Let $\omega(t)\leq t\left(\log\frac{1}{t}\right)^{\alpha}$ for some $\alpha>0$. For the bound on $\kappa$ it is necessary to make the dependencies of the parameters $\varphi$, $N_{0}$, $M$ and $M_{0}$ on $L$ (we will eventually apply this with $L$ replaced by $L\sqrt{k}$) in Lemma~\ref{lemma:parameters} explicit. It is straightforward to verify that \eqref{eq:parameters} then becomes 
	\begin{linenomath}
		\begin{multline*}
		\varphi=\frac{\poly\varepsilon}{\poly{L}},\qquad N_{0}=\frac{\poly{\log\frac{1}{c}}\poly{L}}{\poly{\varepsilon}},\\
		M=\frac{\poly{N}\poly{L}}{\poly{\varepsilon}},\qquad M_{0}=\frac{\poly{\log\frac{1}{c}}\poly{L}}{\poly{\varepsilon}},
		\end{multline*}
	\end{linenomath}
\hl{where in the equations above and in the remaining discussion $\poly{\cdot}$ always stands for $\pl{d,\gamma}{}{\cdot}$.}
	Implementing the arguments of Subsection~\ref{subsec:largeness} with these more precise expressions for the parameters leads to more precise bounds for the quantities $r$ and $\ell(S)$ with $S\in \Sq_{i}$, namely
	\begin{linenomath}	
	\begin{equation*}
	r(d,L\omega,\varepsilon,c)\leq\frac{\poly{L}}{c\poly{\varepsilon}}=\poly{L}\cdot r(d,\omega,\varepsilon,c)
	\end{equation*}
	\end{linenomath}
	and
	\begin{linenomath}
	\begin{equation*}
	\ell(S)\geq\left(\frac{\poly{\varepsilon}c}{\poly{L}}\right)^{(i+1)^2}\qquad \text{for all $S\in\Sq_{i}$.}
	\end{equation*}
	\end{linenomath}	
	The bound above on $\kappa$ is obtained by implementing the proof of Lemma~\ref{lemma:M1} with the more detailed bounds on $r$ and $\ell(S)$ stated above.
	Note that at this stage we consider the quantity $r(d,L\sqrt{k}\omega,\varepsilon,c)$.
\end{remark}

\section{Discussion and open problems.}
We stated and proved our results only for certain special families of moduli of continuity.
However, fixing a density $\rho$, the only information about a modulus $\omega(t)$ that determines whether there is a bi-$\omega$ solution $f$ to the pushforward equation
\begin{linenomath}
	\begin{equation}\label{eq:pfeq}
	f_\sharp\rho\leb=\rest{\leb}{f(I^d)}
	\end{equation}
\end{linenomath}
is the rate of growth of $\frac{\omega(t)}{t}$ as $t$ goes to $0$: \jl{i}t is clear that whenever one can find a modulus $\omega'(t)$ for which there is a bi-$\omega'$ non-realisable density $\rho$ and, at the same time, there is $t_0>0$ such that $\omega(t)\leq\omega'(t)$ for every $t\in (0,t_0)$, then $\rho$ is bi-$\omega$ non-realisable as well.

The techniques presented here yield that a generic continuous function $\rho$ is bi-$\omega$ non-realisable for $\omega(t)=t\br*{\log(1/t)}^{\alpha_0}$, where $0<\alpha_0<1$ is very small and depending on $d$.
It seems unlikely to us that the same technique could be used to prove the existence of bi-$\omega$ non-realisable functions with respect to $\omega(t)=t\log(1/t)$, say.
One of the key reasons is that in order to argue that the expression $\kappa(\varepsilon)=\kappa(d,\omega, L, k, \varepsilon)$ from Lemma~\ref{lemma:M1} stays at least bounded, one would need a very good upper bound on $r(\varepsilon)$, namely something as good as $O\br*{1/\sqrt[d]{\varepsilon}}$, probably even better. But this seems out of the reach of the present technique, because the
bound on $r(\varepsilon)$ we can obtain must be of order $\Omega\br*{\frac{1}{\varphi(\varepsilon)}}$. And the best bound on $\varphi(\varepsilon)$ for the Lipschitz modulus $\omega(t)=t$, even in dimension $d=1$, that we could come up with is of order $\Theta(\varepsilon^3)$.
While it would obviously be possible to get tighter bounds on various parameters at several places, we believe that these improvements could at best provide a quantitative estimate on $\alpha_0$, which would be much less than $1$ and could not settle the case of $\alpha_0\geq 1$.
\begin{quest}
Are there any bi-$\omega$ non-realisable continuous functions $I^d\to (0,\infty)$ for $\omega(t)=t\log(1/t)$?
\end{quest}

Thanks to Lemma~\ref{lemma:discrete_to_cts}, a positive answer to the above question would immediately yield \jl{an $\omega$-irregular} separated net. The same would be true if one provided a bi-$\omega$ non-realisable function $\rho\in L^\infty(I^d)$ with both its essential infimum and supremum in $(0,\infty)$.

In the present article, we have verified existence of densities $\rho$ excluding bi-$\omega$ solutions of \eqref{eq:pfeq}; the next natural task is to exclude solutions in the much larger class of $\omega$-mappings. For the Lipschitz modulus of continuity $\omega(t)=t$, Kopeck\'a and the authors achieved this in \cite{DKK2018}. This led to a negative answer of a question of Feige~\cite[Quest.~2.12]{Matousek_open} (see also \cite[Quest.~1.1]{DKK2018}). The result may be stated precisely as follows. For a set $S\subseteq\Z^{d}$ containing precisely $n^{d}$ points for some $n\in\N$ let $L_{S}$ denote the best Lipschitz constant of any (bijective) mapping of $S$ \emph{onto} the regular grid $\set{1,\ldots,n}^{d}$. Put differently, let
	\begin{linenomath}
		\begin{equation*}
		L_{S}:=\inf\set{\lip(f)\middle|\, \text{$f\colon S\to\set{1,\ldots,n}^{d}$ is a bijection}}.
		\end{equation*}
	\end{linenomath}
	The main result (Theorem~1.2) of \cite{DKK2018} states that the sequence
	\begin{linenomath}
		\begin{equation*}
		C_{n}:=\sup\set{L_{S}\colon S\subseteq \Z^{d},\, \abs{S}=n^{d}}, \qquad n\in\N,
		\end{equation*}
	\end{linenomath}
	is unbounded. We propose the name `Feige sequence' for the sequence $(C_{n})_{n=1}^{\infty}$. Whilst \cite{DKK2018} verifies that the Feige sequence is unbounded, there are no non-trivial bounds on its rate of growth. Moreover, the Lipschitz modulus of continuity $\omega(t)=t$ remains the weakest modulus of continuity for which it is known that there are bounded and bounded away from zero densities $\rho$ excluding $\omega$-continuous solutions $f$ of \eqref{eq:pfeq}. Thus, providing such densities $\rho$ for any strictly weaker modulus of continuity $\omega$ would be an interesting result and it seems plausible that this could also have implications for the problem of determining the asymptotics of the Feige sequence. We advertise this as a direction of possible future research.

\bibliographystyle{plain}
\bibliography{citations}

\begin{thebibliography}{10}

\bibitem{Aliste}
J.~Aliste-Prieto, D.~Coronel, and J.-M. Gambaudo.
\newblock Linearly repetitive {D}elone sets are rectifiable.
\newblock {\em Annales de l'Institut Henri Poincaré C, Analyse non linéaire},
  30(2):275--290, 2013.
\newblock \url{https://doi.org/10.1016/j.anihpc.2012.07.006}.

\bibitem{AlgorithmicRY}
A.~Aviny{\'o}, J.~Sol{\`a}-Morales, and M.~Val{\`e}ncia.
\newblock On maps with given {J}acobians involving the heat equation.
\newblock {\em Zeitschrift f{\"u}r angewandte Mathematik und Physik},
  54(6):919--936, 2003.
\newblock \url{https://doi.org/10.1007/s00033-003-0070-y}.

\bibitem{aperiodic1}
M.~Baake and U.~Grimm.
\newblock {\em Aperiodic order}, volume~1 of {\em Encyclopedia of Mathematics
  and its Applications}.
\newblock Cambridge University Press, 2013.
\newblock \url{https://doi.org/10.1017/CBO9781139025256}.

\bibitem{aperiodic2}
M.~Baake and U.~Grimm.
\newblock {\em Aperiodic Order}, volume~2 of {\em Encyclopedia of Mathematics
  and its Applications}.
\newblock Cambridge University Press, 2017.
\newblock \url{https://doi.org/10.1017/9781139033862}.

\bibitem{quasicrystals}
M.~Baake and R.~V. Moody.
\newblock {\em Directions in Mathematical Quasicrystals}.
\newblock CRM monograph series. American Mathematical Society, 2000.
\newblock \url{https://doi.org/10.1090/crmm/013}.

\bibitem{benyamini1998geometric}
Y.~Benyamini and J.~Lindenstrauss.
\newblock {\em Geometric Nonlinear Functional Analysis}.
\newblock Number Vol. 1 in American Mathematical Society Colloquium
  Publications. American Mathematical Soc., 1998.
\newblock \url{https://doi.org/10.1090/coll/048}.

\bibitem{BK1}
D.~Burago and B.~Kleiner.
\newblock Separated nets in {Euclidean} space and {Jacobians} of {biLipschitz}
  maps.
\newblock {\em Geometric and Functional Analysis}, 8:273--282, 1998.
\newblock \url{http://dx.doi.org/10.1007/s000390050056}.

\bibitem{BK2}
D.~Burago and B.~Kleiner.
\newblock Rectifying separated nets.
\newblock {\em Geometric and Functional Analysis}, 12:80--92, 2002.
\newblock \url{http://dx.doi.org/10.1007/s00039-002-8238-8}.

\bibitem{metric_geom}
Y.~Cornulier and P.~de~la Harpe.
\newblock {\em Metric Geometry of Locally Compact Groups}.
\newblock European Mathematical Society Publishing House, 2016.
\newblock \url{https://doi.org/10.4171/166}.

\bibitem{Navas}
M.~I. Cortez and A.~Navas.
\newblock Some examples of repetitive, nonrectifiable {D}elone sets.
\newblock {\em Geometry \& Topology}, 20(4):1909--1939, 2016.
\newblock \url{https://doi.org/10.2140/gt.2016.20.1909}.

\bibitem{CDK09}
G.~Cupini, B.~Dacorogna, and O.~Kneuss.
\newblock On the equation ${{\mathrm{det}}\,\nabla{u}=f}$ with no sign
  hypothesis.
\newblock {\em Calculus of Variations and Partial Differential Equations},
  36:251--283, 2009.
\newblock \url{https://doi.org/10.1007/s00526-009-0228-3}.

\bibitem{DMo}
B.~Dacorogna and J.~Moser.
\newblock On a partial differential equation involving the {Jacobian}
  determinant.
\newblock {\em Annales de l'institut Henri Poincar\'e (C) Analyse non
  lin\'eaire}, 7(1):1--26, 1990.
\newblock \url{http://eudml.org/doc/78211}.

\bibitem{Deim}
K.~Deimling.
\newblock {\em {Nonlinear functional analysis}}.
\newblock Springer-Verlag Berlin Heidelberg, 1985.
\newblock \url{https://doi.org/10.1007/978-3-662-00547-7}.

\bibitem{GGT}
C.~Dru{\c{t}}u and M.~Kapovich.
\newblock {\em Geometric group theory}, volume~63 of {\em Colloquium
  Publications}.
\newblock American Mathematical Society, 2018.
\newblock \url{https://doi.org/10.1090/coll/063}.

\bibitem{DK2021divergence}
M.~Dymond and V.~Kalu{\v{z}}a.
\newblock Divergence of separated nets with respect to displacement
  equivalence.
\newblock \url{https://arxiv.org/abs/2102.13046}, 2021.
\newblock arXiv preprint.

\bibitem{DKK2018}
M.~Dymond, V.~Kalu{\v{z}}a, and E.~Kopeck{\'a}.
\newblock {Mapping $n$ grid points onto a square forces an arbitrarily large
  Lipschitz constant}.
\newblock {\em Geometric and Functional Analysis}, 28(3):589--644, 2018.
\newblock \url{https://doi.org/10.1007/s00039-018-0445-z}.

\bibitem{Garber}
A.~I. Garber.
\newblock On equivalence classes of separated nets.
\newblock {\em Modelirovanie i Analiz Informatsionnykh Sistem}, 16(2):109--118,
  2009.
\newblock \url{http://mi.mathnet.ru/eng/mais57}.

\bibitem{Grom}
M.~L. Gromov.
\newblock Asymptotic invariants of infinite groups.
\newblock In Graham~A. Niblo and Martin~A. Roller, editors, {\em Geometric
  Group Theory}, volume~2 of {\em London Mathematical Society lecture note
  series}. Cambridge University Press, 1993.
\newblock \url{https://doi.org/10.1017/CBO9780511629273}.

\bibitem{Katok_handbook}
B.~Hasselblatt and A.~Katok.
\newblock Chapter 1 principal structures.
\newblock In B.~Hasselblatt and A.~Katok, editors, {\em Handbook of Dynamical
  Systems}, volume~1, pages 1--203. Elsevier Science, 2002.
\newblock \url{https://doi.org/10.1016%2Fs1874-575x%2802%2980003-0}.

\bibitem{Hatcher}
A.~Hatcher.
\newblock {\em Algebraic Topology}.
\newblock Cambridge University Press, 2002.

\bibitem{HKW2014}
A.~Haynes, M.~Kelly, and B.~Weiss.
\newblock Equivalence relations on separated nets arising from linear toral
  flows.
\newblock {\em Proceedings of the London Mathematical Society},
  109(5):1203--1228, 2014.
\newblock \url{http://dx.doi.org/10.1112/plms/pdu036}.

\bibitem{laczkovich92}
M.~Laczkovich.
\newblock {Uniformly Spread Discrete Sets in $\mathbb{R}^{d}$}.
\newblock {\em Journal of the London Mathematical Society}, s2-46(1):39--57,
  1992.
\newblock \url{http://dx.doi.org/10.1112/jlms/s2-46.1.39}.

\bibitem{Magazinov}
A.~N. Magazinov.
\newblock The family of bi-{L}ipschitz classes of {D}elone sets in {E}uclidean
  space has the cardinality of the continuum.
\newblock {\em Proceedings of the Steklov Institute of Mathematics},
  275(1):78--89, 2011.
\newblock \url{https://doi.org/10.1134/S0081543811080050}.

\bibitem{Matousek_open}
J.~Matou\v{s}ek and A.~Naor~(eds.).
\newblock {Open problems on low-distortion embeddings of finite metric spaces},
  2011 (last revision).
\newblock {Available} at \url{kam.mff.cuni.cz/~matousek/metrop.ps}.

\bibitem{Mattila}
P.~Mattila.
\newblock {\em Geometry of {S}ets and {M}easures in {E}uclidean {S}paces:
  {F}ractals and rectifiability}.
\newblock Cambridge Studies in Advanced Mathematics. Cambridge University
  Press, 1999.
\newblock \url{https://doi.org/10.1017/CBO9780511623813}.

\bibitem{McM}
C.~T. McMullen.
\newblock Lipschitz maps and nets in {Euclidean} space.
\newblock {\em Geometric and Functional Analysis}, 8:304--314, 1998.
\newblock \url{http://dx.doi.org/10.1007/s000390050058}.

\bibitem{Moser}
J.~Moser.
\newblock {O}n the {V}olume {E}lements on a {M}anifold.
\newblock {\em Transactions of the American Mathematical Society},
  120(2):286--294, 1965.
\newblock \url{http://doi.org/10.2307/1994022}.

\bibitem{Reimann}
H.~M. Reimann.
\newblock {H}armonische {F}unktionen und {J}acobi-{D}eterminanten von
  {D}iffeomorphismen.
\newblock {\em Commentarii Mathematici Helvetici}, 47(1):397--408, 1972.
\newblock \url{https://doi.org/10.1007%2Fbf02566813}.

\bibitem{RY}
T.~Rivi\`ere and D.~Ye.
\newblock Resolutions of the prescribed volume form equation.
\newblock {\em Nonlinear Differential Equations and Applications},
  3(3):323--369, 1996.
\newblock \url{http://dx.doi.org/10.1007/BF01194070}.

\bibitem{Solomon2011}
Y.~Solomon.
\newblock Substitution tilings and separated nets with similarities to the
  integer lattice.
\newblock {\em Israel Journal of Mathematics}, 181(1):445--460, 2011.
\newblock \url{https://doi.org/10.1007/s11856-011-0018-4}.

\bibitem{Solomon2014}
Y.~Solomon.
\newblock A simple condition for bounded displacement.
\newblock {\em Journal of Mathematical Analysis and Applications}, 414(1):134
  -- 148, 2014.
\newblock \url{https://doi.org/10.1016/j.jmaa.2013.12.050}.

\bibitem{Viera_final}
R.~Viera.
\newblock Densities non-realizable as the {Jacobian} of a 2-dimensional
  {bi-Lipschitz} map are generic.
\newblock {\em Journal of Topology and Analysis}, 10(04):933--940, 2018.
\newblock \url{https://doi.org/10.1142/S1793525318500322}.

\bibitem{Ye}
D.~Ye.
\newblock Prescribing the {Jacobian} determinant in {Sobolev} spaces.
\newblock {\em Annales de l'institut Henri Poincar\'e (C) Analyse non
  lin\'eaire}, 11(3):275--296, 1994.
\newblock \url{http://eudml.org/doc/78332}.

\bibitem{zajicek2005}
L.~Zajíček.
\newblock On $\sigma$-porous sets in abstract spaces.
\newblock {\em Abstr. Appl. Anal.}, (5):509--534, 2005.
\newblock \url{http://dx.doi.org/10.1155/AAA.2005.509}.

\end{thebibliography}

\noindent Institut für Mathematik\\
Universität Innsbruck\\
Technikerstraße 13,\\
6020 Innsbruck,\\
Austria\\[3mm]
\noindent Michael Dymond\\Mathematisches Institut\\
Universität Leipzig\\
PF 10 09 02\\
04109 Leipzig\\
Deutschland\\
\texttt{michael.dymond@math.uni-leipzig.de}\\[3mm]

\noindent Vojtěch Kaluža\\
IST Austria,\\
Am Campus 1,\\
3400 Klosterneuburg,\\
Austria\\
\texttt{vojtech.kaluza@ist.ac.at}
\end{document}